\documentclass{article}

\usepackage{fullpage}
\usepackage{graphicx}
\usepackage{prettyref}
\usepackage{subcaption}
\usepackage{amsmath,amssymb,amsthm}
\usepackage{xcolor}
\usepackage{xfrac}
\usepackage{nicefrac}
\usepackage{enumitem}
\usepackage{bookmark}

\usepackage{todonotes}

\usepackage[normalem]{ulem} 

\newtheorem{remark}{Remark}
\newtheorem{problem}{Problem}
\newtheorem{definition}{Definition}
\newtheorem{assumption}{Assumption}
\newtheorem{theorem}{Theorem}
\newtheorem{lemma}{Lemma}
\newtheorem{proposition}{Proposition}[section]
\newtheorem{corollary}{Corollary}[section]

\usepackage[ruled,vlined,linesnumbered,algosection,algo2e]{algorithm2e}

\newrefformat{def}{Definition~\ref{#1}}
\newrefformat{prob}{Problem~\ref{#1}}
\newrefformat{thm}{Theorem~\ref{#1}}
\newrefformat{prop}{Proposition~\ref{#1}}
\newrefformat{lem}{Lemma~\ref{#1}}
\newrefformat{cor}{Corollary~\ref{#1}}
\newrefformat{sec}{Section~\ref{#1}}
\newrefformat{sub}{Subsection~\ref{#1}}
\newrefformat{rem}{Remark~\ref{#1}}
\newrefformat{alg}{Algorithm~\ref{#1}}
\newrefformat{ex}{Example~\ref{#1}}
\newrefformat{app}{Appendix~\ref{#1}}
\newrefformat{subfig}{Figure~\ref{#1}}
\newrefformat{fig}{Figure~\ref{#1}}

\usepackage{authblk}

\begin{document}

\title{Separation-free exponential fitting with structured noise, with applications to inverse problems in parabolic PDEs \thanks{R. Katz was supported by the Alon Fellowship from the Council of Higher
Education in Israel. G. Giordano acknowledges support from the European Union
through the ERC INSPIRE grant (project number 101076926); views and opinions
expressed are however those of the authors only and do not necessarily reflect
those of the European Union or the European Research Council; neither the
European Union nor the granting authority can be held responsible for them. D.
Batenkov was partially supported by Israel Science Foundation Grant 1793/20.}}

\author[1]{Rami Katz} 
\author[2,3]{Dmitry Batenkov}
\author[4]{Giulia Giordano} 

\affil[1]{School of Electrical and Computer Engineering, Tel Aviv University, Tel Aviv, Israel (e-mail: ramkatsee@tauex.tau.ac.il).}
\affil[2]{Basis Research Institute, New York City, NY, USA (e-mail: dima@basis.ai).}
\affil[3]{Department of Applied Mathematics, School of Mathematical Sciences, Tel Aviv University, Tel Aviv, Israel.}
\affil[4]{{Department of Industrial Engineering, University of Trento, Trento, Italy (e-mail: giulia.giordano@unitn.it).}}

\date{}

\maketitle

\begin{abstract}

    We investigate the recovery of exponents and amplitudes of an exponential sum, where the exponents $\left\{\lambda_n \right\}_{n=1}^{N_1}$ are the first $N_1$ eigenvalues of a Sturm-Liouville operator, from finitely many measurements subject to measurement noise.
    This inverse problem is extremely ill-conditioned when the noise is  arbitrary and unstructured. Surprisingly, however, the extreme ill-conditioning exhibited by this problem disappears when considering a \emph{structured}
    noise term,
    taken as an exponential sum with exponents given by the subsequent eigenvalues $\left\{\lambda_n \right\}_{n=N_1+1}^{N_1+N_2}$ of the Sturm-Liouville operator, multiplied by a noise magnitude parameter $\varepsilon>0$. In this case, we rigorously show that the  
    exponents and amplitudes can be recovered with super-exponential accuracy: we both prove the theoretical result and show that it can be achieved numerically by a specific algorithm. By leveraging recent results on the mathematical theory of super-resolution, we show in this paper that the classical Prony's method attains the analytic optimal error decay also in the ``separation-free'' regime where $\lambda_n \to \infty$ as $n \to \infty$, thereby extending the applicability of Prony's method to new settings. As an application of our theoretical analysis, we show that the approximated  eigenvalues obtained by our method can be used
    to recover an unknown potential in a linear reaction-diffusion equation from discrete solution traces.
\end{abstract}

\section{Introduction}\label{sec:intro}

Exponential fitting is a ubiquitous problem in applied mathematics \cite{istratov1999, pereyra2010}. One of the simplest formulations of the exponential fitting problem amounts to recovering the unknown parameters $\{y_n,\lambda_n\}_{n=1}^{N_1}$, $N_1\in \mathbb{N}$, of the function
\begin{equation}\label{eq:sum-of-exps}
y(t) = \sum_{n=1}^{N_1} y_n e^{-\lambda_n t}
\end{equation}
from $2N_1$ equispaced and noise-free samples taken at times $t_k=k\Delta$, $k=0,1,\dots,2N_1-1$:
\begin{equation*}
y(t_k) = \sum_{n=1}^{N_1} y_n e^{-\lambda_n \Delta k}=: \sum_{n=1}^{N_1}y_n\phi_n^k, \quad \phi_n:=e^{-\lambda_n \Delta},
\end{equation*}
 where $\Delta>0$ is the inter-sampling step size. The famous  method by Prony \cite{prony1795}, illustrated in Section~\ref{Sec:PronyMethod} below, solves the problem by exploiting a very elegant connection between the sequence $y(t_k)$ and the coefficients of the polynomial $p(z)=\prod_{n=1}^{N_1}\left(z-\phi_n\right)$.

Exponential fitting works best when the number of terms $N_1$ is small and fixed, while the minimal separation between the exponentials $\phi_n$ can be controlled and is, in particular, lower bounded by a small and fixed constant $\delta>0$, namely, $\min_{1\leq n,m\leq N_1,m\neq n}\left|\phi_n - \phi_m \right|\geq \delta$. The numerical stability, or noise robustness, deteriorates quickly when $N_1\to \infty$ or $\delta\to 0^+$. To see this, consider the case of known exponents $\{ \lambda_n \}_{n=1}^{N_1}$, and therefore known exponentials $\left\{\phi_n \right\}_{n=1}^{N_1}$. Recovery of the amplitudes $\{y_n \}_{n=1}^{N_1}$ then requires the inversion of the Vandermonde matrix $V:=\left[\phi_n^k\right]_{k=0,\dots,N_1-1}^{n=1,\dots,N_1}$. Given perturbed or noisy measurements $\widetilde{y(t_k)}$, $n=1,\dots, N_1$, with measurement error or noise $e_k=\widetilde{y(t_k)}-y(t_k)$, $n=1,\dots,N_1$, we have the following bound for the error in the recovered amplitudes $\left\{\hat{y}_n\right\}_{n=1}^{N_1}$:
\begin{equation}\label{eq:example-amplitude-error}
\bigl\|\operatorname{col}\{\hat{y}_n-y_n\}_{n=1}^{N_1}\bigr\|_{\infty} \leq \bigl\|V^{-1}\bigr\|_{\infty} \cdot \bigl\|\operatorname{col}\{e_k\}_{k=0}^{N_1-1}\bigr\|_{\infty},
\end{equation}
where $\operatorname{col}\{v_j\}_{j=1}^{N_1}$ denotes the $N_1\times 1$ column vector with entries $v_1,\dots,v_{N_1}$. A well-known result by Gautschi \cite{gautschi1962} (see also \cite{beckermann2000, beckermann2019}) states that
$$
\left\|V^{-1}\right\|_{\infty}\leq \max_{1\leq i\leq N_1} \prod_{j\neq i} \frac{1+\left|\phi_j\right|}{\left|\phi_i -\phi_j\right|},
$$
and this upper bound is \emph{attainable} when all the $\phi_i$'s have the same (complex) sign. In particular, if $N_1\to \infty$ and $\left\{\lambda_n\right\}_{n=1}^{\infty}$ is a real sequence satisfying $\lim_{n\to \infty}\lambda_n=\infty$, as will be the case in the problem we consider in this paper, the (exact) upper bound on
$\left\|V^{-1} \right\|_{\infty}$ grows rapidly as $N_1\to \infty$, thereby implying the ill-posedness of the recovery of $\left\{y_n \right\}_{n=1}^{N_1}$ for arbitrary measurement noise. Robustness of recovering the exponents $\{\lambda_n\}_{n=1}^{N_1}$ when those are unknown exhibits a similar behavior \cite{batenkov2013b}. Note that the well-posedness of the problem has different properties when the real eigenvalues $\left\{\lambda_n \right\}_{n=1}^{\infty}$ satisfying $\lim_{n\to \infty}\lambda_n=\infty$ are replaced by purely imaginary exponents. Such an exponential fitting problem has been studied much more in the inverse problems and applied harmonic analysis communities, especially in the context of mathematical super-resolution and non-harmonic Fourier transform \cite{batenkov2023, batenkov2021b, batenkov2020, batenkov2021, diab2024,katz2024accuracy,decimatedProny}; see also \cite{aubel2019, barnett2022, demanet2015, kunis2020, li2021a,  moitra2015, pan2016} and the references therein.

Motivated by inverse problems in parabolic PDEs and applications to control (cf. Sections~\ref{sec:problem-def-and-motivation} and~\ref{sec:inverse-problems-pde} below), in this work we consider the model \eqref{eq:sum-of-exps} with varying $N_1$ and $\left\{\lambda_n \right\}_{n=1}^{N_1}$ being eigenvalues of a Sturm-Liouville operator. We show that, in this case, a direct parametric inversion approach in the model \eqref{eq:sum-of-exps} can still be viable even when the right-hand side of \eqref{eq:example-amplitude-error} is unbounded, provided the measurement noise takes the structured form of an exponential sum, with exponents given by the subsequent eigenvalues $\left\{\lambda_n \right\}_{n=N_1+1}^{N_1+N_2}$ of the same Sturm-Liouville operator, multiplied by a noise magnitude parameter $\varepsilon>0$ (see \eqref{eq:MeasExpsum} below); in this case, the norm of the error vector $\operatorname{col}\{e_k\}_{k=0}^{N_1-1}$ is on the order of $\varepsilon$, therefore the right-hand side of \eqref{eq:example-amplitude-error} is still unbounded. In addition to showing well-posedness of the parametric inversion, we demonstrate that Prony's method \cite{prony1795}, illustrated in Section~\ref{Sec:PronyMethod} below, achieves first order (in $\varepsilon$) optimal recovery, at least for small noise intensity $\varepsilon$. In the remainder of the introduction, we present the notations used throughout the manuscript, provide a formal problem definition and motivation for its consideration, and summarize the main contributions of this work and its organization.

\subsection{Notation and conventions}

Given $N\in \mathbb{N}$, we denote $[N]=\left\{1,\dots, N \right\}$. For $a\in \mathbb{R}$, we denote $\lfloor a\rfloor = \max \left\{m\in \mathbb{Z} \ ; \ m\leq a \right\}$.
We denote by $I_N$ the identity matrix of size $N$.
For matrices $\left\{A_i\right\}_{i=1}^k$ of appropriate dimension, $\operatorname{diag}\left\{A_i\right\}_{i=1}^k$ is the block-diagonal matrix with the $A_i$'s on the diagonal, whereas $\operatorname{row}\left\{A_i\right\}_{i=1}^k$ (respectively, $\operatorname{col}\left\{A_i\right\}_{i=1}^k$) is the block matrix stacking the $A_i$'s in consecutive columns (respectively, rows). We denote by $L^2(0,1)$ the Hilbert space of square-integrable scalar-valued functions on $(0,1)$ with inner product $\langle \cdot,\cdot \rangle$, while we denote by $\mathrm{H}^2(0,1)$ (respectively, $\mathrm{H}_0^1(0,1)$) the Sobolev space of functions $f$ on $(0,1)$ that are twice (respectively, once) weakly differentiable with $f''\in L^2(0,1)$ (respectively, $f'\in L^2(0,1)$ and $f(0)=f(1)=0$). We denote by $\mathcal{D}'(0,1)$ the space of bounded linear functionals on $\mathrm{H}^2(0,1)\cap \mathrm{H}_0^1(0,1)$.

Furthermore, we introduce the set
\begin{equation}\label{eq:setN}
\mathcal{N}_{\eta} := \left\{ (n,N_1)\in \mathbb{N}^2 \ ; \ 1\leq n \leq \lfloor \eta N_1 \rfloor \right\} \quad \text{for a fixed} \,\, \eta \in (0,1]
\end{equation}
and the function
\begin{equation}\label{eq:CalSdef}
\Psi(n;N_1) := \sum_{j=n+1}^{N_1}(j^2-n^2), \quad \text{for} \,\, n\in [N_1].
\end{equation}

We also adopt the following convention: given any numbers $\left\{b_j\right\}_{j=1}^{\infty}$,  whenever $t<s$, we have $\sum_{j=s}^tb_j = 0$ and $\prod_{j=s}^t b_j=1$. Therefore, we have $\Psi(N_1;N_1)=0$.

\subsection{Problem formulation, standing assumptions and motivation}\label{sec:problem-def-and-motivation}

We consider the following recovery problem.
\begin{problem}[Recovery problem]\label{prob:main-problem}
Given the exponential sum $y(t) = \sum_{n=1}^{N_1}y_ne^{-\lambda_n t}+\varepsilon \sum_{n=N_1+1}^{N_1+N_2}y_ne^{-\lambda_n t}$, where $N_1,N_2\in \mathbb{N} $, $\eta \in (0,1]$ and $0<\varepsilon \ll 1$, recover the exponents $\left\{\lambda_{n}\right\}_{n=1}^{\lfloor \eta N_1 \rfloor}$ and the amplitudes $\left\{y_n\right\}_{n=1}^{\lfloor \eta N_1 \rfloor}$ from its equispaced samples (i.e., discrete-time measurements) taken on a finite set of $2 N_1$ discrete time instants $t_k =k\Delta$, where $k=0,1,\dots,2N_1-1$, with sampling step-size $\Delta>0$:
\begin{equation}\label{eq:MeasExpsum}
    y(t_k) = \sum_{n=1}^{N_1}y_ne^{-\lambda_n t_k}+\varepsilon \sum_{n=N_1+1}^{N_1+N_2}y_ne^{-\lambda_n t_k} := y_{\text{main}}(t_k) + \varepsilon y_{\text{tail}}(t_k), \quad \,\, t_k=k\Delta,\; k=0,1,\dots,2N_1-1.
\end{equation}
Here $\varepsilon$ is a positive \emph{noise intensity} parameter, which  multiplies the \emph{structured noise term} $y_{\text{tail}}$, while $y_{\text{main}}$ is the \emph{noise-free term} that we wish to reconstruct. 
\end{problem}
The additional parameter $\eta \in (0,1]$ is introduced because, as we will show, the quality of recovery of $\lambda_n$ and $y_n$ for $n \in [\lfloor \eta N_1 \rfloor ]$ deteriorates the closer $\eta$ is to the value $1$.

We further make the following standing assumptions.
\begin{assumption}\label{assump:yassump}
There exists $\ell>1$ such that $\ell^{-1}\leq |y_n|\leq \ell$ for all $n\in [N_1]$. Also, there exists $M_y>0$ such that
$\frac{|y_n|}{|y_k|}\leq M_y$ for all $k\in [N_1]$ and $n\in [N_1+N_2]\setminus [N_1]$.
Hence, $\left|y_n\right|\leq \ell M_y=:M_y^{(1)}$ for all $n\in [N_1+N_2]\setminus [N_1]$. 
\end{assumption}
\begin{assumption}\label{assump:eigenvalues}
The exponents $\left\{\lambda_n \right\}_{n=1}^{N_1+N_2}$ in \eqref{eq:MeasExpsum} originate from the monotonically increasing and unbounded real sequence $\left\{\lambda_n \right\}_{n=1}^{\infty}$ of simple eigenvalues of a \emph{fixed} Sturm-Liouville operator $\mathcal{A}$ defined as
\begin{equation}\label{eq:sl-operator}
\left[\mathcal{A}h\right](x) = -\left(p(x)h'(x) \right)'-q(x)h,\quad x\in(0,1); \qquad  \operatorname{Dom}(\mathcal{A}) = \left\{h\in \mathrm{H}^2(0,1) \ ; \ h(0)=h(1)=0 \right\},
\end{equation}
where $p(x)$ and $q(x)$ are \emph{unknown} smooth functions on $[0,1]$ (i.e.,  $p = p_1\vert_{[0,1]}$ and $q = q_1\vert_{[0,1]}$ for some smooth $p_1$ and $q_1$ defined on a neighborhood of $[0,1]$) and $\min_{x\in [0,1]}p(x) = \underline{p}>0$ for some $\underline{p}>0$. The existence of the positive lower bound $\underline{p}$ is a standard assumption, guaranteeing well-posedness and regularity of solutions of \eqref{eq:diffusion-pde} 
\cite{zettl2005,engel2000one}. However, knowledge of $\underline{p}$ is not needed for the considered reconstruction problem.
\end{assumption}

In view of Assumption~\ref{assump:eigenvalues}, the growth of the exponents $\lambda_n$,  $n=1,2,\dots$, is always approximately quadratic in the index $n$, in the precise sense formulated in the following proposition.

    \begin{proposition}\label{prop:EigDiffBound}
    For the eigenvalues $\left\{\lambda_n \right\}_{n=1}^{\infty}$ in Assumption~\ref{assump:eigenvalues}, there exist constants $0<\upsilon \leq \Upsilon$ such that, for any $1\leq n \leq  m$, the following estimate holds:
    \begin{equation}\label{eq:EigDiff}
      \upsilon \left(m^2-n^2 \right) \leq \lambda_m-\lambda_n\leq \Upsilon \left(m^2-n^2 \right).
    \end{equation}
    \end{proposition}
    \begin{proof}
        See Section~\ref{sec:proof:prop:EigDiffBound}.
    \end{proof}
The quadratic growth estimate \eqref{eq:EigDiff} is an essential feature of the considered problem, and will
play a central role throughout our results. For the real eigenvalues $\left\{\lambda_n \right\}_{n=1}^{\infty}$ in Assumption~\ref{assump:eigenvalues}, we denote
\begin{equation}\label{eq:phinDef}
\phi_n : = e^{-\lambda_n \Delta},\ n\in \mathbb{N},
\end{equation}
whence \eqref{eq:MeasExpsum} can be presented as 
\begin{equation}\label{eq:MeasExpsum1}
    y(t_k) = \sum_{n=1}^{N_1}y_n\phi_n^k+\varepsilon \sum_{n=N_1+1}^{N_1+N_2}y_n\phi_n^k, \quad \,\, t_k=k\Delta,\; k=0,1,\dots,2N_1-1.
\end{equation}

Our main motivation stems from the initial-boundary value problem for the linear reaction-diffusion equation
\begin{equation}\label{eq:diffusion-pde}
\partial_t z(x,t) = \left[\mathcal{A}z\right](x,t), \quad (x,t) \in(0,1) \times \mathbb{R}^+,\;\; z(0,t)=z(1,t)=0,\; z(x,0)=f(x),
\end{equation}
where $\mathcal{A}$ is given in Assumption~\ref{assump:eigenvalues}. By the standard separation of variables method, under appropriate assumptions on the initial condition $f(x)$, the solution of \eqref{eq:diffusion-pde} can be represented as a series \cite{evans2010}:
\begin{equation}\label{eq:diffusion-solution-series}
    z(x,t)=\sum_{n=1}^\infty \left\langle f, \psi_n\right\rangle_{L^2(0,1)} \exp(-\lambda_n t)\psi_n(x),
\end{equation}
where $\left\{\psi_n\right\}_{n=1}^{\infty}$ are the eigenfunctions of $\mathcal{A}$ corresponding to the eigenvalues $\left\{\lambda_n\right\}_{n=1}^{\infty}$ and $\left\{\left<f,\psi_n\right>_{L^2(0,1)} \right\}_{n=1}^{\infty}$ are the $L^2$ projections (modes) of the initial condition $f$ on the eigenfunctions. Recovering the eigenvalues $\lambda_n$, $n=1,2,\dots$, of $\mathcal{A}$ and the projections $\left<f,\psi_n\right>_{L^2(0,1)}$, $n=1,2,\dots$, from partial measurements of  $z(x,t)$ is a well-studied inverse problem \cite{isakov2017,pierce1979a,rundell2024}. As we demonstrate in Section~\ref{sec:related-work} below, under appropriate assumptions, discrete-time measurements of $z(x,t)$ at instants $t_k=k\Delta$, $k=0,1,\dots$, can be represented by an expression similar to the one given in \eqref{eq:MeasExpsum}; for
example, by substituting a fixed $x=x_0\in (0,1)$, one obtains the trace
\begin{equation}\label{eq:PDESamplesControl}
y(t_k) = z(x_0,t_k) = \sum_{n=1}^\infty \underbrace{\left[\left\langle f, \psi_n\right\rangle_{L^2(0,1)} \psi_n(x_0)\right]}_{y_n} \phi_n^k,\quad t_k=k\Delta, \ k=0,\dots,2N_1-1.
\end{equation}
Under suitable assumptions, a similar presentation can be obtained for a more general trace $y(t)=\varphi(z(\cdot,t))$, where $\varphi$ is a linear functional applied to $z(\cdot,t)$. The inverse problem of recovering the eigenvalues $\lambda_n$ of $\mathcal{A}$ and the projections $\left<f,\psi_n\right>_{L^2(0,1)}$ from partial measurements of  $z(x,t)$ such as \eqref{eq:PDESamplesControl} also has important applications for the control of PDEs such as \eqref{eq:diffusion-pde}. For example, stabilization of \eqref{eq:diffusion-pde} by the modal decomposition technique \cite{katz2020constructive,katz2020boundary,katz2021delayed,katz2021finite,katz2021regional,katz2022global,katz2021finite,lhachemi2019control,balas1986finite,christofides2002nonlinear,harkort2011finite} employs the presentation of $z(x,t)$ as in \eqref{eq:diffusion-solution-series}. Controller synthesis is then based on the assumption of complete knowledge of the eigenvalues and eigenfunctions of $\mathcal{A}$. Unfortunately, in many control-related applications exact knowledge of $\mathcal{A}$ cannot be assumed, thereby making the implementation of such control strategies difficult in practice. Thus, we expect the recovery approach proposed in this manuscript to have significant impact on data-driven control of systems such as \eqref{eq:diffusion-pde}, under conditions where the operator $\mathcal{A}$ is only partially known, or fully unknown, and only finitely many measurements of $z(x,t)$ are available.

\subsection{Manuscript contributions and organization}
We study \prettyref{prob:main-problem} in three challenging regimes, with a fixed and finite $N_2 \in \mathbb{N}$:
\begin{equation}\label{Eq:Regimes}
\begin{array}{lll}
\underline{\text{Regime 1:}} &  \Delta>0 \text{ fixed},\quad  N_1\gg 1. \\
\underline{\text{Regime 2:}} & N_1 \text{ large and fixed},\quad  \Delta\gg 0. \\
\underline{\text{Regime 3:}} & N_1\Delta \equiv T>0 \text{ fixed},\quad  N_1\gg 1.
\end{array}
\end{equation}

In Section~\ref{Sec:InvProbPres}, we rigorously formulate a recovery criterion and show its well-posedness for sufficiently small $\varepsilon>0$. We obtain explicit expressions for the first-order condition numbers $\mathcal{K}_{\lambda}(n)$ and $\mathcal{K}_{y}(n)$ in $\varepsilon$ associated with the recovery of $\left\{\lambda_n \right\}_{n=1}^{\lfloor \eta N_1\rfloor}$ and $\left\{y_n \right\}_{n=1}^{\lfloor \eta N_1\rfloor}$ respectively, with $\eta \in (0,1]$, which are valid for \prettyref{prob:main-problem} regardless of the considered values of $N_1 \geq 1$ and $\Delta>0$. In Section~\ref{Sec:AsympAnalysis}, we subsequently study the asymptotic behavior of $\mathcal{K}_{\lambda}(n)$ and $\mathcal{K}_{y}(n)$, when the noise intensity $\varepsilon>0$ is small, in the three regimes in \eqref{Eq:Regimes}. In more detail, in Theorem~\ref{Thm:FirstOrdCond} we show that the condition numbers decay exponentially with $\Delta$ (Regime 2) and \emph{super-exponentially} with $N_1$ (Regimes 1 and 3), implying the extreme stability exhibited by the inverse problem in the considered regimes. Although we also consider the case of $\eta=1$, corresponding to the recovery of the full sets $\left\{\lambda_{n}\right\}_{n=1}^{N_1}$ and $\left\{y_n\right\}_{n=1}^{N_1}$, we are able to show decay of the first-order condition numbers only if $\eta < 1$; this is further confirmed by the numerical simulations in Section~\ref{Sec:NumericalExamples}. 

In Section~\ref{Sec:PronyProb}, we consider the well-known Prony's method \cite{prony1795}, described in Section~\ref{Sec:PronyMethod}, and analyze its performance for the recovery of $\left\{\lambda_n\right\}_{n=1}^{\lfloor\eta N_1 \rfloor}$ and $\left\{y_n\right\}_{n=1}^{\lfloor\eta N_1 \rfloor}$ from \eqref{eq:MeasExpsum} when $N_2=1$ (i.e., when the structured noise term $y_{\text{tail}}$ in \eqref{eq:MeasExpsum} contains a single mode). In Section~\ref{sec:pronystab}, we show that, in all the regimes listed in \eqref{Eq:Regimes}, the condition numbers of Prony's method exhibit the same asymptotic behavior as the analytic condition numbers in Theorem~\ref{Thm:FirstOrdCond}. This result, together with the numerical simulations presented in Section~\ref{sec:numerics-condition-numbers}, provides strong evidence for the first-order optimality of Prony's method for recovery of $\left\{\lambda_{n}\right\}_{n=1}^{N_1}$ and $\left\{y_n\right\}_{n=1}^{N_1}$ from the measurements \eqref{eq:MeasExpsum}.  In Sections~\ref{sec:numerics-integral-measurements} and~\ref{sec:potential-recovery-point-measurements} we also demonstrate how our solution to recovery of $\left\{\lambda_n \right\}_{n=1}^{\lfloor \eta N_1\rfloor}$ and $\left\{y_n \right\}_{n=1}^{\lfloor \eta N_1\rfloor}$ from  \eqref{eq:MeasExpsum} can be used to reconstruct the spatial differential operator $\mathcal{A}$ in \eqref{eq:diffusion-pde}. 

In Section~\ref{sec:related-work} we discuss some related works and put our results into a broader context.

\section{Exponential fitting as an inverse problem}\label{Sec:WellPosed}

Recall that we consider Problem~\ref{prob:main-problem} of recovering the exponents $\left\{\lambda_{n}\right\}_{n=1}^{\lfloor \eta N_1\rfloor}$ and amplitudes $\left\{y_n\right\}_{n=1}^{\lfloor \eta N_1\rfloor}$, with $\eta\in (0,1]$, from discrete measurements \eqref{eq:MeasExpsum} of an exponential sum, under Assumptions~\ref{assump:yassump} and~\ref{assump:eigenvalues}. In this section, we first rigorously formulate a recovery criterion and show its well-posedness for sufficiently small noise intensity $\varepsilon>0$.
Then, when the noise intensity $\varepsilon>0$ is small, we study the asymptotic behavior of the first-order condition numbers $\mathcal{K}_{\lambda}(n)$ and $\mathcal{K}_{y}(n)$ in the regimes defined in \eqref{Eq:Regimes}.

\subsection{Well-posedness of the inverse problem and first-order condition numbers}\label{Sec:InvProbPres}

Given a finite set of discrete-time measurements $y(t_k)$ as in \eqref{eq:MeasExpsum}, where $t_k = k\Delta$, with step-size $\Delta>0$ and $k=0,1,\dots, 2N_1-1$, define the nonlinear function $\mathcal{F} \colon \mathbb{R}^{2N_1+1}\to \mathbb{R}^{2N_1}$ as
\begin{equation}\label{eq:ImplicitFunc}
\begin{array}{lll}
&\mathcal{F}\left(\left\{\hat{y}_n\right\}_{n=1}^{N_1},\left\{\hat{\lambda}_n \right\}_{n=1}^{N_1},\varepsilon  \right) = \operatorname{col}\left\{\sum_{n=1}^{N_1}\hat{y}_ne^{-\hat{\lambda}_n k \Delta}- y(t_k)\right\}_{k=0}^{2N_1-1}\\
&\hspace{30mm}= \operatorname{col}\left\{\sum_{n=1}^{N_1}\hat{y}_ne^{-\hat{\lambda}_n k \Delta}- \sum_{n=1}^{N_1}y_ne^{-\lambda_nt}-\varepsilon \sum_{n=N_1+1}^{N_1+N_2}y_ne^{-\lambda_nt}\right\}_{k=0}^{2N_1-1}.
\end{array}
\end{equation}
Function $\mathcal{F}$ depends on the noise $\varepsilon$ through the measurements $\left\{ y(t_k)\right\}_{k=0}^{2N_1-1}$ and, intuitively, quantifies the discrepancy between such measurements and their noise-free approximations of the form $\left\{\sum_{n=1}^{N_1}\hat{y}_ne^{-\hat{\lambda}_n k \Delta} \right\}_{k=0}^{2N_1-1} $
obtained from a candidate set of reconstruction parameters $\hat{P}:=\left(\left\{\hat{y}_n\right\}_{n=1}^{N_1},\left\{\hat{\lambda}_n \right\}_{n=1}^{N_1}\right)$, which we refer to as \emph{approximation candidate}.  We further denote $P:=\left(\left\{y_n\right\}_{n=1}^{N_1},\left\{\lambda_n \right\}_{n=1}^{N_1}\right)$.
\begin{definition}\label{def:EpsApprox}
Given $\varepsilon\in \mathbb{R}$, the approximation candidate $\hat{P}$ is an \emph{$\varepsilon$-approximation} of $P$ if $\mathcal{F}(\hat{P},\varepsilon) =0$.
\end{definition}
Definition~\ref{def:EpsApprox} is motivated by the fact that $\mathcal{F}\left(P,0  \right) =0$, meaning that in the absence of noise (i.e., when $\varepsilon=0$), the choice $\hat{P}=P$ is a zero of $\mathcal{F}$. In the absence of any additional datum beyond $\left\{y(t_k) \right\}_{k=0}^{2N_1-1}$, we propose that also for the case $\varepsilon> 0$ (i.e., in the presence of structured noise), a suitable approximation $\hat{P}$ is sought to guarantee $\mathcal{F}(\hat{P},\varepsilon)=0$. Moreover, for $\varepsilon> 0$, we expect the $\varepsilon$-approximation $\hat{P}$ to satisfy 
\begin{equation}\label{eq:EpsExpans}
\hat{\lambda}_n = \hat{\lambda}_n(\varepsilon) = \lambda_n +  \mathcal{K}_{\lambda}(n) \varepsilon+ o(\varepsilon) \,\, \text{and} \,\,  \hat{y}_n = \hat{y}_n(\varepsilon) = y_n +  \mathcal{K}_{y}(n) \varepsilon+ o(\varepsilon),\quad n\in [N_1],
\end{equation}
where the condition numbers $\mathcal{K}_{\lambda}(n)$ and $\mathcal{K}_{y}(n)$ describe how the measurement error due to the presence of structured noise with intensity $\varepsilon\neq 0$ is amplified in the recovery of an $\varepsilon$-approximation $\hat{P}$, and are formally defined as follows.
\begin{definition}\label{def:FstOrdCond}
Consider \eqref{eq:ImplicitFunc} and \eqref{eq:EpsExpans}. The first-order condition numbers (in $\varepsilon$) for the inverse problem of obtaining an $\varepsilon$-approximation $\hat{P}$ are defined by
\begin{equation*}
\begin{array}{lll}
&\mathcal{K}_{\lambda}(n) = \left. \left(\frac{\mathrm{d}}{\mathrm{d}\varepsilon}\hat{\lambda}_n(\varepsilon)\right)\right|_{\varepsilon =0}  \,\,\quad \mbox{and} \,\,\quad \mathcal{K}_{y}(n) = \left. \left(\frac{\mathrm{d}}{\mathrm{d}\varepsilon}\hat{y}_n(\varepsilon)\right)\right|_{\varepsilon =0},\,\,\quad n\in [N_1].
\end{array}
\end{equation*}
\end{definition}

To formalize the previous discussion, we first recall the definitions of Lagrange and Hermite interpolation polynomials: the Lagrange interpolation basis on the distinct nodes $\left\{\chi_n \right\}_{n=1}^{N_1}\subset \mathbb{C}$ is defined as
\begin{equation}\label{eq:Lagrange}
    \left\{ L_{n}(z) \right\}_{{n}=1}^{N_1} = \left\{\prod_{j\neq n}\frac{z-\chi_{j}}{\chi_n-\chi_{j}} \right\}_{{n}=1}^{N_1},
\end{equation}
and the Hermite interpolation basis is then given by
\begin{equation}\label{eq:Hermite}
\left\{H_{n}(z), \,\, \widetilde{H}_{n}(z)\right\}_{{n}=1}^{N_1} = \left\{ \left[ 1-2(z-\chi_{n})L_{n}'(\chi_{n})\right] L_{n}^2(z), \,\,\, (z-\chi_{n})L_{n}^2(z) \right\}_{{n}=1}^{N_1}.
\end{equation}

We are now ready to show that, when the noise intensity $\varepsilon>0$ is small, the stated inverse problem is well-posed, namely:
\begin{itemize}[noitemsep]
\item The estimates $\{\hat{\lambda}_n\}_{n=1}^{N_1}$ of $\{{\lambda}_n\}_{n=1}^{N_1}$ and $\{\hat{y}_n\}_{n=1}^{N_1}$ of $\{{y}_n\}_{n=1}^{N_1}$ exist and are unique, at least for small noise intensity;
\item The associated first-order condition numbers are well-defined and finite (i.e., the solution to $\mathcal{F}(\hat{P},\varepsilon)=0$  depends continuously on the data given by the measurements $\left\{ y(t_k)\right\}_{k=0}^{2N_1-1}$).
\end{itemize}

\begin{theorem}\label{thm:FirstOrdCondRep}
There exist $\varepsilon_*>0$ and functions 
$\left\{\hat{y}_n(\varepsilon)\right\}_{n=1}^{N_1}$ and $\left\{\hat{\lambda}_n(\varepsilon) \right\}_{n=1}^{N_1}$, defined and continuously differentiable on $(-\varepsilon_*,\varepsilon_*)$, 
such that the approximation candidate $\hat{P}(\varepsilon):=\left(\left\{\hat{y}_n(\varepsilon)\right\}_{n=1}^{N_1},\left\{\hat{\lambda}_n(\varepsilon) \right\}_{n=1}^{N_1}\right)$ satisfies $\hat{P}(0) = P$ and
\begin{equation}\label{eq:IFT}
\forall \varepsilon\in  (-\varepsilon_*,\varepsilon_*) \qquad \mathcal{F}(\hat{P},\varepsilon)=0 \iff \hat{P}=\hat{P}(\varepsilon).   
\end{equation}
Furthermore, the first-order condition numbers $\mathcal{K}_{y}(n)$ and $\mathcal{K}_{\lambda}(n)$ have the explicit expression
\begin{equation}\label{eq:Deriv}
\operatorname{col}\left\{
\begin{bmatrix}
\mathcal{K}_{y}(n) \\ \mathcal{K}_{\lambda}(n)  
\end{bmatrix}
\right\}_{n=1}^{N_1} = \operatorname{col}\left\{
\begin{bmatrix}
\frac{\mathrm{d}\hat{y}_{n} }{\mathrm{d} \varepsilon}(0) \\ \frac{\mathrm{d}\hat{\lambda}_{n} }{\mathrm{d} \varepsilon}(0)    
\end{bmatrix}
\right\}_{n=1}^{N_1} =\sum_{m=N_1+1}^{N_1+N_2}y_m \operatorname{col}\left\{ 
\begin{bmatrix}
 H_{n}(\phi_m)\\
-\frac{1}{\Delta y_{n} \phi_n } \tilde{H}_{n}(\phi_m)
\end{bmatrix}
\right\}_{n=1}^{N_1},
\end{equation}   
where $\left\{H_n,\tilde{H}_n \right\}_{n=1}^{N_1}$ are the Hermite interpolation polynomials given in \eqref{eq:Hermite}.
\end{theorem}
\begin{proof}
See Appendix~\ref{Append:Pf5}.
\end{proof}

\subsection{Asymptotic analysis of the first-order condition numbers}\label{Sec:AsympAnalysis}
We now analyze the asymptotic behavior of the first-order condition numbers $\mathcal{K}_{\lambda}(n)$ and $\mathcal{K}_y(n)$, given in Definition~\ref{def:FstOrdCond}, in the regimes listed in \eqref{Eq:Regimes}. According to Theorem~\ref{thm:FirstOrdCondRep}, this can be done by studying the asymptotic behavior of $H_n(\phi_m)$ and $\frac{1}{\Delta y_n \phi_n}\tilde{H}_n(\phi_m)$ for fixed $n\in [N_1]$  and $m\in [N_1+N_2]\setminus[N_1]$. 



\begin{theorem}\label{Thm:FirstOrdCond}
Let $\eta \in (0,1)$ and $n,N_1\in \mathbb{N}$ with $ 
(n,N_1)\in \mathcal{N}_{\eta}$ and consider the analytic first order condition numbers $\mathcal{K}_{\lambda}(n)$ and $\mathcal{K}_{y}(n)$ given in Definition~\ref{def:FstOrdCond}, whose explicit expression is in \eqref{eq:Deriv}.
Then, in the regimes listed in \eqref{Eq:Regimes}, we can state the following. 
\begin{itemize}
    \item \underline{Regime 1:} There exist $\gamma_y(\eta,\Delta),  \ \gamma_{\lambda}(\eta,\Delta), \ \Gamma(\eta) >0$ such that, when $N_1\gg 1$,
    \begin{equation*}
    \begin{array}{lll}
    \max_{n \in [\lfloor \eta N_1\rfloor]}\left|\mathcal{K}_{y}(n)\right|\leq \gamma_y(\eta, \Delta)  e^{ -\Gamma(\eta)\Delta N_1^3},\\  \max_{ n \in [\lfloor \eta N_1\rfloor]}\left|\mathcal{K}_{\lambda}(n)\right| \leq \gamma_{\lambda}(\eta, \Delta)e^{ -\Gamma(\eta)\Delta N_1^3}.
    \end{array}
    \end{equation*}
    \item \underline{Regime 2:} Given any $\vartheta>0$, there exists $N_1(\vartheta)\in \mathbb{N}$  and $\gamma_y(\eta,\vartheta), \ \gamma_{\lambda}(\eta,\vartheta)>0$ such that, for all $N_1\geq N_1(\vartheta)$, when $\Delta \gg 0$,
    \begin{equation*}
    \begin{array}{lll}
    \max_{ n \in [\lfloor \eta N_1 \rfloor]}\left|\mathcal{K}_{y}(n)\right|\leq \gamma_y(\eta, \vartheta)  e^{ -\vartheta \Delta },\\ \max_{ n \in [\lfloor \eta N_1 \rfloor]}\left|\mathcal{K}_{\lambda}(n)\right| \leq \frac{\gamma_{\lambda}(\eta, \vartheta)}{\Delta}e^{ -\vartheta \Delta}.
    \end{array}
    \end{equation*}
    \item \underline{Regime 3:} There exist $\gamma_y(\eta,T), \ \gamma_{\lambda}(\eta,T), \ \Gamma(\eta,T)>0$ such that, when $N_1 \gg 1$, 
    \begin{equation*}
    \begin{array}{lll}
    \max_{n \in [\lfloor \eta N_1\rfloor]}\left|\mathcal{K}_{y}(n)\right|\leq \gamma_y(\eta, T)e^{-\Gamma(\eta,T)N_1^2},\\  \max_{ n \in [\lfloor \eta N_1\rfloor]}\left|\mathcal{K}_{\lambda}(n)\right| \leq \gamma_{\lambda}\left(\eta,T \right)e^{-\Gamma(\eta,T)N_1^2}.
    \end{array}
    \end{equation*}
\end{itemize}
\end{theorem}
\begin{proof}
See Appendix~\ref{Sec:ProofsAsymptoticAnalysis}.
\end{proof}
Theorem~\ref{Thm:FirstOrdCond} provides upper bounds on the asymptotic behavior of the condition numbers $\mathcal{K}_{y}(n)$ and $\mathcal{K}_{\lambda}
(n)$ in the regimes listed in \eqref{Eq:Regimes}, uniformly in $n\in [\lfloor \eta N_1 \rfloor]$ for a fixed $\eta\in (0,1)$.
We consider $\eta\in(0,1)$, instead of $\eta\in(0,1]$, because, when $\eta =1$, the proof of Theorem~\ref{Thm:FirstOrdCond} only yields $\mathcal{K}_{\lambda}(N_1) = O(1)$ and $\mathcal{K}_y(N_1) = O(1)$, uniformly in $N_1$ and in $\Delta$, as is also demonstrated by the numerical simulations in Section~\ref{Sec:NumericalExamples}.

In the next section, we provide theoretical guarantees showing that the condition numbers of Prony's method for exponential fitting exhibit the same asymptotic behavior as the analytic condition numbers in \eqref{eq:Deriv}, thereby implying that Prony's method is optimal for the considered recovery problem.  

\section{Prony's method and its stability}\label{Sec:PronyProb}
We employ Prony's method for exponential fitting to solve the problem of recovering  the eigenvalues $\left\{\lambda_n \right\}_{n=1}^{N_1}$ and the amplitudes $\left\{y_n \right\}_{n=1}^{N_1}$ from the set of discrete measurements of an exponential sum
\begin{equation}\label{eq:MeasurementsProny}
\begin{array}{lll}
y(t_k) = \sum_{n=1}^{N_1}y_ne^{-\lambda_n k \Delta}+\varepsilon y_{N_1+1}e^{-\lambda_{N_1+1} k \Delta} \overset{\eqref{eq:phinDef}}{=}\sum_{n=1}^{N_1}y_n\phi_n^k+\varepsilon y_{N_1+1}\phi_{N_1+1}^k,
\end{array}
\end{equation}
where $t_k=k \Delta$, $k=0,1,\dots ,2N_1-1$, which amounts to considering \eqref{eq:MeasExpsum} with $N_2=1$. We begin by describing the method in Section~\ref{Sec:PronyMethod}. We subsequently show that given the measurements \eqref{eq:MeasurementsProny}, for the regimes introduced in \eqref{Eq:Regimes}, Prony's method achieves first-order optimal recovery, meaning that the asymptotic behavior of the condition numbers associated with Prony's method is the same as that of the analytic first-order condition numbers, which we have obtained in Theorem~\ref{Thm:FirstOrdCond}. 


\begin{remark}
 Differently from the case of the analytic condition numbers, here we will also consider the case $\eta=1$ at different stages of the proofs. Specifically, deriving an upper bound on $|\phi_n-\hat{\phi}_n|$, $n\in [N_1]$, where $\hat{\phi}_n$ is the Prony approximation obtained for $\phi_n$ from the perturbed Prony polynomial $\bar{q}(z)$ (see Section~\ref{Sec:PronyMethod}), will be crucial to bound the condition numbers $\mathcal{K}_{\lambda}(n)$ and $\mathcal{K}_{y}(n)$ with $n\in [\lfloor \eta N_1 \rfloor ]$ for a fixed $\eta \in (0,1)$.  
\end{remark}

\begin{remark}\label{rem:combExpl}
In \eqref{eq:MeasurementsProny} we consider a single term in the structured noise term $y_{\text{tail}}$, i.e., $N_2=1$.
The analysis of Prony's method for $N_2>1$ would exhibit a combinatorial explosion of terms, which \emph{essentially} complicates the derivation of the condition numbers. Therefore, the consideration of measurements \eqref{eq:MeasExpsum}
with arbitrary $N_2\in \mathbb{N}$ is left for future research.
\end{remark}

\subsection{Prony's method}\label{Sec:PronyMethod}

Prony's method reduces the exponential fitting problem to a three-step procedure, which involves solving two linear systems, in combination with a root-finding step, as described in \prettyref{alg:classical-prony}.

Let us denote the \emph{unperturbed} measurements, obtained by setting $\varepsilon=0$ in \eqref{eq:MeasurementsProny}, as
$$
m_k := \sum_{n=1}^{N_1} y_n \phi_n^k, \quad k=0,1,\dots, 2N_1-1.
$$

Unless stated otherwise, all polynomials are considered to be of complex variable.
The \emph{unperturbed} monic Prony polynomial is  given by
\begin{equation}\label{eq:prony-monic-def}
    p\left(z\right):=\prod_{n=1}^{N_1}\left(z-\phi_{n}\right) = z^{N_1}+\sum_{j=0}^{N_1-1}p_{j}z^{j}.
\end{equation}
    
Prony's main insight is that the unperturbed measurements $\{m_k\}_{k=0}^{2N_1-1}$ satisfy a linear recurrence relation with the coefficients $\{p_j\}_{j=0}^{N_1-1}$, which can therefore be obtained by solving a Hankel-type linear system:
\begin{align}
    &H_{N_1} \cdot \operatorname{col}\left\{ p_j\right\}_{j=0}^{N_1-1}  = -   \operatorname{col}\left\{m_k \right\}_{k=N_1}^{2N_1-1}, \quad
    H_{N_1}: = \begin{bmatrix}
m_0 & m_1 & \dots &m_{N_1-1}\\
\vdots & \vdots & \ddots & \vdots\\
m_{N_1-1} & m_{N_1} & \dots & m_{2N_1-2}
 \end{bmatrix}.\label{eq:UnpertHanekl}
\end{align}
The fact that the actual measurements are $\tilde{m}_k=y(t_k)$  in \eqref{eq:MeasurementsProny} gives rise to a perturbed Hankel matrix 
\begin{align}
    &\tilde{H}_{N_1}: = \begin{bmatrix}
        \tilde{m}_{0} & \tilde{m}_{1} & \tilde{m}_{2} & \dots & \tilde{m}_{N_1-1}\\
        \vdots & \vdots & \vdots & \ddots & \vdots\\
        \tilde{m}_{N_1-1} & \tilde{m}_{N_1} & \tilde{m}_{N_1+1} & \dots & \tilde{m}_{2N_1-2}
        \end{bmatrix}, \qquad \tilde{m}_k=y(t_k).\label{eq:PertHankel}
\end{align}
The solution to the linear system 
\begin{align}
    &\tilde{H}_{N_1}\cdot \operatorname{col}\left\{q_j \right\}_{j=0}^{N_1-1} = -   \operatorname{col}\left\{\tilde{m}_k\right\}_{k=N_1}^{2N_1-1} \label{eq:PertHankelSol}
\end{align}
provides the coefficients of the \emph{perturbed} monic Prony polynomial
\begin{equation}\label{eq:q-monic-def}
q\left(z\right):=\prod_{n=1}^{N_1}\left(z-\tilde{\phi}_{n}\right) = z^{N_1}+\sum_{j=0}^{n-1}q_{j}z^{j}.
\end{equation}
The roots $\{\tilde{\phi}_n\}_{n=1}^{N_1}$ of $q\left(z\right)$ provide the node estimates and are then used to obtain the amplitude estimates $\left\{\tilde{y}_n \right\}_{n=1}^{N_1}$  by solving the Vandermonde linear system
\begin{equation}\label{eq:PertVandSol}
\tilde{V} \cdot \operatorname{col}\{\tilde{y}_n\}_{n=1}^{N_1} = \operatorname{col}\{\tilde{m}_n\}_{n=0}^{N_1-1},\quad \tilde{V} := \begin{bmatrix}
    1&  \dots &1 &1\\
    \tilde{\phi}_1 &  \dots & \tilde{\phi}_{N_1-1}& {\tilde{\phi}}_{N_1}\\
    \vdots  & \ddots & \vdots&\vdots\\
    \tilde{\phi}_1^{N_1-1} &  \dots & \tilde{\phi}_{N_1-1}^{N_1-1}& {\tilde{\phi}}_{N_1}^{N_1-1}
    \end{bmatrix}.
\end{equation}

\begin{algorithm2e}[!htb]
    \SetKwInOut{Input}{Input} \SetKwInOut{Output}{Output}
    \Input{Sequence $\{\tilde{m}_k\}_{k=0}^{2N_1-1}$}
    \Output{Estimates for the nodes $\{\phi_n\}_{n=1}^{N_1}$ and amplitudes
      $\{y_n\}_{n=1}^{N_1}$}
      From $\{\tilde{m}_k\}$, construct the Hankel matrix $\tilde{H}_{N_1}$ given in \eqref{eq:PertHankel} \;

      Assuming $\operatorname{det}\tilde{H}_{N_1} \neq 0$, solve the linear system \eqref{eq:PertHankelSol} to obtain ${q_j}$ \;
  
      Compute the roots $\{\tilde \phi_n\}_{n=1}^{N_1}$ of the (perturbed) Prony polynomial $q(z)$ with coefficients ${q_j}$ as in \eqref{eq:q-monic-def} \;

      Construct $\tilde{V}:= \big[\tilde{\phi}_j^k\big]_{k=0,\dots,n-1}^{j=1,\dots,n}$  and solve the linear system \eqref{eq:PertVandSol} to obtain $\{\tilde y_n\}_{n=1}^{N_1}$
 
      \Return the estimates $\{\tilde \phi_n\}_{n=1}^{N_1}$ and $\{\tilde y_n\}_{n=1}^{N_1}$.
    \caption{Classical Prony's method}
    \label{alg:classical-prony}
  \end{algorithm2e}

Following \cite{katz2024accuracy}, we consider the  ``homogeneous'' version of Prony's method given in Algorithm~\ref{alg:homo-prony}, which is computationally equivalent to \prettyref{alg:classical-prony}; in this case, $\bar{p}(z)$ denotes the corresponding homogenized version of $p(z)$,
\begin{align}
\bar{p}\left(z\right)=\operatorname{det}\begin{bmatrix}1 & z & z^{2} & \dots & z^{N_1-1} & z^{N_1}\\
m_{0} & m_{1} & m_{2} & \dots & m_{N_1-1} & m_{N_1}\\
\vdots & \vdots & \vdots & \ddots & \vdots & \vdots\\
m_{N_1-1} & m_{N_1} & m_{N_1+1} & \dots & m_{2N_1-2} & m_{2N_1-1}
\end{bmatrix}=: \det \mathfrak{P}(z)\label{eq:HomogPronyPol},
\end{align}
while $\bar{q}(z)$ denotes the homogenized version of $q(z)$,
	\begin{equation}\label{eq:qbar-def}
	\bar{q}(z)=\operatorname{det}\begin{bmatrix}1 & z & z^{2} & \dots & z^{N_1-1} & z^{N_1}\\
\tilde{m}_{0} & \tilde{m}_{1} & \tilde{m}_{2} & \dots & \tilde{m}_{N_1-1} & \tilde{m}_{N_1}\\
\vdots & \vdots & \vdots & \ddots & \vdots & \vdots\\
\tilde{m}_{N_1-1} & \tilde{m}_{N_1} & \tilde{m}_{N_1+1} & \dots & \tilde{m}_{2N_1-2} & \tilde{m}_{2N_1-1}
\end{bmatrix}=: \det \mathfrak{Q}(z).
\end{equation}

Polynomials $p(z)$ and $\bar{p}(z)$, as well as $q(z)$ and $\bar{q}(z)$, have the same roots, although $p(z)$ and $q(z)$ are monic while $\bar{p}(z)$ and $\bar{q}(z)$ are not. In the following, we will consider the Homogeneous Prony's method in Algorithm~\ref{alg:homo-prony} and hence focus on $\bar{p}$ and $\bar{q}$.

\begin{algorithm2e}[!htb]
  \SetKwInOut{Input}{Input} \SetKwInOut{Output}{Output}
  \Input{Sequence $\{\tilde{m}_k\}_{k=0}^{2N_1-1} $}
  \Output{Estimates for the nodes $\{\phi_n\}_{n=1}^{N_1}$ and amplitudes
    $\{y_n\}_{n=1}^{N_1}$}
	From $\{\tilde{m}_k\}$, compute the roots $\{\tilde{\phi}_n\}_{n=1}^{N_1}$ of the \emph{perturbed} Prony polynomial $\bar{q}(z)$ in \eqref{eq:qbar-def} \;
Construct $\tilde{V}:= \big[\tilde{\phi}_j^k\big]_{k=0,\dots,n-1}^{j=1,\dots,n}$  and solve the linear system \eqref{eq:PertVandSol}  to obtain $\{\tilde y_n\}_{n=1}^{N_1}$ 

\Return the estimates $\{\tilde \phi_n\}_{n=1}^{N_1}$ and $\{\tilde y_n\}_{n=1}^{N_1}$.
\caption{Homogeneous Prony's method}
 \label{alg:homo-prony}
\end{algorithm2e}

\subsection{Bounds on the condition numbers for Prony's method}\label{sec:pronystab}

Recalling the definition of $\phi_n$ in \eqref{eq:phinDef}, we start by presenting a simple estimate that is fundamental to the subsequent analysis and will be used multiple times throughout the paper.

\begin{proposition}\label{Prop:relSepUnif} Let $\Delta_*>0$. There exists $0<\tau<\frac{1}{2}$, which depends only on  $\Delta_*$, such that, uniformly in $n\in \mathbb{N}$ and $\Delta \geq \Delta_*$,
\begin{equation}\label{eq:TauSeparationDelta}
j\in \mathbb{N}\setminus \left\{n \right\} \Longrightarrow \frac{\left|\phi_j-\phi_n \right|}{\phi_n}\geq 2\tau.  
\end{equation}
\end{proposition}
\begin{proof}
See Appendix~\ref{Sec:ProofPropositionPhi}.
\end{proof}

We can now show the optimality of Prony's method, whose first-order condition numbers have the same asymptotic behavior as the analytic ones.

{\bfseries \noindent Proof strategy outline.}  
Before stating our main theorems, we provide an overview of the proof strategy behind them. To obtain bounds on the first-order condition numbers for Prony's method in the regimes listed in \eqref{Eq:Regimes}, we employ an argument consisting of the following steps:
\begin{enumerate}
    \item We consider the \emph{unperturbed} non-monic Prony polynomial $\bar{p}(z)$, associated with noise-free measurements \eqref{eq:MeasurementsProny} with $\varepsilon=0$, and the corresponding \emph{perturbed} non-monic Prony polynomial $\bar{q}(z)$, associated with noisy measurements \eqref{eq:MeasurementsProny} with $\varepsilon>0$, as defined in Section~\ref{Sec:PronyMethod}. Building upon previous results obtained in \cite{katz2024accuracy}, we derive an explicit representation for the discrepancy $\bar{q}(z)-\bar{p}(z)$: see Lemma~\ref{lem:PolyDiscrep} in Appendix~\ref{Append:Pf8} below.

    \item We consider the regimes listed in \eqref{Eq:Regimes} and we show that, for $\tau>0$ chosen as in Proposition~\ref{Prop:relSepUnif}, there exists $\varepsilon_0>0$ such that, for all $\varepsilon\in [0,\varepsilon_0]$, the perturbed Prony polynomial $\bar{q}(z)$ has roots $\hat{\phi}_n^{Pr}$, $n\in [N_1]$, satisfying the relative error estimate $\left|\hat{\phi}_n^{Pr} - \phi_n \right|<\tau \phi_n$ for all $n\in [N_1]$, as stated in Theorem~\ref{thm:unifbdd} below.  The roots $\hat{\phi}_n^{Pr} $ are then employed to generate approximations $\hat{\lambda}_n^{Pr}$ of the eigenvalues, obtained as $\hat{\lambda}_n^{Pr}:=-\Delta^{-1} \ln \left(\hat{\phi}_n^{Pr} \right)$.

    \item Then, we fix $\eta\in (0,1)$ and we show that $\left|\hat{\lambda}_n^{Pr}-\lambda_n \right| = \mathcal{K}_{\lambda}^{Pr}(n)\cdot \varepsilon$ for appropriately defined condition numbers $\mathcal{K}_{\lambda}^{Pr}(n)$ associated with Prony's method. Moreover, for $\varepsilon\in (0,\varepsilon_1]$ with some $0<\varepsilon_1\leq \varepsilon_0$, $\max_{n\in [\lfloor \eta N_1 \rfloor]}\left| \mathcal{K}_{\lambda}^{Pr}(n)\right|$ has an asymptotic behavior that matches the analytic asymptotic behavior obtained in  Theorem~\ref{Thm:FirstOrdCond}, for all regimes \eqref{Eq:Regimes}. This result is shown in Theorem~\ref{thm:unifbdd11}.

    \item Finally, we show that Prony's method yields approximations $\hat{y}_n^{Pr}$ satisfying $\left|\hat{y}_n^{Pr}-y_n \right|=\mathcal{K}_y^{Pr}(n)\cdot \varepsilon$. We derive an asymptotic bound  on $\max_{n\in [N_1]}\left| \phi_n -\hat{\phi}^{Pr}_n\right|$, which is used to bound $\max_{n\in [\lfloor \eta N_1 \rfloor]}\left|\mathcal{K}_y^{Pr}(n)\right|$. We then show that $\max_{n\in [\lfloor \eta N_1 \rfloor]}\left|\mathcal{K}_y^{Pr}(n)\right|$ also exhibits the same asymptotic behavior as the analytic one obtained  in  Theorem~\ref{Thm:FirstOrdCond} for all regimes \eqref{Eq:Regimes}, provided $\varepsilon \in (0,\varepsilon_1]$. This result is proved in Theorem~\ref{thm:PronyAmplitudeCondition}.

\end{enumerate}

\begin{remark}\label{rem:prony-sequence}
For each regime, we analyze a sequence of Prony problems with different characteristic parameters ($N_1$ or $\Delta$). For example, for Regime 1 in \eqref{Eq:Regimes}, the degree of both polynomials $\bar{p}(z)$ and $\bar{q}(z)$ grows as $N_1$ when $N_1 \to \infty$. This problem is different from the classical super-resolution problem (see, e.g., \cite{katz2024accuracy}), where the number of nodes $\left\{\phi_n \right\}_{n=1}^{N_1}$ is usually fixed. Furthermore, the classical super-resolution problem assumes that the nodes $\left\{\phi_j\right\}_{j=1}^{N_1}$ satisfy a minimal separation bound of the form  $\min_{i\neq j}\left|\phi_i - \phi_j \right|\geq \delta $ for some fixed $\delta>0$. Since the eigenvalues $\left\{\lambda_n \right\}_{n=1}^{\infty}$ are such that $\lim_{n\to \infty}\lambda_n = +\infty$, in our case a minimal separation parameter $\delta>0$ does not exist for $\left\{\phi_n\right\}_{n=1}^{N_1}$ of the form $\phi_n = e^{-\Delta \lambda_n}$ if $N_1$ is large enough. Compared to the classical super-resolution problem, these differences make the analysis of Prony's method in the considered context significantly more challenging.
\end{remark}

\begin{theorem}\label{thm:unifbdd}
    Let $\Delta_* >0$ be fixed and consider Regimes 1 and 2 in \eqref{Eq:Regimes}. There exist $\varepsilon_{\Delta_*}>0$ such that, for all $(n,N_1)\in \mathcal{N}_{1}$ and all $\varepsilon\in [0,\varepsilon_{\Delta_*}]$, the perturbed Prony polynomial $\bar{q}\left(z\right)$ in \eqref{eq:qbar-def} corresponding to $\varepsilon$ has a root $\hat{\phi}_n^{Pr}$ satisfying the relative error bound 
    \begin{equation}\label{eq:ZetaBound}
    \left|\phi_n-\hat{\phi}_n^{Pr} \right|\leq \tau\phi_n.
    \end{equation}
    Analogously, let $T>0$ be fixed and consider Regime 3 in \eqref{Eq:Regimes}. There exist $\varepsilon_{T}>0$ such that, for all $(n,N_1)\in \mathcal{N}_{1}$ and all $\varepsilon\in [0,\varepsilon_{T}]$, the perturbed Prony polynomial $\bar{q}\left(z\right)$ corresponding to $\varepsilon$ has a root $\hat{\phi}_n^{Pr}$ satisfying the relative error bound in \eqref{eq:ZetaBound}.
\end{theorem}
\begin{proof}
    See Appendix~\ref{Append:ProofUnifBdd}.
\end{proof}

\begin{corollary}\label{cor:phinPrreal}
Under the assumptions of Theorem~\ref{thm:unifbdd}, all the roots $\hat{\phi}_n^{Pr}$, $n\in [N_1]$, of the perturbed Prony polynomial $\bar{q}\left(z\right)$ in \eqref{eq:qbar-def} corresponding to $\varepsilon$ are real.
\end{corollary}




Theorem~\ref{thm:unifbdd} shows that  there exists a sufficiently small noise intensity such that the perturbed Prony polynomial $\bar{q}(z)$ has all of its roots $\left\{ \hat \phi_n^{Pr}\right\}_{n=1}^{N_1}$ in a controlled relative neighborhood of the nodes $\left\{ \phi_n\right\}_{n=1}^{N_1}$ to be estimated, meaning that the roots of $\bar{q}(z)$ satisfy  the relative error estimate $\frac{|\phi_n-\hat{\phi}_n^{Pr}|}{\phi_n}\leq \tau$, \emph{uniformly in the parameters specified by each regime}. The resulting inequality
\begin{equation*}    
(1-\tau)\phi_n\leq |\hat{\phi}_n^{Pr}|\leq  (1+\tau)\phi_n
\end{equation*}
will be used to control the magnitude of the approximations $\hat{\phi}_n^{Pr}$ in the next theorem, which is one of our main results.

\begin{theorem}\label{thm:unifbdd11}
    Fix $\Delta_*>0$. There exists $\varepsilon_*>0$ such that, for all $\varepsilon \in [0,\varepsilon_*]$, the perturbed Prony polynomial $\bar{q}(z)$ in \eqref{eq:qbar-def} obtained for noise intensity $\varepsilon$ yields roots $\hat{\phi}_n^{Pr}$ and corresponding estimates $\hat{\lambda}_n^{Pr} = -\frac{\ln\left(\hat{\phi}_n^{Pr}\right)}{\Delta}$ of the exponents $\lambda_n$ that satisfy
    \begin{equation}\label{eq:Basic estimates}
    \begin{array}{lll}
     &\left|\hat{\phi}_n^{Pr}-\phi_n \right|\leq \varepsilon \mathcal{D}_{\phi} \,\, \frac{\phi_n \left( \prod_{s\neq n}\phi_s^2\right)}{\prod_{s\neq n}|\phi_{s}-\phi_{n}|^2 },\qquad \left|\hat{\lambda}_n^{Pr}-\lambda_n \right|\leq \varepsilon \mathcal{D}_{\lambda} \,\, \frac{1}{\Delta} \frac{\left( \prod_{s\neq n}\phi_s^2\right)}{\prod_{s\neq n}|\phi_{s}-\phi_{n}|^2 },
    \end{array}
    \end{equation}
    for some constants $\mathcal{D}_{\phi}>0$ and $\mathcal{D}_{\lambda}>0$ that depend on the considered regime.
    In particular, given $\eta\in (0,1)$, the following asymptotic bounds are obtained:
    
    \smallskip
    \underline{Regime 1:} For $N_1\gg 1$, 
    \begin{equation}\label{eq:Asymptoticphi1}
    \begin{array}{lll}
        &\left|\hat{\phi}_n^{Pr}-\phi_n \right|\leq \varepsilon \mathcal{D}^{(1)}_{\phi}e^{-\Delta \Gamma^{Pr}(\eta)N_1^3},\quad (n,N_1)\in \mathcal{N}_{\eta} ,\\
        &\left|\hat{\phi}_n^{Pr}-\phi_n \right|\leq \varepsilon \mathcal{D}_{\phi}^{(2)}e^{-\Delta \Gamma^{Pr}_1(\eta)N_1^2},\quad (n,N_1)\in \mathcal{N}_1 \setminus \mathcal{N}_{\eta},\\
        &\left|\hat{\lambda}_n^{Pr}-\lambda_n \right|\leq \varepsilon \frac{\mathcal{D}^{(1)}_{\lambda}}{\Delta}e^{-\Delta \Gamma^{Pr}(\eta)N_1^3},\quad (n,N_1)\in \mathcal{N}_{\eta},
    \end{array}
    \end{equation}
    for some $\Gamma^{Pr}(\eta),\Gamma_1^{Pr}(\eta),\mathcal{D}^{(1)}_{\phi},\mathcal{D}^{(2)}_{\phi},\mathcal{D}^{(1)}_{\lambda}>0$.
    
    \medskip
    \underline{Regime 2:} Let $\vartheta>0$. There exists $N_1(\vartheta)\in \mathbb{N}$ such that, for $\Delta\gg 0$,
    \begin{equation}\label{eq:Asymptoticphi11}
    \begin{array}{lll}
        &\left|\hat{\phi}_n^{Pr}-\phi_n \right|\leq \varepsilon \mathcal{D}^{(3)}_{\phi}e^{-\vartheta \Delta}, \quad (n,N_1)\in \mathcal{N}_{1},\\
        &\left|\hat{\lambda}_n^{Pr}-\lambda_n \right|\leq \varepsilon \frac{\mathcal{D}^{(1)}_{\lambda}}{\Delta}e^{-\vartheta \Delta},\quad (n,N_1)\in \mathcal{N}_{\eta},
    \end{array}
    \end{equation}
    uniformly in $N_1\geq N_1 (\vartheta)$ for some constant $\mathcal{D}^{(3)}_{\phi}>0$.

    \medskip
    \underline{Regime 3:} For $T>0$, $\Delta N_1 =T$ and  $N_1\gg 1$,
    \begin{equation}\label{eq:Asymptoticphi111}
    \begin{array}{lll}
        &\left|\hat{\phi}_n^{Pr}-\phi_n \right|\leq \varepsilon \mathcal{D}_{\phi}^{(4)} e^{-\Gamma^{Pr}_2(T,\eta)N_1^2}, \quad (n,N_1) \in \mathcal{N}_{\eta},\\
        &\left|\hat{\phi}_n^{Pr}-\phi_n \right|\leq \varepsilon \mathcal{D}_{\phi}^{(5)} e^{-\Gamma^{Pr}_3(T,\eta)N_1}, \quad (n,N_1) \in \mathcal{N}_{1}\setminus \mathcal{N}_{\eta},\\
        &\left|\hat{\lambda}_n^{Pr}-\lambda_n \right|\leq \varepsilon \frac{\mathcal{D}_{\lambda}^{(2)}}{\Delta}e^{-\Gamma^{Pr}_2(T,\eta)N_1^2},\quad (n,N_1)\in \mathcal{N}_{\eta}, 
    \end{array}
    \end{equation}
    for some $\Gamma^{Pr}_2(T,\eta),\Gamma^{Pr}_3(T,\eta),\mathcal{D}_{\phi}^{(4)},\mathcal{D}_{\phi}^{(5)},\mathcal{D}_{\lambda}^{(2)}>0$.
    
    \medskip
    In particular, in all the regimes with $(n,N_1)\in \mathcal{N}_{\eta}$, the bounds on the condition numbers $\mathcal{K}_{\lambda}^{Pr}(n)$, associated with Prony's method, exhibit the same asymptotic behavior as the condition numbers obtained in Theorem~\ref{Thm:FirstOrdCond}. 
    \end{theorem}
    
    \begin{proof}
    See Appendix~\ref{Append:Pf13}.
    \end{proof}

    Theorem~\ref{thm:unifbdd11} shows that the first-order condition numbers of Prony's method associated with the recovery of the eigenvalues $\lambda_n$, $n\in [\lfloor \eta N_1 \rfloor]$, exhibit an asymptotic behavior that is identical to that of the corresponding analytic first-order condition numbers, obtained in Theorem~\ref{Thm:FirstOrdCond}, for all regimes \eqref{Eq:Regimes}. 
    
    We can also prove that the same holds for the first-order condition numbers of Prony's method associated with the recovery of the amplitudes $y_n$, $n\in [\lfloor \eta N_1 \rfloor]$. As shown, e.g., in Eqn. (45) in \cite{katz2024accuracy}, in Prony's method, the approximations $\left\{\hat{y}_n^{Pr} \right\}_{n=1}^{N_1}$ satisfy the relation
    \begin{equation}\label{eq:PronyYdiff}
    \begin{array}{lll}
    \operatorname{col}\left\{ \hat{y}_n^{Pr} - y_n\right\}_{n=1}^{N_1}  = \hat{V}^{-1}\cdot \operatorname{col} \left\{\varepsilon y_{N_1+1}e^{-\Delta \lambda_{N_1+1} n} \right\}_{n=0}^{N_1-1} - \left(I_{N_1}-\hat{V}^{-1}V \right)\cdot \operatorname{col}\left\{y_n \right\}_{n=1}^{N_1},
    \end{array}
    \end{equation}
    where
    \begin{equation}\label{eq:VDMmatDef}
    \begin{array}{lll}
    &\hat{V}^{-1}:=\left[\begin{array}{ccc}
    \hat{l}_{1,0} & \ldots & \hat{l}_{1, N_1-1} \\
    \vdots & \ddots & \vdots \\
    \hat{l}_{N_1, 0} & \ldots & \hat{l}_{N_1, N_1-1}
    \end{array}\right],\quad V^{-1}:=\left[\begin{array}{ccc}
    l_{1,0} & \ldots & l_{1, N_1-1} \\
    \vdots & \ddots & \vdots \\
    l_{N_1, 0} & \ldots & l_{N_1, N_1-1}
    \end{array}\right],\vspace{0.2cm}\\
    &\hat{L}_i(z):=\prod_{k \neq i} \frac{z-\hat{\phi}_k}{\hat{\phi}_i-\hat{\phi}_k}=\sum_{s=0}^{N_1-1} \hat{l}_{i, s} z^s, \quad L_i(z):=\prod_{k \neq i} \frac{z-\phi_k}{\phi_i-\phi_k}=\sum_{s=0}^{N_1-1} l_{i, s} z^s,
    \end{array}
    \end{equation}
   and $\hat{L}_i(z)$ and $L_i(z)$, $i\in [N_1]$, are the Lagrange polynomials corresponding to $\hat{\phi}_n^{Pr}$ and $\phi_n$, $n\in [N_1]$, respectively. Employing the presentation \eqref{eq:PronyYdiff} enables us to prove the other main result of this section.
    \begin{theorem}\label{thm:PronyAmplitudeCondition}
    Let $\Delta_*>0$ and $\eta\in (0,1)$ be fixed. There exists $\varepsilon_*>0$ such that for all $\varepsilon \in (0,\varepsilon_*]$, the following asymptotic bounds hold:

    \smallskip
    \underline{Regime 1:} Uniformly in $\Delta\in [\Delta_*,\infty)$, 
    \begin{equation}\label{eq:Asymptoticy1}
    \begin{array}{lll}
        &\left|\hat{y}_n^{Pr}-y_n \right|\leq \varepsilon \mathcal{D}^{(1)}_{y}e^{-\Delta \Gamma^{Pr}_y(\eta)N_1^3},\quad (n,N_1)\in \mathcal{N}_{\eta},\ N_1\gg 1,
    \end{array}
    \end{equation}
    for some $\Gamma^{Pr}_y(\eta),\mathcal{D}^{(1)}_{y}>0$.

    \medskip
    \underline{Regime 2:} Let $\vartheta>0$. There exists $N_1(\vartheta)\in \mathbb{N}$ such that, uniformly in $N_1 \geq N_1(\vartheta)$, 
    \begin{equation}\label{eq:Asymptoticy11}
    \begin{array}{lll}
        &\left|\hat{y}_n^{Pr}-y_n \right|\leq \varepsilon \mathcal{D}^{(2)}_{y}e^{-\vartheta \Delta}, \quad (n,N_1)\in \mathcal{N}_{\eta},\ \Delta \gg0,
    \end{array}
    \end{equation}
    for some constant $\mathcal{D}^{(2)}_{y}>0$.

     \medskip
    \underline{Regime 3:} Given $T>0$, $\Delta N_1 =T$, 
    \begin{equation}\label{eq:Asymptoticy111}
    \begin{array}{lll}
        &\left|\hat{y}_n^{Pr}-y_n \right|\leq \varepsilon \mathcal{D}_{y}^{(3)} e^{-\Gamma^{Pr,1}_{y}(T,\eta)N_1^2}, \quad (n,N_1) \in \mathcal{N}_{\eta},\ N_1 \gg 1,
    \end{array}
    \end{equation}
    for some $\Gamma^{Pr,1}_y(T,\eta),\mathcal{D}_{y}^{(3)}>0$.
    \end{theorem}
    \begin{proof}
    See Appendix~\ref{Sec:PronyAmplitProof}.
    \end{proof}
    The results in Theorems~\ref{thm:unifbdd11} and~\ref{thm:PronyAmplitudeCondition}, further demonstrated by the simulations shown in Section~\ref{Sec:NumericalExamples}, 
    prove that Prony's method is first-order optimal for recovering the eigenvalues $\left\{\lambda_n\right\}_{n\in [\lfloor \eta N_1 \rfloor]}$ and amplitudes  $\left\{y_n\right\}_{n\in [\lfloor \eta N_1 \rfloor]}$ from the measurements \eqref{eq:MeasurementsProny} in the small noisy intensity regime. The analysis of Prony's method in the case of general measurements \eqref{eq:MeasExpsum} is essentially more complex, as mentioned in Remark~\ref{rem:combExpl}, and we leave it for future research.

\section{Related work and discussion}\label{sec:related-work}

\subsection{Algebraic methods of multi-exponential analysis}

Prony's method is a primary example of the so-called algebraic methods \cite{istratov1999}, building a system of algebraic equations that relate the measurements to the parameters (in our case, exponents and amplitudes), and subsequently solving these equations.  Other  methods from this category include the Pisarenko method, the Approximate Prony's method, the Matrix Pencil method, the MUSIC/ESPRIT methods from signal processing, and many others \cite{batenkov2013b}. The different methods have their relative strengths and weaknesses; however, all methods deteriorate when the number of nodes becomes very large, and even worse, the minimal separation between the exponents shrinks to zero. Such ill-conditioning is inherent in the problem without any additional assumptions, namely, when the perturbation/noise is only assumed to be of bounded magnitude/variance, as we have discussed in the paragraph following equation \eqref{eq:example-amplitude-error} in Section~\ref{sec:intro}. One of the main contributions of this paper is to show that such ill-conditioning completely disappears  when the noise is itself of a multi-exponential form, causing error cancellations. The cancellation phenomenon was previously observed in the Fourier setting \cite{batenkov2012a, batenkov2015, levinov2025} but only with fixed $N_1$; here we demonstrate it also in the Laplace transform setting (cf. \eqref{eq:LaplaceTransform} below) and $N_1\to\infty$. Finally, the algebraic methods were recently shown to be optimal for the problem of super-resolution of sparse measures \cite{batenkov2021b, katz2024accuracy}; we utilize the algebraic framework developed there and extend it significantly to the new setting given by the measurements \eqref{eq:MeasExpsum}.


\subsection{Inverse problems for parabolic PDEs} \label{sec:inverse-problems-pde} 

Considering \eqref{eq:diffusion-pde} above, where $\mathcal{A}$ satisfies Assumption~\ref{assump:eigenvalues}, the measurement trace $y(t)$ in \prettyref{prob:main-problem} with samples of the form \eqref{eq:MeasExpsum} can be obtained from the series representation \eqref{eq:diffusion-solution-series} in different ways:
\begin{enumerate}
    \item Sampling the solution $z(x,t)$ at a fixed point $x=x_0\in (0,1)$ (which is not in the zero locus of $\{\psi_n(x)\}$, a countable set) to obtain $y(t)=z(x_0,t)$; see \eqref{eq:PDESamplesControl}. Then, $y_n=\left\langle f,\psi_n\right\rangle \psi_n(x_0)$.
    \item More generally, by considering a (distributional) measurement kernel $c\in \mathcal{D'}(0,1)$, see e.g. \cite{strichartz2003}, and setting $y(t)=\left\langle z(\cdot, t), c\right\rangle$. As a special case, the point measurement $y(t)=z(x_0,t)$ is obtained by taking $c=\delta(x-x_0)$, where $\delta(\cdot)$ is the standard Dirac delta. Then, $y_n = \left\langle f,\psi_n\right\rangle \left\langle c,\psi_n\right\rangle$.
\end{enumerate}
The spectral cutoff at $n=N_2$ and the sharp drop in the relative magnitude between the signal $y_{\text{main}}$ and the noise $\varepsilon y_{\text{tail}}$ in Assumption~\ref{assump:yassump} are sensible approximations in practice. For example, such noise control can be guaranteed by the smoothness of the solution in case 1, or the approximate band-limitedness of the measurement kernel in case 2 above. See \cite{katz2024,katz2024c} for more details.

Let us set $p(x)\equiv 1$ in \eqref{eq:sl-operator} and consider the problem of recovering the potential $q(x)$ from the measurement trace $y(t)$. This class of problems is widely studied in the literature, cf. \cite{avdonin1997, isakov2017, hasanovhasanoglu2021a, pierce1979a,zhang2022a, rundell2024} and references therein. In this context, solving Problem~\ref{prob:main-problem} is frequently considered as a first step in the following two-step procedure \cite{pierce1979a}:
\begin{enumerate}
    \item[I.] Recover the exponents and amplitudes  $\{\lambda_n,y_n\}_{n=1}^{N_1}$ by inverting the Dirichlet series $\sum_{n=1}^\infty y_n e^{-\lambda_n t}$. This problem is equivalent to analytic continuation and thus uniquely solvable, as the series converges in the right half-plane; see e.g. \cite[Lemma~3.1]{rundell2024}.
    \item[II.] Recover the potential $q(x)$ from the spectral data  by solving the so-called inverse spectral problem, referred to in this context as the inverse Sturm-Liouville problem. A widely used class of approaches utilizes Gel'fand-Levitan integral equations, cf. \cite{chadan1997a,GelLev51,kravchenko2020, pierce1979a, rundell1992, lowe1992, fabiano1995,jiang2022, kravchenko2021} and references therein.
\end{enumerate}

While  step II above is well-studied, step I is less so, in particular relating to its numerical stability. Existing approaches include: Laplace transform methods, the method of limits, the ``hot-spot'' method, and classical spectral estimation \cite{boumenir2010}. Additional conditions are usually required to ensure well-posedness of step 2, such as multiple spectra for different initial/boundary conditions. For an example of the latter, see recent work by Rundell \cite{rundell2024} where the special case of $n=1,2$ is utilized in a multi-impedance recovery scheme. 

In the Numerical Experiments Section~\ref{Sec:NumericalExamples}, we implement the two-step approach above to recover the unknown potential $q(x)$. Specifically, in step I, approximations of the exponents and amplitudes are obtained from the measurements \eqref{eq:MeasExpsum} by employing Prony's method, whereas step II employs a
special case of the method described by Rundell in \cite[Section~3.14]{chadan1997a}, which is valid for symmetric potentials $q(x)$ with $\|q\|\ll 1$.

\subsection{Our earlier work}
The results reported in this paper significantly generalize and  extend our preliminary works \cite{katz2024accuracy,katz2024, katz2024c} on the subject. In \cite{katz2024c} we considered only Regime 2, with application to delay estimation in reaction-diffusion systems, while in \cite{katz2024} we investigated only Regime 1. Both \cite{katz2024} and \cite{katz2024c} only provided estimates on the analytic (first-order) condition numbers, without demonstrating rigorously that a specific numerical method for exponent and amplitude recovery actually attains the bounds on the analytic condition numbers. The study of the first order optimality of Prony's method is a central topic of the present work. Our techniques for the analysis of Prony's method build upon the earlier work \cite{katz2024accuracy} where we considered non-structured noise in the case of super-resolution, with  the exponents being purely imaginary numbers (cf. \prettyref{rem:prony-sequence}). As mentioned in the Introduction, the super-resolution setting is an essentially different exponential fitting problem from the one we consider in this work.

\subsection{Resolution theory of the Laplace transform}

As discussed in the Introduction, the parametric inversion of the model \eqref{eq:MeasExpsum} in Problem~\ref{prob:main-problem} is expected to exhibit extreme ill-posedness if the measurement errors/noise $y_{tail}(t_k)$, $k=0,\dots ,2N_1-1$, therein are replaced with arbitrary non-structured errors/noise. The aim of this section is to provide an intuitive explanation why for the structured noise $y_{tail}(t_k)$, $k=0,\dots ,2N_1-1$, our results show that Problem~\ref{prob:main-problem} is well-posed and, actually, extremely well-conditioned.

\prettyref{prob:main-problem} can be considered a special case of the so-called spectroscopic method of exponential analysis \cite{istratov1999}, where one needs to reconstruct a ``distribution function''  $g(\lambda)$ from its Laplace transform
\begin{equation}\label{eq:LaplaceTransform}
y(t) = \int_0^\infty e^{-\lambda t} g(\lambda) d\lambda \equiv (\mathcal{L}g)(t).
\end{equation}

In our case, we will subsequently be interested in a specific $g$, given by $g_{SL}(\lambda) = \sum_{n=1}^{\infty}y_n \delta(\lambda-\lambda_n)$, with $\left\{\lambda_n \right\}_{n=1}^{\infty}$ being the eigenvalues of the Sturm-Liouvile operator $\mathcal{A}$ in \eqref{eq:sl-operator} and corresponding Laplace measurements $y_{SL}(t)$. In what follows, we intuitively explain how existing results on the famous inverse problem \eqref{eq:LaplaceTransform} in the continuous domain can shed light on our discrete setting  \eqref{eq:MeasExpsum}. In particular, we will be interested in the question: under what conditions can we guarantee that among the $\left\{\lambda_n \right\}_{n=1}^{\infty}$ defining $g_{SL}(\lambda)$, the values $\left\{\lambda_n \right\}_{n=1}^{N_1}$ can be recovered from $y_{SL}(t)$ and, furthermore, this holds for \emph{all} $N_1\gg 1$ (Regime 1 above)?

A classical work on the subject,  \cite{ostrowsky1981}, provides a recipe for stable inversion of the Fredholm integral equation \eqref{eq:LaplaceTransform} according to the principles of information theory, as follows. First, noise in $y(t)$ limits the reconstruction accuracy of $g$, and this tradeoff is quantified using the spectral theory of the operator $\mathcal{L}$ \cite{mcwhirter1978, bertero1982a, bertero1984}. Namely, a stable, or bandlimited, reconstruction $\tilde{g}$ will only include those spectral components of $g$ that are above the noise floor. In more detail, the point spectrum $\operatorname{spec}_p\left(\mathcal{L} \right)=  \{\sigma_{\omega}\}$, which consists of all eigenvalues of $\mathcal{L}$, can be indexed by a ``Laplace frequency'' $\omega$ with the asymptotic relation $\sigma_{\omega} \sim \exp\left(-c|\omega|\right)$ with some $c>0$. Therefore, given noisy data with signal-to-noise ratio (SNR) $\eta$, the maximal frequency above the noise floor (namely, the maximal frequency content of $g(\lambda)$ which is expected to be recoverable from $y(t)$ in \eqref{eq:LaplaceTransform}) is given asymptotically by $\exp(-c\omega_{\max})\sim \eta^{-1}$, i.e. $\omega_{\max} \sim \ln\left(\eta\right)$.

To recover the bandlimited $\tilde{g}$ from $y(t)$ in practice, \cite{ostrowsky1981} proceeds by  choosing a sampling set $\left\{\underline{\lambda}_n \right\}_{n=1}^{\infty}$ (in the $\lambda$ variable) and solving a linear least squares problem to find the values $\tilde{g}(\underline{\lambda}_n)$ for $n=1,2,\dots,N_1$. Here comes the crucial observation that to ensure unique recovery of $\tilde{g}$, the spacing of $\left\{\underline{\lambda}_n \right\}_{n=1}^{\infty}$ must satisfy a resolution criterion (or a sampling theorem): assuming the $\underline{\lambda}_n$'s are arranged in ascending order, \cite[Equation (14)]{ostrowsky1981} imposes the following asymptotic lower bound on consecutive sampling points:
\begin{equation}\label{eq:LaplaceSampling}
\ln \frac{\underline{\lambda}_n}{\underline{\lambda}_{n-1}} \gtrsim \omega_{\max}^{-1}.
\end{equation}
Sampling with higher density will be unstable because the frequency components above $\omega_{\max}$ are unreliable, as they are beyond the noise floor.
There is no actual sampling of $\tilde{g}$ taking place; rather, $\tilde{g}$ is parametrized by the values $\tilde{g}(\underline{\lambda}_n)$, which are ultimately recovered from (the samples of) $y(t)$ in \eqref{eq:LaplaceTransform}.

Going back to our setting, the true signal is the discrete spectral measure $g=g_{SL}$. In Regime 1, as $N_1\to\infty$, the noise goes to zero, and consequently $\omega_{\max}\to\infty$. Therefore $\tilde{g}\to g_{SL}$ in an appropriate sense, and if  \eqref{eq:LaplaceSampling} is satisfied, $g$ is asymptotically recovered. However, since $\underline{\lambda}_n$ are sampling points of $\tilde{g}$, one must  choose them to be asymptotically equal to the true support $\{\lambda_n\}$ of $g$. The problem is therefore reduced to checking when 
\begin{equation}\label{eq:LaplaceSampling2}
\ln \frac{{\lambda}_n}{{\lambda}_{n-1}} \gtrsim \frac{1}{\ln \eta}.
\end{equation}





 To estimate $\eta$, suppose that $y_{SL}(t)$ is measured at a fixed large time $t=T$ and that we regard the  tail 
 \begin{equation*}
   y_{SL}^{\text{tail}}(t)=\int_0^{\infty}e^{-\lambda t}g_{SL}^{\text{tail}}(\lambda)d\lambda=\sum_{n=N_1+1}^{\infty}y_ne^{-\lambda_{n}t}, \quad \text{with} \quad  g_{SL}^{\text{tail}}(\lambda)=\sum_{n=N_1+1}^{\infty} y_n \delta(\lambda-\lambda_n),  
 \end{equation*}
as noise and assume further that $\left\{|y_n|\right\}_{n=1}^{\infty}$ is a rapidly decaying sequence of positive reals such that $\frac{|y_n|}{|y_{n+1}|}$ is bounded. Then the threshold signal-to-noise ratio is equivalent to the ratio of the terms $y_{N_1+1} e^{-\lambda_{N_1+1} t}$ and $y_{N_1} e^{-\lambda_{N} t}$ at time $t=T$. Indeed, these terms are the \emph{largest} component in the noise $y_{SL}^{\text{tail}}(t)$
and the \emph{smallest} component in the Laplace transform of the main signal $\sum_{n=1}^{N_1}y_n\delta(\lambda-\lambda_n)$, respectively. Hence, the threshold SNR $\eta$ is equivalent to
\begin{align}\label{eq:LaplaceSNR}
\begin{split}
\eta &\sim \frac{|y_{N_1}| e^{-\lambda_{N_1} T}}{|y_{N_1+1}| e^{-\lambda_{N_1+1} T}} = \left|\frac{y_{N_1}}{y_{N_1+1}}\right|\exp\left((\lambda_{N_1+1}-\lambda_{N_1})T\right)\\
\implies \ln \eta &\sim \ln \left(\left|\frac{y_{N_1}}{y_{N_1+1}}\right|\exp\left((\lambda_{N_1+1}-\lambda_{N_1})T\right) \right)\sim \left(\lambda_{N_1+1}-\lambda_{N_1} \right),
\end{split}
\end{align}
where we used the boundedness assumptions on the $|y_n|$'s and the fact that $T$ is fixed.
Substituting the latter into \eqref{eq:LaplaceSampling2}, we obtain the following condition on minimal spacing for stable recovery
\begin{equation}\label{eq:Relation}
\ln\lambda_{N_1}-\ln\lambda_{N_1-1} \gtrsim \frac{1}{\lambda_{N_1+1}-\lambda_{N_1}}.
\end{equation}

By comparing \eqref{eq:LaplaceSampling} and \eqref{eq:Relation}, we have that $\omega_{max}\sim \lambda_{N_1+1}-\lambda_{N_1}$. Since we allow for arbitrary $N_1\gg1$, to ensure recovery of the entire pulse train $g_{SL}(\lambda)$, the growth of the maximally recoverable frequency $\omega_{max}\sim \lambda_{N_1+1}-\lambda_{N_1}$ needs to satisfy the relation \eqref{eq:Relation}, where the increase in the logarithmic spacing of the $\lambda_n$'s is shown on the left-hand side.
Suppose that $\lambda_n = cn^p+O(n^{p-1})$ for some $p>0$. Then,  
\begin{equation*}
\ln(\lambda_{N_1})\sim p\ln(N_1)\sim p\sum_{j=1}^{N_1}j^{-1},
\end{equation*}
where the last step follows from the formula for the Euler-Mascheroni constant. Substituting this into \eqref{eq:Relation}, and using the fact that, by Lagrange's theorem,
\begin{equation*}
\lambda_{N_1+1}-
\lambda_{N_1}\sim c\left(N_1+1\right)^{p}-cN_1^{p}\sim N_1^{p-1},
\end{equation*}
we get $\frac{1}{N_1} \gtrsim \frac{1}{N_1^{p-1}}$, which is satisfied for $p\geq 2$. Thus the quadratic growth of the $\lambda_n$'s is a natural edge case for which the Laplace transform inversion is well-conditioned and all spectral components can in principle be recovered. Recalling Proposition~\ref{prop:EigDiffBound}, we see that this is precisely the setting of our paper. 

Making all of the above arguments fully rigorous is out of the scope of this paper.
Moreover, the presented arguments do not extend to Regimes 2 and 3. Finally, although various  numerical schemes for Laplace inversion have been developed \cite{dingfelder2015, lederman2016, horvath2023}, none of them are directly applicable to the fully discrete problem \eqref{eq:MeasExpsum}. Thus, one contribution of this manuscript is to show that the intuitive result provided in this section can, in fact, be justified rigorously. This is the implication of our Theorems~\ref{thm:FirstOrdCondRep}, \ref{Thm:FirstOrdCond}, \ref{thm:unifbdd11} and \ref{thm:PronyAmplitudeCondition}.

\section{Numerical experiments}\label{Sec:NumericalExamples}

\subsection{Asymptotic decay of the condition numbers and Prony's method errors}\label{sec:numerics-condition-numbers}

In this section we report on numerical experiments validating the theoretical predictions in Theorems~\ref{Thm:FirstOrdCond},~\ref{thm:unifbdd11}, and~\ref{thm:PronyAmplitudeCondition} for condition number decay across all the three asymptotic regimes in \eqref{Eq:Regimes}. Our implementation uses the \texttt{Julia} language \cite{bezanson2017} and high-precision arithmetic (\texttt{BigFloat} type  with 9000 bits, approximately 2700 decimal digits) to ensure accurate computation of condition numbers for both the analytic estimates obtained in Section~\ref{Sec:WellPosed} and Prony's method in Section~\ref{Sec:PronyProb}.

We implement the regimes \eqref{Eq:Regimes} as follows:
\begin{itemize}
\item \textbf{Regime 1}: Fixed $\Delta = 0.1$, $N_1 \in [25, 65]$ (expanding interval with increasing sampling density)
\item \textbf{Regime 2}: Fixed $N_1 = 10$, $\Delta \in [0.1, 2.5]$ (expanding interval with increasing step size)  
\item \textbf{Regime 3}: Fixed $T = 2.0$, $N_1 \in [10, 65]$ with $\Delta = T/N_1$ (fixed sampling interval with increasing density)
\end{itemize}

We generate eigenvalues $\lambda_n = -n^2$ and constant amplitudes $y_n = 1.0$ for all $n$. The number of recovered modes is set to $M = \lfloor \eta N_1 \rfloor$ with $\eta = 0.5$, and noise intensities $\varepsilon$ vary by regime: $\varepsilon = 10^{-6}$ for Regimes 1 and 3, $\varepsilon = 10^{-1}$ for Regime 2.

For each regime, we compute condition numbers $\kappa$ for both eigenvalues $\lambda$ and amplitudes $y$ using both analytic estimates and the implemented Prony algorithm. The data is fitted to exponential relationships $\kappa \sim \exp(- c \cdot N_1^\alpha)$ (Regimes 1 and 3) or $\kappa \sim \exp(- c \cdot \Delta^\alpha)$ (Regime 2) for some positive $c$, by plotting $\ln(-\ln(\kappa))$ versus $\ln(N_1)$ or $\ln(\Delta)$ respectively to extract decay exponents. The results are shown in Figure~\ref{fig:condition_number_analysis_3x2}. Each row corresponds to a different regime, with the left column showing log-log vs log plots for asymptotic analysis using $\ln(-\ln(\kappa))$ (yielding negative slopes for decay) and the right column displaying semi-log plots of the condition numbers. The computed slopes closely match the analytic decay rates, confirming that Prony's method achieves the predicted asymptotic performance in Theorems~\ref{Thm:FirstOrdCond},~\ref{thm:unifbdd11}, and~\ref{thm:PronyAmplitudeCondition}. The plots practically coincide, demonstrating excellent agreement with predictions and optimality of the asymptotic dependence of the condition numbers on $N_1$ or $\Delta$. The $\eta=1$ curves (shown only on the right column) confirm the $O(1)$ behavior of condition numbers for complete spectral recovery, validating the remark after Theorem~\ref{Thm:FirstOrdCond}.

\begin{figure}[htbp]
    \centering
    \includegraphics[width=\textwidth]{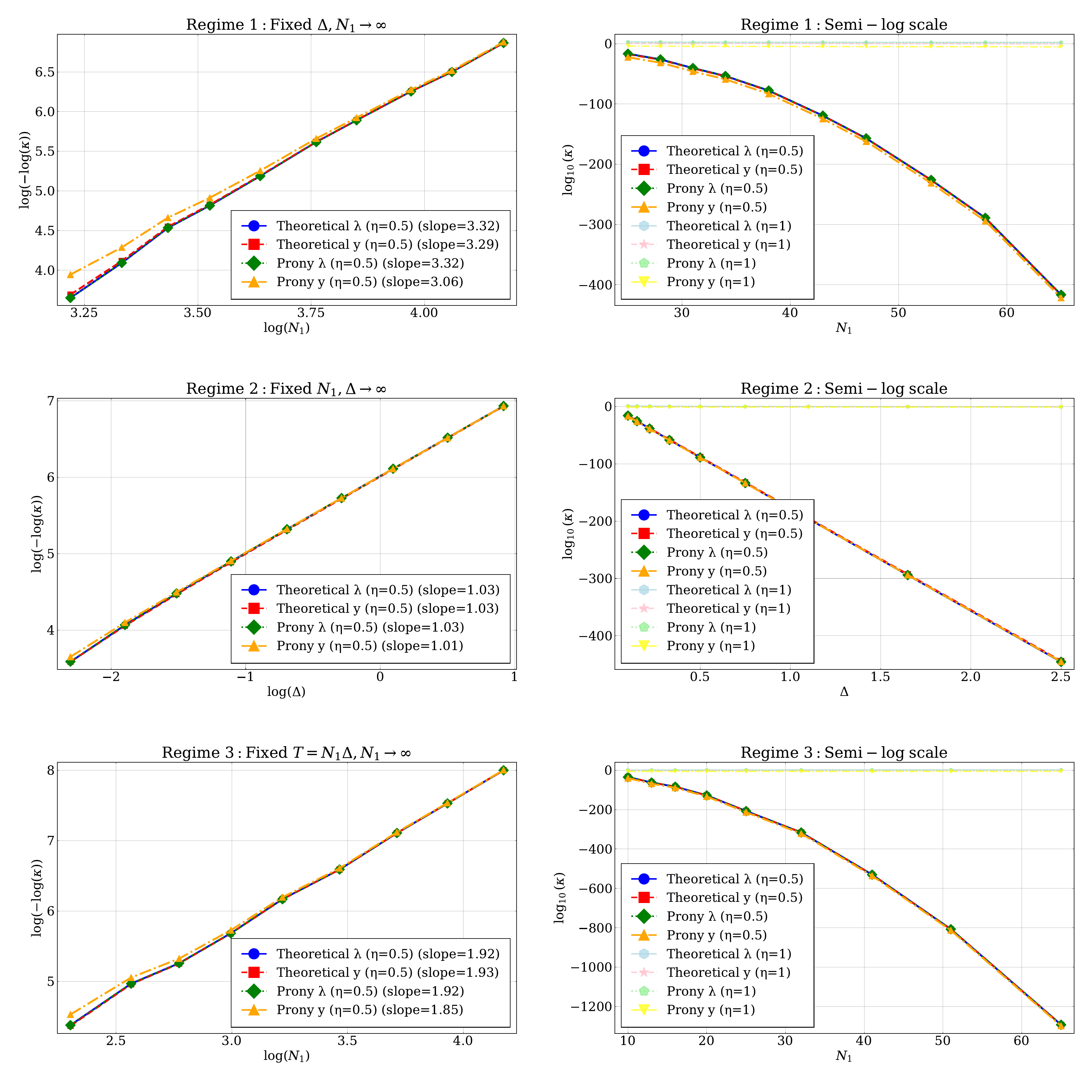}
    \caption{Comprehensive condition number analysis across the three considered asymptotic regimes. Each row represents a different regime: (1) Fixed $\Delta$, $N_1 \to \infty$; (2) Fixed $N_1$, $\Delta \to \infty$; (3) Fixed $T = N_1\Delta$, $N_1 \to \infty$. Left column shows log-log vs log plots using $\ln(-\ln(\kappa))$ for asymptotic analysis (with slopes shown in legends), right column shows semi-log plots of condition numbers. The fitted slopes validate that Prony's method achieves the analytic decay rates predicted by Theorems~\ref{Thm:FirstOrdCond},~\ref{thm:unifbdd11}, and~\ref{thm:PronyAmplitudeCondition}. Additionally, we include results for $\eta=1$ (full mode recovery) shown only on the right column, confirming the $O(1)$ scaling of condition numbers for eigenvalues $\lambda_n$ and amplitudes $y_n$ as discussed in the paragraph following Theorem~\ref{Thm:FirstOrdCond}.}
    \label{fig:condition_number_analysis_3x2} 
\end{figure}

We further investigated the influence of the eigenvalue decay rate on the condition numbers. We consider the linear ($\lambda_n \sim -n$) and the cubic ($\lambda_n \sim -n^3$) decay rates, and compare the resulting condition numbers to the ones obtained with the quadratic decay rate, as shown in Figure~\ref{fig:condition_number_analysis_decay_rate}. As expected from the ``back-of-the-envelope'' analysis in \prettyref{sec:related-work}, the slower-than-quadratic rates cannot guarantee recovery, while increasing the rate beyond quadratic leads to faster condition number decays.

\begin{figure}[htbp]
    \centering
    \includegraphics[width=\textwidth]{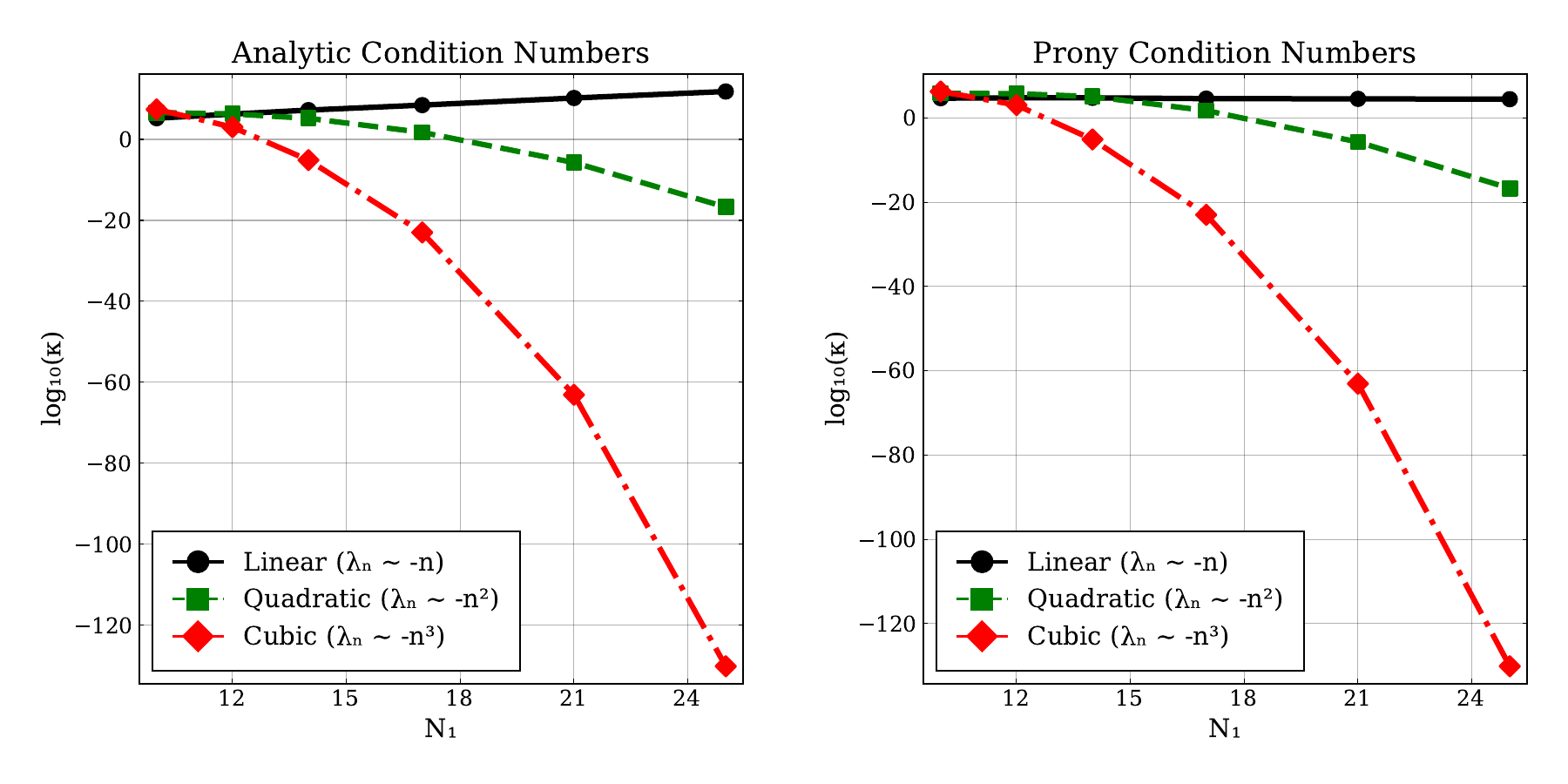}
    \caption{Comparison of condition number decay for different eigenvalue decay rates.}
    \label{fig:condition_number_analysis_decay_rate}
\end{figure}

\subsection{Potential recovery from integral measurements}\label{sec:numerics-integral-measurements}

Here, we demonstrate the application of our Prony-based estimation framework to recover an unknown potential in a linear reaction-diffusion equation from sampled integral measurements. These simulations extend our earlier reported results in \cite{katz2024}. In these experiments, we used the \texttt{mpmath} library in Python for high-precision arithmetic, with 100 digits of precision.

\subsubsection{Problem setup and data generation}
We consider the initial-boundary value problem for the linear reaction-diffusion equation \eqref{eq:diffusion-pde} with $\mathcal{A}$ as in \eqref{eq:sl-operator}, now given by $\mathcal{A} = -\partial_{xx} - q(x)$:
\begin{equation}\label{eq:PDE_potential}
z_t(x,t) = z_{xx}(x,t) + q(x)z(x,t), \quad z(0,t) = z(1,t) = 0, \quad (x,t) \in (0,1) \times \mathbb{R}^+, \;z(x,0)=f(x),
\end{equation}
where  $q(x)$ is an unknown potential to be recovered. We consider integral measurements of the form
\begin{equation}\label{eq:integral_measurement}
y(t) = \int_0^1 c(x)z(x,t)\,dx,
\end{equation}
sampled at times $t_k = k\Delta$. The measurement kernel $c(x)\in L^2(0,1)$ is constructed as a finite sum of sine functions $c(x) = \sum_{k=1}^{M} c_k \sin(\pi k x)$ with random coefficients in the range $[1, 2]$. Notice that since the eigenfunctions of $\mathcal{A}$ are different from $\sin(\pi k x)$ unless $q=\text{const}$, in general $c(x)$ may not be bandlimited (i.e. $c(x)$ may not have a finite expansion in the eigenfunctions of $\mathcal{A}$). Still, due to the smoothness of $c(x)$, there is a fast decay of $\left|\left\langle c, \psi_n\right\rangle\right|$, and therefore the model \eqref{eq:MeasExpsum} approximately holds (recall the discussion in \prettyref{sec:related-work}). We note that the special case of constant $q(x)$ was treated in \cite{katz2024}.

The initial condition is set to $f(x) = \sum_{k=1}^{60} \frac{(-1)^{k+1}}{k^3} \sin(\pi k x)$, which provides a smooth initial state with rapidly decaying Fourier coefficients.


To generate the data, first the PDE solution is computed via the method of lines using a high-order finite difference scheme with Chebyshev collocation for spatial discretization. The samples of $y(t)$ are subsequently computed using adaptive quadrature via \texttt{mpmath.quad}, with linear  interpolation of the solution to enable adaptive integration. The data generation is summarized in \prettyref{alg:data-generation} below.

\begin{algorithm2e}[H]
\KwIn{Potential $q(x)$, initial condition $f(x)$, measurement kernel $c(x)$, time span $[0, T]$, spatial resolution $N_x$, temporal resolution $N_t$, number of samples $N_s$}
\KwOut{Sampled measurements $\{y(t_k)\}_{k=0}^{2N_1-1}$}

Discretize spatial domain $[0,1]$ with $N_x$ Chebyshev collocation points\;
Construct differentiation matrices for spatial derivatives\;
Form the spatial operator matrix $A = -D_{xx} - \text{diag}(q(x_i))$\;
Apply Dirichlet boundary conditions: $A_{1,:} = A_{N_x,:} = 0$, $A_{1,1} = A_{N_x,N_x} = 1$\;
Discretize time domain $[0, T]$ with $N_t$ equispaced points\;
Compute $z(x_i,t_j)$ for each $t_j$ using the  matrix exponential $z(\cdot,t_j)=e^{At_j} f(\cdot)$\;
\textbf{Compute measurements:}\\
\For{each time point $t_j$}{
    Approximate the integral $y(t_j)$ with adaptive quadrature in \texttt{mpmath.quad} and linear interpolation to compute $z(x,t_j)$ at any given $x$ using the previously computed $z$ on a grid\;
}
\textbf{Return} sampled measurements $\{y(t_k)\}_{k=0}^{2N_s-1}$.
\caption{PDE Solution and Measurement Generation}
\label{alg:data-generation}
\end{algorithm2e}

\subsubsection{Potential recovery algorithm}

Our approach consists of two main steps: (I) recovery of eigenvalues of $\mathcal{A}$ using Prony's method, and (II) recovery of the potential via the inverse Sturm-Liouville problem.

\noindent\textbf{Step I: Eigenvalue Recovery via Prony's Method} Given the sampled measurements $\{y(t_k)\}_{k=0}^{2N_s-1}$, we apply \prettyref{alg:prony_eigenvalues} to recover the leading eigenvalues $\{\lambda_n\}_{n=1}^{N_0}$. In this numerical implementation of the vanilla Prony's method \prettyref{alg:classical-prony}, we use SVD to solve the linear system, filter roots to be real, and discard roots with small amplitudes. We use $2n_{\text{prony}}$ measurements to obtain enough initial roots.

\begin{algorithm2e}[H]
\KwIn{Sampled measurements $\{y(t_k)\}_{k=0}^{2N_s-1}$, $n_{\text{prony}} \leq N_s$, threshold $\tau > 0$}
\KwOut{Recovered eigenvalues $\{\lambda_n\}_{n=1}^{N_0}$}

Construct square Hankel matrix $H \in \mathbb{R}^{n_{\text{prony}} \times n_{\text{prony}}}$: $H_{i,j} = y(t_{i+j-1})$, $i, j = 1, \ldots, n_{\text{prony}}$\;
Solve system $Hp = -y_{n_{\text{prony}}:2n_{\text{prony}}-1}$ using SVD-based least squares\;
Compute roots $\{z_j\}_{j=1}^{n_{\text{prony}}}$ of characteristic polynomial $P(z) = z^{n_{\text{prony}}} + \sum_{k=0}^{n_{\text{prony}}-1} p_k z^k$\;
\textbf{Filter roots:}\\
\quad Keep only real roots (within numerical tolerance)\;
\quad Compute the amplitudes $\{y_n\}$ and retain only roots with amplitudes above threshold $\tau$\;
Extract eigenvalues: $\lambda_j = -\ln(z_j)/\Delta$ for filtered roots\;
\textbf{Return} $\{\lambda_n\}_{n=1}^{N_{0}}$.
\caption{High-Precision Filtered Prony Algorithm for Eigenvalue Recovery}
\label{alg:prony_eigenvalues}
\end{algorithm2e}

\paragraph{Step II: Potential Recovery via Inverse Sturm-Liouville Problem}

Given the recovered eigenvalues $\{\lambda_n\}_{n=1}^{N_0}$, we solve the inverse Sturm-Liouville problem to recover the potential $q(x)$. The algorithm is based on the approach described in \cite[Section~3.14]{chadan1997a}, which has been shown to converge for symmetric potentials when the potential is small enough. We represent $q(x)$ as a finite Fourier cosine series:
\begin{equation}\label{eq:potential_series}
q(x) = \sum_{k=0}^{M-1} a_k \cos(2\pi k x),
\end{equation}
where $M$ is the number of coefficients to be determined.

The recovery algorithm uses a differentiable ODE solver to optimize the coefficients $\{a_k\}_{k=0}^{M-1}$:

\begin{algorithm2e}[H]
\caption{Potential Recovery}\label{alg:inverse_potential}
\KwIn{Recovered eigenvalues $\{\lambda_n\}_{n=1}^{N_0}$, number of coefficients $M$}
\KwOut{Potential coefficients $\{a_k\}_{k=0}^{M-1}$}

Set $N_{\text{opt}} = \min(N_0, M)$ where $M$ is the number of Fourier coefficients in the potential\;
Initialize potential coefficients $\{a_k\}_{k=0}^{N_{\text{opt}}-1}$ (random or zero)\;
\For{each eigenvalue $\lambda_j$}{
    Solve initial value problem  $y''(x) = (q(x) - \lambda_j) y(x)$, $y(0) = 0$, $y'(0) = 1$ using differentiable ODE solver \texttt{torchdiffeq}\;
    Compute $y_j(1)$ (solution value at $x = 1$)\;
}
Compute loss function: $\mathcal{L}(\{a_k\}) = \sum_{j=1}^{N_{\text{opt}}} y_j(1)^2$\;
Optimize coefficients using L-BFGS with multiple random restarts\;
Pad result with zeros if $N_{\text{opt}} < M$ to match original coefficient count\;
\textbf{Return} optimized coefficients $\{a_k\}_{k=0}^{M-1}$.
\end{algorithm2e}

The optimization is implemented using PyTorch's automatic differentiation, allowing efficient computation of gradients with respect to the potential coefficients. The algorithm leverages the theoretical convergence guarantees from \cite{chadan1997a} for symmetric potentials, ensuring reliable recovery when the potential magnitude is sufficiently small.

\subsubsection{Numerical results}

We present two experimental scenarios to demonstrate the effectiveness of our approach.

\begin{figure}[htbp]
    \centering
    \begin{subfigure}{\textwidth}
    \includegraphics[width=\textwidth]{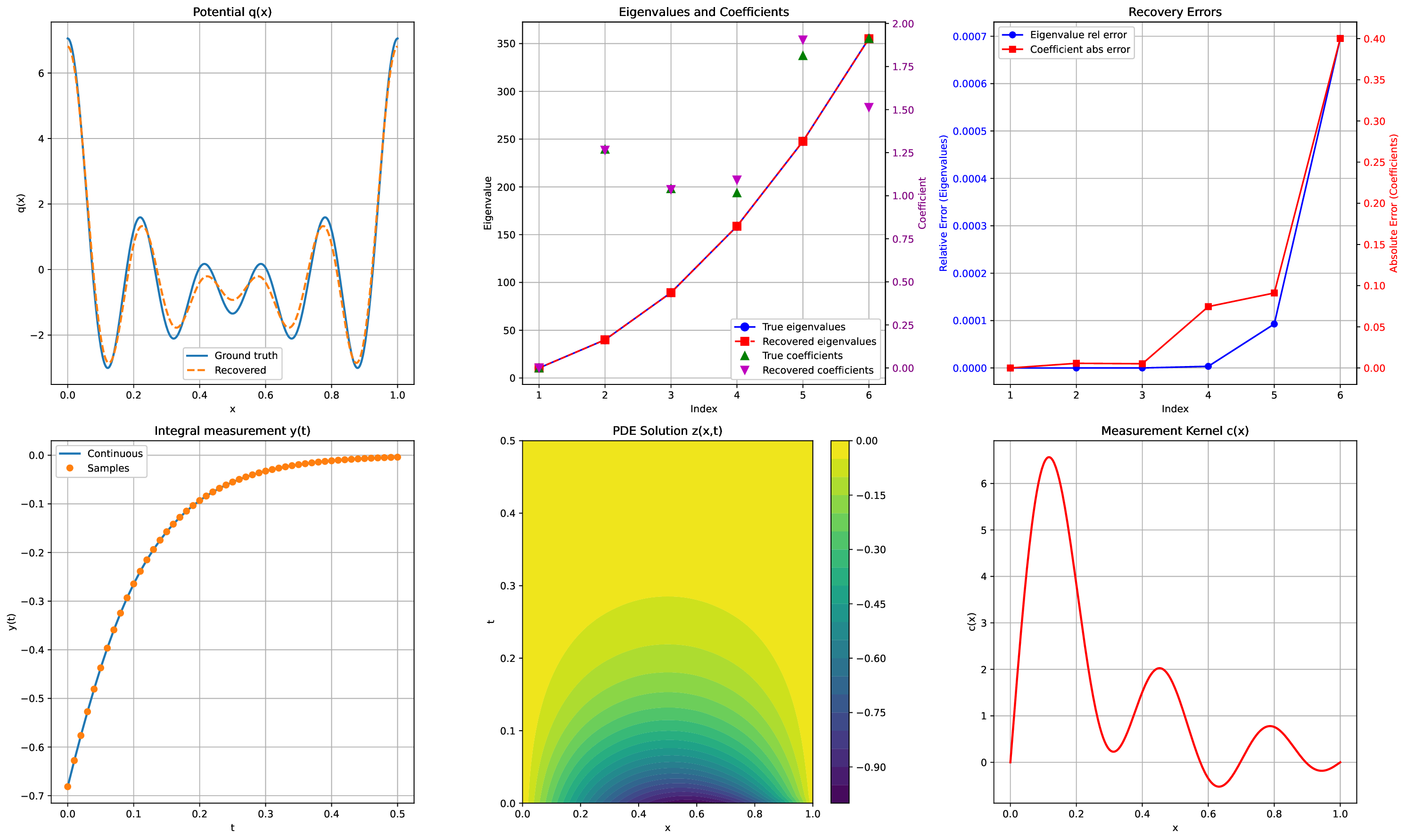}
    \caption{Random potential recovery. Top row: (left) the original and the reconstructed potential; (middle) the original and reconstructed eigenvalues and potential coefficients; (right) the corresponding recovery errors are relatively small. Bottom row: the measurement function, the computed solution and the measurement kernel.}
    \label{fig:potential_random}
    \end{subfigure}
    \hfill
    \begin{subfigure}{\textwidth}
    \includegraphics[width=\textwidth]{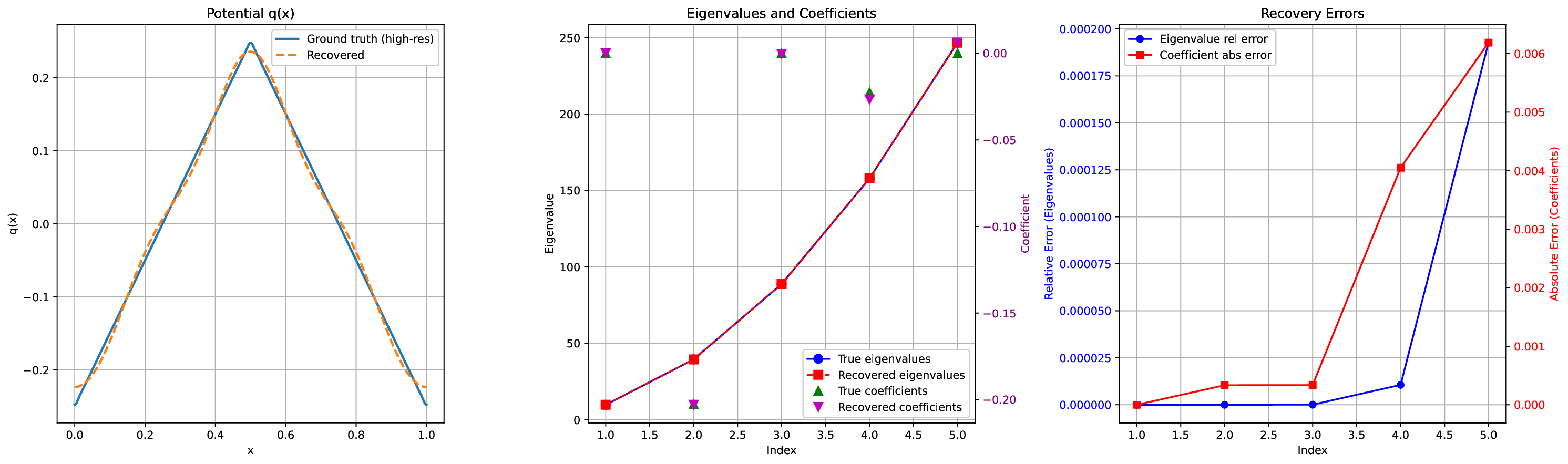}
    \caption{Triangular potential recovery. Descriptions are the same as in \prettyref{fig:potential_random}, top row. }
    \label{fig:potential_triangle}
    \end{subfigure}
    \caption{Potential recovery using Algorithms~\ref{alg:prony_eigenvalues} and~\ref{alg:inverse_potential}. Both experiments demonstrate accurate recovery of eigenvalues and potential coefficients, with the triangular potential case showing robustness to longer Fourier tails. }
\end{figure}

\noindent\textbf{Experiment 1: Random Potential}

We consider a random potential $q(x)$ with $M = 6$ Fourier coefficients:
\begin{equation}
q(x) = a_0 + \sum_{k=1}^5 a_k \cos(2\pi k x),
\end{equation}
where $a_0 = 0.0$ and $\{a_k\}_{k=1}^5$ are randomly generated coefficients in the range $[1, 2]$. The other parameters were set to $\Delta = 0.001$, $T = 0.5$, $N_x = 80$, $N_t=51$ and $n_{\text{prony}} = 35$. The threshold parameter $\tau$ was set to $\tau=10^{-6}$.

The algorithm successfully recovers both the eigenvalues and the potential coefficients with high accuracy, as shown in Figure~\ref{fig:potential_random}. The relative errors in the recovered eigenvalues are below $10^{-6}$, and the potential reconstruction closely matches the ground truth.

\paragraph{Experiment 2: Triangular Potential}

We consider a triangular potential defined by:
\begin{equation}
q(x) = 1 - |x - 0.5| - \bar{q},
\end{equation}
where $\bar{q} = \int_0^1 (1 - |x - 0.5|) \, dx = 0.75$ is the mean value, ensuring that the potential has zero mean. This function has a longer tail in its Fourier expansion compared to the random potential case. The other parameters were set to  $\Delta = 0.001$, $T=0.5$, $N_x = 60$ spatial points, $n_{\text{prony}} = 25$ and $\tau=10^{-6}$. We aim to recover $M=5$ coefficients of $q(x)$.

The algorithm performs well, accurately recovering both the eigenvalues and the triangular potential shape, as illustrated in Figure~\ref{fig:potential_triangle}. This experiment demonstrates the robustness of our approach when dealing with potentials that have significant higher-order Fourier components.

These experiments validate our theoretical framework and demonstrate the practical applicability of Prony-based methods for potential recovery in linear reaction-diffusion equations. The enhanced visualizations provide detailed comparisons between true and recovered quantities, highlighting the accuracy of our approach.

\newpage

\subsection{Point measurements and convergence study}\label{sec:potential-recovery-point-measurements}

In the next set of experiments, we replace the integral measurement \eqref{eq:integral_measurement} with a point measurement
$$
y(t) = z(x_0,t),\quad x_0=0.45.
$$
We repeat the experiments in \prettyref{sec:numerics-integral-measurements} with the obvious modification to the data generation \prettyref{alg:data-generation}, which now samples the solution $z(x,t)$ directly instead of computing quadrature. Here we use 150 decimal points for high-precision computations. The experimental parameters were $\Delta=0.01$, $T=0.5$, $N_t=51$, $N_x=60$ and $\tau=10^{-6}$.

Results of an example run with $n_{\text{prony}}=25$ and $M=4$ are shown in \prettyref{fig:potential_triangle_point}. The accuracy is comparable to that in \prettyref{fig:potential_triangle}.

\begin{figure}[htbp]
    \centering
    \includegraphics[width=\textwidth]{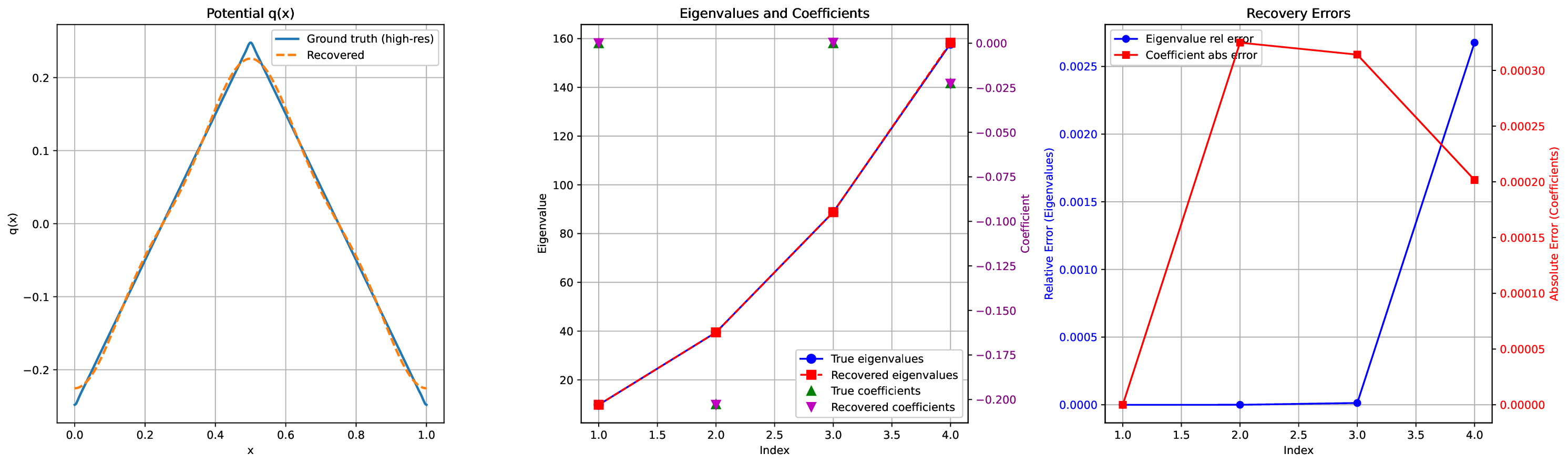}
    \caption{Triangle potential recovery using point measurements at $x_0 = 0.45$. The descriptions are the same as in \prettyref{fig:potential_random}, top row.}
    \label{fig:potential_triangle_point}
\end{figure}

\begin{figure}[htbp]
    \centering
    \includegraphics[width=\textwidth]{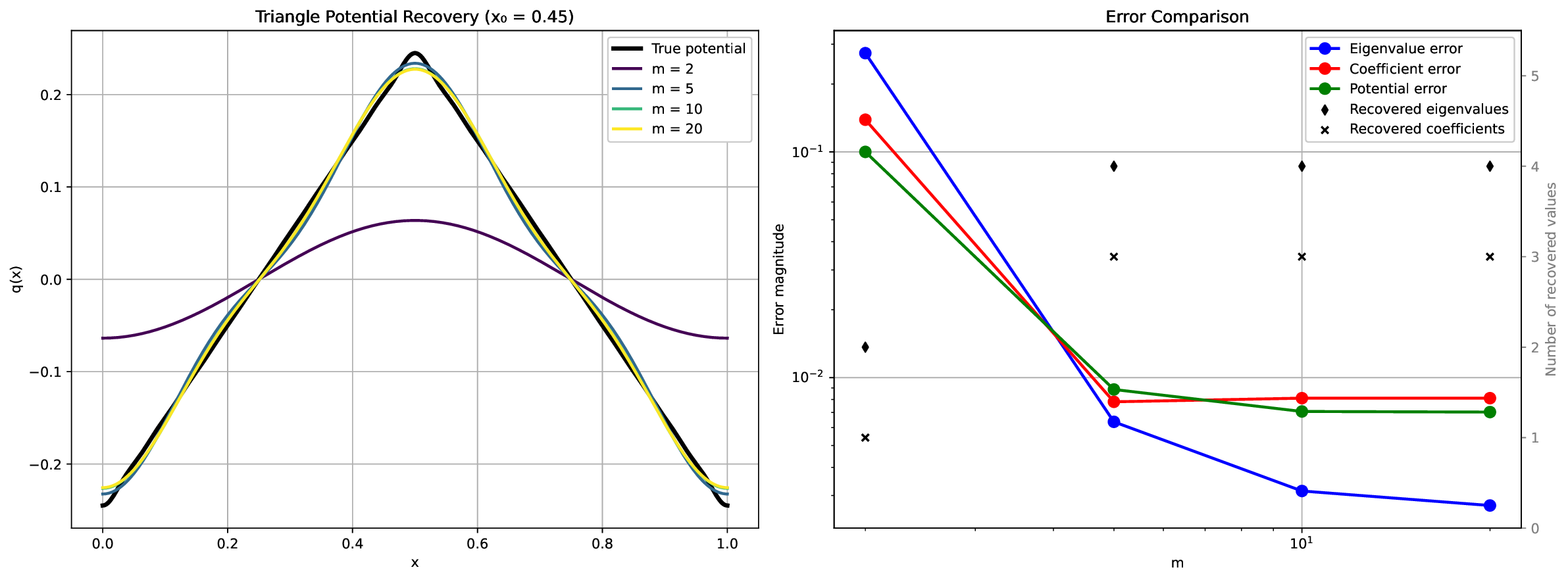}
    \caption{Convergence study for triangle potential recovery using point measurements. The left panel shows the true triangle potential (black solid line) alongside the recovered potentials for different values of $m$ (colored dashed lines). The right panel displays the convergence of three error metrics: relative error in the recovered eigenvalues (blue circles), absolute error in the recovered Fourier coefficients (red squares), $L^2$ error in the reconstructed potential function (green triangles).}
    \label{fig:triangle_convergence}
\end{figure}

To study convergence, we vary the parameter $m\equiv n_{\text{prony}}=N_s$ that controls the Prony recovery process in \prettyref{alg:prony_eigenvalues}. The experiment tests $m \in \{2, 5, 10, 20\}$ to examine how the accuracy of eigenvalue and coefficient recovery scales with the Prony parameter. Note that the actual number of recovered eigenvalues and coefficients is given by the output value $N_0$ in \prettyref{alg:prony_eigenvalues}, which is typically smaller than the input parameter $m$. The number of recovered coefficients equals the number of recovered eigenvalues, where coefficients include the constant term $q_0$ of the potential.

Figure~\ref{fig:triangle_convergence} presents the results of this convergence study. The method achieves excellent accuracy even with relatively few recovered modes ($N_0 = 4$). The algorithm maintains numerical stability across all tested values of $m$. The eigenvalue errors show a clear downward trend with increasing $m$, indicating that Prony's method successfully recovers the dominant eigenvalues even in the presence of the high-frequency components characteristic of the triangle potential. The convergence plateaus as the number of recoverable modes is reached. The $L^2$ error in potential reconstruction demonstrates that the recovered coefficients effectively capture the essential features of the triangle potential, with errors decreasing systematically as the Prony parameter $m$ increases, until the limit of recoverable modes is reached.

\newpage

\appendix
\section{Preliminaries, auxiliary results and proofs}

We provide in the appendix the proofs of our results, along with preliminaries and auxiliary results.

\subsection{Proof of Proposition~\ref{prop:EigDiffBound}}\label{sec:proof:prop:EigDiffBound}

We first prove \eqref{eq:EigDiff} for $L\leq n\leq m$ for some (large) $L\in \mathbb{N}$.  By \cite[Equation 4.21]{fulton1994eigenvalue}, the eigenvalues have the asymptotic behavior:
\begin{equation}\label{eq:LambdaAsymBehav}
\lambda_j = \frac{\pi^2}{B^2}j^2+a_0+\alpha(j),\quad j\geq 1,
\end{equation}
where $B$ and $a_0$ are \emph{positive} constants and $\alpha(j) = O\left(j^{-2} \right)$ as $j\to \infty$ (i.e., $\alpha \colon \mathbb{N} \to \mathbb{R}$ satisfies $|\alpha(j)| \leq C j^{-2}$ for some $C>0$). Hence, for $L\leq n\leq m$, we have $\lambda_m-\lambda_n= \frac{\pi^2}{B^2} \left(m^2-n^2 \right)+ \alpha(m)-\alpha(n)$. 
Taking $L\in \mathbb{N}$ large so that $L\leq n$ implies $\left|\alpha(n)\right|\leq \min \left(\frac{1}{2},\frac{\pi^2}{4B^2} \right)$ we obtain that, for $L\leq n\leq m$,
\begin{equation}\label{eq:EigDiffTemp}
\upsilon_1(m^2-n^2):=\frac{\pi^2}{2B^2}(m^2-n^2)\leq \lambda_m-\lambda_n\leq \left(\frac{\pi^2}{B^2}+1 \right)(m^2-n^2)=:\Upsilon_1 (m^2-n^2),
\end{equation}
because $\lambda_m-\lambda_n\leq \tfrac{\pi^2}{B^2}\left(m^2-n^2 \right) + 1 \leq \Upsilon_1 \left(m^2-n^2 \right)$ and $\lambda_m-\lambda_n\geq \tfrac{\pi^2}{B^2}\left(m^2-n^2 \right)-\tfrac{\pi^2}{2B^2}\geq \upsilon_1\left(m^2- n^2 \right)$.

By decreasing $\upsilon_1$ and increasing $\Upsilon_1$, since $\left\{\lambda_j\right\}_{j=1}^{\infty}$ are distinct, we can guarantee that \eqref{eq:EigDiffTemp} holds also in the finite set $1\leq n \leq m \leq L$. We also show that there exist $0<\upsilon_2<\Upsilon_2$ such that \eqref{eq:EigDiff} holds with $\upsilon,\Upsilon$ replaced with $\upsilon_2,\Upsilon_2$, respectively, and $1\leq n\leq L\leq m$. Assume by contradiction that the left inequality in \eqref{eq:EigDiff} fails. Then, for any $j\in \mathbb{N}$, there exists $1\leq n_j\leq L$ and $m_j\geq L$
such that  
\begin{equation*}
\frac{1}{j}> \frac{\lambda_{m_j}-\lambda_{n_j}}{m_j^2-n_j^2} \overset{\eqref{eq:LambdaAsymBehav}}{=} \frac{\left(\pi/B \right)^2m_j^2+a_0+\alpha(m_j)-\lambda_{n_j}}{m_j^2-n_j^2}.
\end{equation*}
By the pigeonhole principle and selection of a subsequence, we can assume, without loss of generality, that $n_j \equiv n \in [L]$ and $\lim_{j\to \infty}m_j = \infty$. Then, taking the limit for $j\to \infty$ yields the contradiction $\left(\pi/B \right)^2\leq 0$. Similar arguments yields the right inequality in \eqref{eq:EigDiff} for $1\leq n \leq L\leq m$. Letting $\upsilon=\min \left(\upsilon_1,\upsilon_2 \right)$ and $\Upsilon = \max \left(\Upsilon_1,\Upsilon_2 \right)$ gives \eqref{eq:EigDiff} for all cases of $1\leq n \leq m$.

\subsection{Auxiliary results and preliminaries for Theorems~\ref{thm:FirstOrdCondRep} and~\ref{Thm:FirstOrdCond}}\label{Sec:Preliminaries}

We begin by defining a function $\mathcal{G}_{w_1,w_2}(\zeta)$ and showing its properties, which are fundamental to derive bounds on the first-order condition numbers.

\begin{proposition}\label{prop:Ical}
    For $\zeta \in (0,\infty)$, the integral    \begin{equation}\label{eq:CalInteg}
    \begin{array}{lll}
    &\mathcal{G}_{w_1,w_2}(\zeta) := \int_{w_1}^{w_2}\mathfrak{g}(\zeta x) \mathrm{d}x>0, \quad 0\leq w_1< w_2\leq \infty, \ &\mathfrak{g}(x):=-\ln \left(1-e^{-x} \right)>0,
    \end{array}
    \end{equation}
    is finite; for $w_1=0$ and/or $w_2=\infty$, $\mathcal{G}_{w_1,w_2}(\zeta)$ is an improper integral. Furthermore, 
    $\mathcal{G}_{w_1,w_2}(\zeta)$ is decreasing in $\zeta$ and
    \begin{equation*}
    \lim_{\zeta \to \infty} \mathcal{G}_{w_1,w_2}(\zeta)=0.
    \end{equation*}
    Also, for any $x>0$, the function $\mathfrak{g}(x)$ is strictly monotonically decreasing.
    \end{proposition}
    \begin{proof}
        We consider the case $w_1=1$ and $w_2=\infty$, as the complementary case $w_1=0$ and $w_2=1$ can be handled similarly. Integrating by parts, we have
        \begin{equation*}
        \begin{array}{ll}
        & \mathcal{G}_{1,\infty}(\zeta) = \ln\left(1-e^{-
        \zeta} \right)-\lim_{x\to \infty}x\ln\left(1-e^{-\zeta x} \right)+ \int_1^{\infty}\frac{\zeta x}{e^{\zeta x}-1}\mathrm{d}x= \ln \left(1-e^{-\zeta} \right)+ \int_1^{\infty}\frac{\zeta x}{e^{\zeta x}-1}\mathrm{d}x<\infty,
        \end{array}
        \end{equation*}
        where we employed
        \begin{equation*}
        \begin{array}{lll}
        &\lim_{x\to \infty}x\ln\left(1-e^{-\zeta x} \right) = -\zeta \lim_{x\to \infty}\frac{x^2}{1-e^{-\zeta x}}e^{-\zeta x} = 0.
        \end{array}
        \end{equation*}
        Since $\mathfrak{g}'(x) = -\frac{1}{e^{x}-1}<0$ for all $x>0$, the function $\mathfrak{g}(x)$ is strictly monotonically decreasing whenever $x>0$, whence so is $\zeta \mapsto \mathcal{G}_{1,\infty}(\zeta)$, due to continuity of $\mathfrak{g}(x)$. Finally, 
        \begin{equation}\label{eq:Intlim}
        \begin{array}{lll}
        &\lim_{\zeta \to \infty}\ln \left(1-e^{-\zeta} \right)=0,\quad
        \lim_{\zeta \to \infty} \int_1^{\infty}\frac{\zeta x}{e^{\zeta x}-1} \mathrm{d}x=\lim_{\zeta\to \infty} \frac{1}{\zeta}\int_{\zeta }^{\infty}\frac{y}{e^y-1}\mathrm{d}y=0.
        \end{array}
        \end{equation}
        Hence, the limits in \eqref{eq:Intlim} yield $\lim_{\zeta \to \infty}\mathcal{G}_{1,\infty}(\zeta) = 0$.
    \end{proof}

The following result, which builds upon Proposition~\ref{prop:EigDiffBound} and Proposition~\ref{prop:Ical}, is instrumental to our asymptotic analysis.        
    \begin{proposition}\label{prop:ComponentBounds0}
    Let $N_1\in \mathbb{N}$ and consider $\left\{\phi_n\right\}_{n=1}^{N_1}$ given in \eqref{eq:phinDef}. Define
    \begin{equation*}
\mathcal{Z}^1_{n,\Delta,N_1} := \prod_{j\in [N_1]\setminus \left\{n \right\}}  (\phi_{N_1+1}-\phi_j)^2, \qquad \mathcal{Z}^2_{n,\Delta} := \prod_{j=1}^{n-1}(\phi_n-\phi_j)^2,\qquad
    \mathcal{Z}^3_{n,\Delta,N_1} := \prod_{j=n+1}^{N_1}(\phi_n-\phi_j)^2.
    \end{equation*}
    Then, we have 
    \begin{equation*}
    \begin{array}{lll}
    &\mathcal{Z}^1_{n,\Delta,N_1} = e^{-2\Delta \sum_{j\in [N_1]\setminus \left\{n \right\}}\lambda_j -2\theta^1_{n,\Delta,N_1}},\\
    & \mathcal{Z}^2_{n,\Delta} = e^{-2\Delta \sum_{j=1}^{n-1}\lambda_j - 2\theta^2_{n,\Delta} },\quad n> 1, \\
    & \mathcal{Z}^3_{n,\Delta,N_1} = e^{-2\Delta (N_1-n)\lambda_n - 2\theta^3_{n,\Delta,N_1} },\quad n< N_1,
    \end{array}
    \end{equation*}
    where, using function $\mathfrak{g}(x)$ defined in Proposition~\ref{prop:Ical} and adopting the notation that, whenever $k<l$, $\prod_{j=l}^k b_j=1$ and $\sum_{j=l}^k b_j=0$, we define 
    \begin{equation}\label{eq:ThetasSumrep}
    \begin{array}{lll}
    &\theta^1_{n,\Delta,N_1} := \sum_{j\in [N_1]\setminus \left\{ n\right\}} \mathfrak{g}(\Delta(\lambda_{N_1+1}-\lambda_j)) ,\vspace{0.1cm}\\
    & \theta^2_{n,\Delta}:=\sum_{j=1}^{n-1}\mathfrak{g}(\Delta(\lambda_n-\lambda_j)),\vspace{0.1cm}\\
    & \theta^3_{n,\Delta,N_1}:=\sum_{j=n+1}^{N_1} \mathfrak{g}(\Delta(\lambda_j-\lambda_n)),
    \end{array}
    \end{equation}
    which satisfy
\begin{equation}\label{eq:thetabounds}
    \begin{array}{lll}
    & \mathcal{G}_{1,2}(\Delta \Upsilon (2N_1+1))+ \mathfrak{g}_{n,\Delta,N_1,\Upsilon} \leq \theta^1_{n,\Delta,N_1}     \leq  \mathcal{G}_{0,\infty}(\Delta \upsilon N_1)+\mathfrak{g}_{n,\Delta,N_1,\upsilon}, \vspace{0.2cm}\\
    &\mathcal{G}_{1,2}(\Delta \Upsilon (2n-1)) \leq \theta^2_{n,\Delta}   \leq  \mathcal{G}_{0,\infty}(\Delta \upsilon (n+1))  ,\vspace{0.2cm}\\
    &\mathcal{G}_{1,2}(\Delta \Upsilon(N_1+n+1)) \leq \theta^3_{n,\Delta,N_1}  \leq \mathcal{G}_{0,\infty}(\Delta \upsilon(2n+1)),
    \end{array}
    \end{equation}
    with 
    \begin{equation*}
    \begin{array}{lll}
    &\mathfrak{g}_{n,\Delta,N_1,\Upsilon}:=-\mathfrak{g}\left(\Delta \Upsilon (N_1+1-n)(2N_1+1) \right)  \quad \mbox{and} \quad \mathfrak{g}_{n,\Delta,N_1,\upsilon}:= -\mathfrak{g}\left(\Delta  \upsilon (N_1+1-n)(N_1+1)  \right).
    \end{array}
    \end{equation*}
    \end{proposition}
    \begin{proof}

        We start with $\mathcal{Z}^1_{n,\Delta,N_1}$ and write 
        \begin{equation*}
        \begin{array}{lll}
        &\ln\left(\mathcal{Z}^1_{n,\Delta,N_1}\right) =  2\sum_{j\in [N_1]\setminus \left\{ n\right\}}\ln \left(e^{-\Delta \lambda_j} \left(1-e^{-\Delta (\lambda_{N_1+1}-\lambda_j)} \right) \right) \overset{\eqref{eq:CalInteg}}{=}
        - 2\Delta \sum_{j\in [N_1]\setminus \left\{ n\right\}}\lambda_j-2\theta^1_{n,\Delta,N_1}. 
        \end{array}
        \end{equation*}
        By Proposition~\ref{prop:Ical}, for any $\Delta>0$, the function $ \mathfrak{g}$ is strictly monotonically decreasing. Combining this with \eqref{eq:EigDiff} in Proposition~\ref{prop:EigDiffBound}, we have 
        \begin{equation}\label{eq:Theta1LowUp}
        \begin{array}{lll}
        \sum_{j\in [N_1]\setminus \left\{ n\right\}} \mathfrak{g}\left(\Delta\Upsilon ((N_1+1)^2-j^2) \right)  \leq \theta^1_{n,\Delta,N_1}\leq \sum_{j\in [N_1]\setminus \left\{ n\right\}}\mathfrak{g}\left(\Delta \upsilon ((N_1+1)^2-j^2) \right).
        \end{array}
        \end{equation}
        In order to convert the lower and upper bounds in \eqref{eq:Theta1LowUp} into Riemann sums, we note that 
        \begin{equation*}
        (N_1+1-j)(N_1+1)\leq (N_1+1)^2-j^2 = (N_1+1-j)(N_1+1+j)\leq (N_1+1-j)(2N_1+1),
        \end{equation*}
        whence 
        \begin{equation*}
        \begin{array}{lll}
        \sum_{j\in [N_1]\setminus \left\{ n\right\}}
        \mathfrak{g}\left(\Delta \Upsilon ((N_1+1)^2-j^2) \right) &\geq \sum_{j\in [N_1]\setminus \left\{ n\right\}}\mathfrak{g}\left(\Delta \Upsilon (N_1+1-j)(2N_1+1) \right)\vspace{0.1cm} \\
        &\hspace{-15mm} = \sum_{j\in [N_1]}\mathfrak{g}\left(\Delta \Upsilon (N_1+1-j)(2N_1+1) \right) - \mathfrak{g}\left(\Delta \Upsilon (N_1+1-n)(2N_1+1) \right)\vspace{0.1cm}\\
        &\hspace{-20mm} \overset{N_1+1-j=:m}{=} \sum_{m\in [N_1]}\mathfrak{g}\left(\Delta \Upsilon(2N_1+1) m \right) + \mathfrak{g}_{n,\Delta,N_1,\Upsilon}\vspace{0.1cm}\\
        &\hspace{-15mm}\geq \int_1^{N_1+1}\mathfrak{g}\left( \Delta \Upsilon (2N_1+1)x
        \right)\mathrm{d}x + \mathfrak{g}_{n,\Delta,N_1,\Upsilon}\vspace{0.1cm}\\
        &\hspace{-15mm} = \mathcal{G}_{1,N_1+1}(\Delta \Upsilon (2N_1+1)) + \mathfrak{g}_{n,\Delta,N_1,\Upsilon}\vspace{0.1cm}\\
        &\hspace{-15mm} \geq \mathcal{G}_{1,2}(\Delta \Upsilon (2N_1+1)) + \mathfrak{g}_{n,\Delta,N_1,\Upsilon}.
        \end{array}
        \end{equation*}
        Here, we have used the properties of $\mathcal{G}_{w_1,w_2}$ and $\mathfrak{g}$ in Proposition~\ref{prop:Ical}: since $\mathfrak{g}(x)$ is positive for $x>0$ and monotonically decreasing, we can lower bound the sum  $\sum_{j\in [N_1]}\mathfrak{g}\left(\Delta \Upsilon (N_1+1-j)(2N_1+1) \right)$ by an integral. Similarly,
        \begin{equation*}
        \begin{array}{lll}
        \sum_{j\in [N_1]\setminus \left\{ n\right\}}\mathfrak{g}\left(\Delta \upsilon ((N_1+1)^2-j^2) \right) &\leq \sum_{j\in [N_1]\setminus \left\{ n\right\}}\mathfrak{g}\left(\Delta \upsilon (N_1+1-j)(N_1+1) \right)\vspace{0.1cm} \\
        & \overset{N_1+1-j=:m}{=} \sum_{m\in [N_1]} \mathfrak{g}\left(\Delta  \upsilon (N_1+1) m\right)-\mathfrak{g}\left(\Delta  \upsilon (N_1+1-n)(N_1+1)\right)\vspace{0.1cm}  \\
        &\leq \int_0^{N_1}\mathfrak{g}\left(\Delta\upsilon (N_1+1) x \right)\mathrm{d}x +\mathfrak{g}_{n,\Delta,N_1,\upsilon}\vspace{0.1cm}\\
        & =  \mathcal{G}_{0,N_1}(\Delta \upsilon (N_1+1)) +\mathfrak{g}_{n,\Delta,N_1,\upsilon}\vspace{0.1cm}\\
        & \leq  \mathcal{G}_{0,\infty}(\Delta \upsilon (N_1+1))+\mathfrak{g}_{n,\Delta,N_1,\upsilon}.
        \end{array}
        \end{equation*}
The above estimates yield the bounds on $\mathcal{Z}^1_{n,\Delta,N_1}$. The bounds on $\mathcal{Z}^2_{n,\Delta}$ and $\mathcal{Z}^3_{n,\Delta,N_1}$ follow from similar arguments.    
\end{proof}
    
    \begin{corollary}\label{Cor:ThetaAssymp}
    Consider $\theta^1_{n,\Delta,N_1}$, $\theta^2_{n,\Delta}$ and $\theta^3_{n,\Delta,N_1}$ in Proposition~\ref{prop:ComponentBounds0}.
    \begin{enumerate}
        \item[\emph{(i)}] Let $\Delta_*>0$ be fixed and $N_1\to \infty$. Then, $\theta^1_{n,\Delta,N_1}$ converges to zero, whereas $\theta^2_{n,\Delta}$ and $\theta^3_{n,\Delta,N_1}$ are $O(1)$, uniformly in $(n,N_1)\in \mathcal{N}_1$ and $\Delta \in [\Delta_*,\infty)$.
        \item[\emph{(ii)}] Let $\Delta \geq \Delta_*>0$. Then, $\theta^1_{n,\Delta,N_1}$, $\theta^2_{n,\Delta}$ and $\theta^3_{n,\Delta,N_1}$ are $O(1)$, uniformly in $(n,N_1)\in \mathcal{N}_1$. If $\Delta \to \infty$, then $\theta^1_{n,\Delta,N_1}$, $\theta^2_{n,\Delta}$ and $\theta^3_{n,\Delta,N_1}$ are $o(1)$, uniformly in $(n,N_1)\in \mathcal{N}_1$.
        \item[\emph{(iii)}] Let $\Delta N_1 \equiv T$ for $T>0$. Then, $\theta^1_{n,\Delta,N_1}$ is bounded, uniformly in $(n,N_1)\in \mathcal{N}_1$.
        Moreover, given $\varphi\in (0,1)$, we have that $\theta_{n,\Delta}^2$, $n\in [\lfloor \varphi N_1 \rfloor]\setminus \left\{1 \right\}$, and $\theta^3_{n,\Delta,N_1}$, $n\in [\lfloor \varphi N_1 \rfloor]$, are $O(N_1)$, uniformly in $N_1\in \mathbb{N}$ (while $\theta_{1,\Delta}^2=0$ by definition), whereas $\theta_{n,\Delta}^2$, $n\in [N_1] \setminus [\lfloor \varphi N_1 \rfloor]$, and $\theta^3_{n,\Delta,N_1}$, $n\in [N_1] \setminus [\lfloor \varphi N_1 \rfloor]$, are $O(1)$, uniformly in $N_1\in \mathbb{N}$.

        
        
    
    \end{enumerate}
    \end{corollary}
    \begin{proof}
        \underline{Statement (i):} By \eqref{eq:thetabounds} and Proposition~\ref{prop:Ical}, we have 
        \begin{equation}\label{eq:Theta23Aux}
        \begin{array}{lll}
        &0 \leq \theta^2_{n,\Delta}   \leq  \mathcal{G}_{0,\infty}(\Delta \upsilon (n+1))\leq \mathcal{G}_{0,\infty}(2\upsilon \Delta)\leq \mathcal{G}_{0,\infty}(2\upsilon \Delta_*)  ,\vspace{0.2cm}\\
        &0 \leq \theta^3_{n,\Delta,N_1}  \leq \mathcal{G}_{0,\infty}(\Delta \upsilon(2n+1))\leq \mathcal{G}_{0,\infty}(3\upsilon \Delta)\leq \mathcal{G}_{0,\infty}(3\upsilon \Delta_*).
        \end{array}
        \end{equation}
        and hence, being bounded, $\theta^2_{n,\Delta}$ and $\theta^3_{n,\Delta,N_1}$ are $O(1)$, uniformly in $(n,N_1)\in \mathcal{N}_1$ and $\Delta \in [\Delta_*,\infty)$.
        
        For $\theta^{1}_{n,\Delta, N_1} $, we have, by \eqref{eq:thetabounds}, 
        \begin{equation}\label{eq:Theta1Aux}
        \begin{array}{lll}
        \mathcal{G}_{1,2}(\Delta \Upsilon (2N_1+1))+ \mathfrak{g}_{n,\Delta,N_1,\Upsilon} \leq \theta^1_{n,\Delta,N_1}     \leq  \mathcal{G}_{0,\infty}(\Delta \upsilon N_1)+\mathfrak{g}_{n,\Delta,N_1,\upsilon},\vspace{0.1cm}\\
        \ln\left(1-e^{-\Delta \Upsilon (2N_1+1)} \right) \leq \mathfrak{g}_{n,\Delta,N_1,\Upsilon} = \ln \left(1-e^{-\Delta \Upsilon (N_1+1-n)(2N_1+1) } \right)\leq \ln\left(1-e^{-\Delta \Upsilon N_1(2N_1+1)} \right),\\
        \ln \left(1-e^{-\Delta \upsilon (N_1+1) } \right) \leq \mathfrak{g}_{n,\Delta,N_1,\upsilon} =  \ln \left(1-e^{-\Delta \upsilon (N_1+1-n)(N_1+1) } \right)\leq \ln \left(1-e^{-\Delta \upsilon N_1(N_1+1) } \right).
        \end{array}
        \end{equation}
        Taking the limit as $N_1\to \infty$, both the upper bound and the lower bound on $\theta^1_{n,\Delta,N_1}$ converge to zero, uniformly in $n\in [N_1]$ and $\Delta \in [\Delta_*,\infty)$, and hence the same holds for $\theta^1_{n,\Delta,N_1}$.

        \underline{Statement (ii):} For $\Delta \geq \Delta_*>0$, \eqref{eq:Theta23Aux} shows that $\theta^2_{n,\Delta}$ and $\theta^3_{n,\Delta}$ are $O(1)$, uniformly in $(n,N_1)\in \mathcal{N}_1$,
        by replacing $\Delta$ with $\Delta_*$, and using the monotonicity properties obtained in Proposition~\ref{prop:Ical}. Furthermore, when $\Delta \to \infty$, $\theta^2_{n,\Delta}$ and $\theta^3_{n,\Delta}$ are $o(1)$, uniformly in $(n,N_1)\in \mathcal{N}_1$.

        For $\theta^1_{n,\Delta,N_1}$, consider \eqref{eq:Theta1Aux} and observe that
        \begin{equation*}
        \begin{array}{lll}
        0\leq \mathcal{G}_{1,2}(\Delta \Upsilon (2N_1+1)) \leq \mathcal{G}_{1,2}\left(3 \Upsilon \Delta  \right),\qquad 0\leq \mathcal{G}_{0,\infty}(\Delta \upsilon N_1) \leq \mathcal{G}_{0,\infty}\left( \upsilon \Delta  \right),
        \end{array}
        \end{equation*}
        where both terms are bounded for $\Delta\geq \Delta_*>0$ and tend to zero as $\Delta \to \infty$, uniformly in $(n,N_1)\in \mathcal{N}_1$. Similarly,
        \begin{equation*}
        \begin{array}{lll}
        &\ln\left(1-e^{-3\Delta \Upsilon } \right)\leq  \mathfrak{g}_{n,\Delta,N_1,\Upsilon} \leq 0,\qquad \ln\left(1-e^{-2\Delta \upsilon } \right)\leq  \mathfrak{g}_{n,\Delta,N_1,\upsilon} \leq 0
        \end{array}
        \end{equation*}
        shows boundedness of $\mathfrak{g}_{n,\Delta,N_1,\Upsilon}$ and $\mathfrak{g}_{n,\Delta,N_1,\upsilon}$, uniformly in $(n,N_1)\in \mathcal{N}_1$, for $\Delta\geq \Delta_*$, as well as their uniform convergence to zero as $\Delta \to \infty$. 
        
        \underline{Statement (iii):} 
        We begin with $\theta_{n,\Delta,N_1}^1$. For $\Delta N_1 \equiv T>0$, we consider \eqref{eq:thetabounds} and we observe that
        \begin{equation*}
        \begin{array}{lll}
        \mathcal{G}_{1,2}(\Delta \Upsilon (2N_1+1)) = \mathcal{G}_{1,2}\left(\Upsilon T\left(2+\frac{1}{N_1} \right) \right)\leq \mathcal{G}_{1,2}(2\Upsilon T),\quad \mathcal{G}_{0,\infty}(\Delta \upsilon N_1) = \mathcal{G}_{0,\infty}(\upsilon T),
        \end{array}
        \end{equation*}
        while
        \begin{equation*}
        \begin{array}{lll}
        &\left|\mathfrak{g}_{n,\Delta,N_1,\Upsilon} \right| = -\ln\left(1-e^{-\Upsilon\left(N_1+1-n \right)\left(2+\frac{1}{N_1} \right)} \right)\leq -\ln \left(1-e^{-2\Upsilon } \right)\\
        &\left|\mathfrak{g}_{n,\Delta,N_1,\upsilon} \right|=-\ln\left(1-e^{-\upsilon\left(N_1+1-n \right)\left(1+\frac{1}{N_1} \right)} \right)\leq -\ln \left(1-e^{2\upsilon } \right).
        \end{array}
        \end{equation*}
        Thus, $\theta^1_{n,\Delta,N_1}$ is bounded, uniformly in $(n,N_1)\in \mathcal{N}_1$. 

        For $\theta^2_{n,\Delta}$ and $\theta^3_{n,\Delta,N_1}$ we modify the argument of Proposition~\ref{prop:ComponentBounds0}. 
        Let $\varphi \in (0,1)$. Then, by \eqref{eq:thetabounds}, 
        \begin{equation}\label{eq:theta23tail}
        \begin{array}{lll}
        &n\in [N_1] \setminus [\lfloor \varphi N_1 \rfloor]\Longrightarrow 0\leq \theta^2_{n,\Delta}\leq  \mathcal{G}_{0,\infty}\left( \upsilon T \frac{\lfloor \varphi N_1 \rfloor +2}{N_1}\right) \leq  \mathcal{G}_{0,\infty}\left( \upsilon T \varphi\right),\\
        &n\in [N_1]\setminus [\lfloor \varphi N_1 \rfloor] \Longrightarrow 0 \leq \theta^3_{n,\Delta,N_1}\leq \mathcal{G}_{0,\infty} \left(\upsilon T \frac{2 \lfloor \varphi N_1 \rfloor +1  }{N_1}\right) \leq  \mathcal{G}_{0,\infty}\left( \upsilon T \varphi\right).
        \end{array}
        \end{equation}
        Since the arguments in both upper bounds in \eqref{eq:theta23tail} tend to a constant as $N_1\to \infty$, we obtain that $\theta^2_{n,\Delta}$ and $\theta^3_{n,\Delta,N_1}$ are uniformly bounded in $N_1\in \mathbb{N}$.

        For $\theta^2_{n,\Delta}$, $n\in [\lfloor  \varphi N_1 \rfloor ]\setminus \left\{1 \right\}$, we employ \eqref{eq:ThetasSumrep} to obtain
        \begin{equation}\label{eq:theta2Temp}
        \begin{array}{lll}
        0\leq \theta^2_{n,\Delta}&\leq \sum_{j=1}^{n-1}-\ln \left(1-e^{-\upsilon T\frac{(n-j)(n+j)}{N_1}} \right) \leq \sum_{j=1}^{n-1}-\ln \left(1-e^{-\upsilon T\frac{j(2n-j)}{N_1}} \right)\\
        &\leq \sum_{j=1}^{n-1}-\ln \left(1-e^{-3\upsilon T\frac{j}{N_1}} \right) \leq \sum_{j=1}^{\lfloor \varphi N_1 \rfloor-1}-\ln \left(1-e^{-3\upsilon T\frac{j}{N_1}} \right)\\
        &= N_1 \sum_{j=1}^{\lfloor \varphi N_1 \rfloor-1}-\ln \left(1-e^{-3\upsilon T\frac{j}{N_1}} \right)\frac{1}{N_1},
        \end{array}
        \end{equation}
        where the first inequality follows from Proposition~\ref{prop:EigDiffBound}, the second follows from changing the summation index (from $j$ to $n-j$), the third follows from the fact that $2n-j\geq n+1\geq 3$ and the last follows from positivity of function $\mathfrak{g}$ in \eqref{eq:CalInteg}. Now, the sum $\sum_{j=1}^{\lfloor \varphi N_1 \rfloor-1}-\ln \left(1-e^{-3\upsilon T\frac{j}{N_1}} \right)\frac{1}{N_1}$ in \eqref{eq:theta2Temp} is upper bounded by $\mathcal{G}_{0,\varphi}(3\upsilon T)<\infty$, in view of monotonicity of function $\mathfrak{g}$, which shows that $\theta^2_{n,\Delta}=O(N_1)$, uniformly in $N_1\in \mathbb{N}$. Similar arguments apply to $\theta^3_{n,\Delta,N_1}$, $n\in [\lfloor \varphi N_1 \rfloor ]$.
    \end{proof}

    The following result will also be crucial to obtain bounds on the first-order condition numbers.
    
    \begin{proposition}\label{Prop:VFrakAsymp}
        The function $\Psi(n;N_1) = \sum_{j=n+1}^{N_1}(j^2-n^2)$, $n\in [N_1]$, defined  in \eqref{eq:CalSdef}, satisfies the following properties. 
        \begin{enumerate}
            \item $\Psi(n;N_1)\geq 0$ for all $(n,N_1)\in \mathcal{N}_{\eta}$ with $\eta \in (0,1]$.
            \item For fixed $N_1$, $\Psi(n;N_1)$ is monotonically decreasing in $n$ and $\Psi(N_1;N_1)=0$.
            \item If $\eta \in (0,1)$, there exits $a(\eta)>0$ such that $\Psi(\lfloor \eta N_1 \rfloor;N_1)\geq a(\eta)N_1^3$ for $N_1\gg 1$. 
        \end{enumerate}
        \end{proposition}
        \begin{proof}
        The first two properties follow from the definition.
        For the third property, let  $N_1\gg 1$. Then, 
        \begin{equation*}
         \Psi(\lfloor \eta N_1 \rfloor; N_1) \geq \frac{N_1^3}{3}\left(1+2\eta^3 -3\eta^2 \right)+ O(N_1^2).
        \end{equation*}
        The function $a(x) =  \frac{1}{6}+\frac{1}{3}x^3-\frac{1}{2}x^2$ is strictly decreasing on $x\in [0,1]$ with $a(1) = 0$, whence $a(\eta)>0$. Therefore,
        \begin{equation*}
            \Psi(\lfloor \eta N_1 \rfloor; N_1)\geq 2a(\eta)N_1^3 + O(N_1^2)\geq a(\eta)N_1^3
        \end{equation*}
for $N_1 \gg 1$.
\end{proof}

Consider the Lagrange polynomials $L_n(z)$ in \eqref{eq:Lagrange} with nodes $\left\{\phi_n \right\}_{n=1}^{N_1}$.
We begin by estimating the asymptotic behavior of $L^2_n(\phi_{N_1+1})$.
\begin{lemma}\label{lem:Lsquared}
Consider the Lagrange polynomials $L_n$ defined in \eqref{eq:Lagrange}, with nodes $\left\{\phi_n \right\}_{n=1}^{N_1}$. In particular,
\begin{equation}\label{eq:Lsquare}
\begin{array}{lll}
L^2_n\left(\phi_{N_1+1} \right)= \begin{cases}
e^{-2\Delta \sum_{j=2}^{N_1}\left( \lambda_j-\lambda_1\right)-2\theta^1_{1,\Delta,N_1}+2\theta^3_{1,\Delta,N_1}},& n=1,\\
e^{-2\Delta \sum_{j=n+1}^{N_1}\left( \lambda_j-\lambda_n\right)-2\theta^1_{n,\Delta,N_1}+2\theta^2_{n,\Delta}+2\theta^3_{n,\Delta,N_1}},& n\in[N_1-1]\setminus\left\{1\right\},\\
e^{-2\theta^1_{N_1,\Delta,N_1}+2\theta^2_{N_1,\Delta}},& n=N_1.
\end{cases}
\end{array}
\end{equation}  
Recalling \eqref{eq:CalSdef}, for any fixed $\eta\in(0,1]$, we have the following asymptotic behaviour in the different regimes.
\begin{itemize}
    \item \underline{Regime 1:} Let $\Delta_*>0$ be fixed and $\Delta\geq \Delta_*$. 
    Then,
    \begin{equation}\label{eq:LSqrdReg1}
    \begin{array}{lll}
    \max_{n\in [\lfloor\eta N_1\rfloor]} L^2_n\left(\phi_{N_1+1} \right)\leq e^{-2\upsilon \Delta \Psi(\lfloor\eta N_1\rfloor; N_1)+O(1)},\quad N_1\gg 1
    \end{array}
    \end{equation}
    \item \underline{Regime 2:} Let $N_1$ be fixed and $\Delta\geq \Delta_*>0$.
    Then,
    \begin{equation}\label{eq:LSqrdReg2}
    \begin{array}{lll}
    \max_{ n\in [\lfloor\eta N_1\rfloor]} L^2_n\left(\phi_{N_1+1} \right)\leq e^{-2\upsilon \Delta \Psi(\lfloor\eta N_1\rfloor; N_1)+o(1)},\quad \Delta\gg 0.
    \end{array}
    \end{equation}
    
    \item \underline{Regime 3:} Let $N_1\Delta=T>0$. 
    Then, 
    \begin{equation}\label{eq:LSqrdReg3}
    \begin{array}{lll}
    \max_{ n\in [\lfloor\eta N_1\rfloor]} L^2_n\left(\phi_{N_1+1} \right)\leq e^{-2\upsilon T \frac{\Psi(\lfloor\eta N_1\rfloor; N_1) }{N_1}+O(N_1)},\quad N_1\gg 1.
    \end{array}
    \end{equation}
\end{itemize}
In all cases, for $\eta = 1$, the upper bounds imply $\max_{ n\in [ N_1]} L^2_n\left(\phi_{N_1+1} \right) = O(1)$.
\end{lemma}
\begin{proof}
We begin by deriving \eqref{eq:Lsquare}. Employing Proposition~\ref{prop:ComponentBounds0}, we have $L^2_n\left(\phi_{N_1+1} \right) = \frac{\mathcal{Z}^1_{n,\Delta,N_1}}{\mathcal{Z}^2_{n,\Delta}\mathcal{Z}^3_{n,\Delta N_1}}$, whence for $n\in [N_1-1]\setminus \left\{1\right\}$
\begin{equation}\label{eq:LegPol1Useful}
    L^2_n\left(\phi_{N_1+1} \right) = e^{-2\Delta \sum_{j=n+1}^{N_1}\left( \lambda_j-\lambda_n\right)-2\theta^1_{n,\Delta,N_1}+2\theta^2_{n,\Delta}+2\theta^3_{n,\Delta,N_1}}.
\end{equation}
For $n=1$, we obtain the same expression, but with $2\theta^2_{n,\Delta}$ omitted; for $n=N_1$, 
$L^2_n\left(\phi_{N_1+1} \right) = e^{-2\theta^1_{N_1,\Delta,N_1}+2\theta^2_{N_1,\Delta}}$.
To prove the second part of the lemma, note that for $n\in [\lfloor\eta N_1\rfloor]$, $\eta \in (0,1]$, we have 
\begin{equation}\label{eq:LegPol2Useful}
\begin{array}{lll}
&-2\Delta \sum_{j=n+1}^{N_1}\left( \lambda_j-\lambda_n\right) \overset{\text{Prop.~\ref{prop:EigDiffBound}}}{\leq } -2\upsilon \Delta  \sum_{j=n+1}^{N_1}\left(j^2-n^2 \right) \overset{\eqref{eq:CalSdef}}{=} -2\upsilon \Delta \Psi(n; N_1)  \leq  -2\upsilon \Delta \Psi(\lfloor \eta N_1 \rfloor; N_1 ),
\end{array}
\end{equation}
since $\Psi(n;N_1)$ is monotonically decreasing in $n$ as per Proposition~\ref{Prop:VFrakAsymp} and we recall that for $n=N_1$, the summation is equal to zero by convention, as stated below \eqref{eq:CalSdef}.
The inequality in \eqref{eq:LegPol2Useful}, combined with Corollary~\ref{Cor:ThetaAssymp}, yields the desired result for Regimes 1 and 2 in \eqref{Eq:Regimes}. 

For Regime 3 in \eqref{Eq:Regimes}, with $\eta\in (0,1)$, we have $\upsilon \Delta \Psi(\lfloor \eta N_1 \rfloor; N_1 ) =\upsilon T \frac{\Psi(\lfloor \eta N_1 \rfloor; N_1 )}{N_1}$.
On the other hand, in the worst-case scenario according to Corollary~\ref{Cor:ThetaAssymp}, we can estimate $\theta^1_{n,\Delta,N_1}$, $\theta^2_{n,\Delta}$ and $\theta^3_{n,\Delta,N_1}$ as $O(N_1)$, which yields \eqref{eq:LSqrdReg3}.  When $\eta=1$, we have $\Psi(N_1; N_1)=0$ by Proposition~\ref{Prop:VFrakAsymp}.
For an arbitrary fixed $\varphi \in (0,1)$, employing Corollary~\ref{Cor:ThetaAssymp} yields 
\begin{equation*}
\begin{array}{lll}
L^2_n\left(\phi_{N_1+1} \right)\leq \begin{cases}
e^{-2\upsilon T \frac{\Psi(\lfloor \varphi N_1 \rfloor; N_1)}{N_1}+O(N_1)}, & n\in [\lfloor \varphi N_1 \rfloor  ],\\
e^{O(1)}, & n\in [N_1]\setminus [\lfloor \varphi N_1 \rfloor ].
\end{cases}
\end{array}
\end{equation*}
Since $\upsilon T \frac{\Psi(\lfloor \eta N_1 \rfloor; N_1 )}{N_1} \geq a(\eta)\upsilon T N_1^2 $ for $N_1\gg 1$ by Proposition~\ref{Prop:VFrakAsymp}, in both cases $\max_{ n\in [ N_1]} L^2_n\left(\phi_{N_1+1} \right) = O(1)$, $N_1\to \infty$.
\end{proof}

Given the Lagrange polynomial $L_n(z)$ in \eqref{eq:Lagrange} with nodes $\left\{\phi_n \right\}_{n=1}^{N_1}$, taking the derivative of $L_n(z)$ and substituting $z=\phi_n$ yields 
    \begin{equation}\label{eq:phinPrime}
     L_n'(\phi_n) = \sum_{k\in [N_1]\setminus \left\{n \right\}}\left( \phi_n-\phi_k\right)^{-1}.   
    \end{equation}

\subsection{Proof of Theorem~\ref{thm:FirstOrdCondRep}}\label{Append:Pf5}

Given a polynomial $q(z) = a_{2N_1-1}z^{2N_1-1}+\dots+a_1z+a_0$, we introduce the coordinate map 
\begin{equation}\label{eq:MathfrakC}
\begin{array}{lll}
& \mathcal{C}:q\mapsto \mathcal{C}_q,:= \operatorname{col}\left\{ a_j\right\}_{j=0}^{2N_1-1},
\end{array}
\end{equation}
where $\mathcal{C}_q$ stacks the polynomial coefficients.

Let $\left\{\nu_n,\tilde{\nu}_n \right\}_{n=1}^{N_1} \subset \mathbb{C}$. If the nodes $\left\{\chi_n \right\}_{n=1}^{N_1} \subset \mathbb{C}$ are distinct, then there exists a unique polynomial $q(z)$ of degree at most $2N_1-1$ such that $q(\chi_n)=\nu_n$ and $q'(\chi_n)=\tilde{\nu}_n$ for all $n\in [N_1]$. In fact, the coefficients of $q$ satisfy the Hermite interpolation condition 
\begin{equation}\label{eq:HermiteMatrix}
\mathcal{H}_{\left\{\chi_n\right\}_{n=1}^{N_1}}\mathcal{C}_q = \operatorname{col}\left\{\begin{bmatrix}
  \nu_n\\
  \tilde{\nu}_n
\end{bmatrix}\right\}_{n=1}^{N_1},\ \mathcal{H}^{\top}_{\left\{\chi_n\right\}_{n=1}^{N_1}} := \operatorname{row} \left\{\begin{bmatrix}
     1 & 0 \\
     \chi_n & 1 \\
     \vdots & \vdots \\
     \chi_n^{2N_1-1} &  (2N_1-1)\chi_n^{2N_1-2} & 
 \end{bmatrix}\right\}_{n=1}^{N_1},
\end{equation}
where $\mathcal{H}_{\left\{\chi_n\right\}_{n=1}^{N_1}}$ maps the coefficients of a polynomial to its values and the values of its derivative on the nodes $\left\{\chi_n\right\}_{n=1}^{N_1}$.
The existence and uniqueness of the solution to the Hermite interpolation problem imply the invertibility of $\mathcal{H}_{\left\{ \chi_n\right\}_{n=1}^{N_1}}$, whence $\mathcal{H}_{\left\{\chi_n \right\}_{n=1}^{N_1}}^{-1}$ maps the values $\left\{ \nu_n,\tilde \nu_n\right\}_{n=1}^{N_1}$ to the coefficients of the corresponding Hermite interpolation polynomial $q(z)$. 

We can write $q(z) = \sum_{n=1}^{N}\left(\nu_nH_n(z)+\tilde \nu_n\tilde{H}_n(z)  \right)$ as the linear combination of the Hermite interpolation basis polynomials in \eqref{eq:Hermite}, namely $H_{n}(z): = \left[ 1-2(z-\chi_{n})L_{n}'(\chi_{n})\right] L_{n}^2(z)$ and $\widetilde{H}_{n}(z):=(z-\chi_{n})L_{n}^2(z)$,
where $L_n(z)$ in \eqref{eq:Lagrange} is a polynomial of the corresponding Lagrange interpolation basis,
$L_{n}(z) :=\prod_{j\neq n}\frac{z-\chi_{j}}{\chi_n-\chi_{j}}$, $n\in [N_1]$.
The polynomials \eqref{eq:Hermite} allow one to explicitly express $\mathcal{H}_{\left\{\chi_n \right\}_{n=1}^{N_1}}^{-1}$: by choosing $\nu_j=1$, $\tilde \nu_j=0$ and $\nu_n=\tilde \nu_n=0$, $n\neq j$, we have that 
\begin{equation*}
 \operatorname{col}\left\{\begin{bmatrix}
  \nu_n\\
  \tilde{\nu}_n
\end{bmatrix}\right\}_{n=1}^{N_1} = \mathrm{e}_{2j-1}  
\end{equation*}
is a standard basis vector in $\mathbb{R}^{2N_1}$. Thus, $\mathcal{H}_{\left\{\chi_n \right\}_{n=1}^{N_1}}^{-1}\mathrm{e}_{2j-1}$ is the $(2j-1)$-th column of $\mathcal{H}_{\left\{\chi_n \right\}_{n=1}^{N_1}}^{-1}$. On the other hand, $\mathcal{H}_{\left\{\chi_n \right\}_{n=1}^{N_1}}^{-1}\mathrm{e}_{2j-1}$ is the vector of coefficients of the polynomial $q(z)$ for which $q(\chi_j)=1$, $q'(\chi_j)=0$ and $q(\chi_n)=q'(\chi_n)=0$, $n\neq j$. Thus, we have $q(z) = H_n(z)$ and
\begin{equation*}
\operatorname{col}_{2j-1}\left( \mathcal{H}_{\left\{\chi_n \right\}_{n=1}^{N_1}}^{-1}\right) = \mathcal{C}_{H_n}.
\end{equation*}
Similar arguments with $\vartheta_j=0$, $\vartheta'_j=1$ and $\vartheta_n=\vartheta'_n=0$ yield
\begin{equation*}
\operatorname{col}_{2j}\left( \mathcal{H}_{\left\{\chi_n \right\}_{n=1}^{N_1}}^{-1}\right) = \mathcal{C}_{\tilde{H}_n}.
\end{equation*}

The proof of Theorem~\ref{thm:FirstOrdCondRep}, then, relies on two key steps. First, we employ the implicit function theorem to show the existence and continuous differentiability of an approximation candidate $\hat{P}(\varepsilon)$. Then, by taking a total derivative of $\mathcal{F}\left(\hat{P}(\varepsilon),\varepsilon \right)=0$ in $\varepsilon$ and substituting $\varepsilon=0$, we relate the condition numbers $\mathcal{K}_{\lambda}(n)$ and $\mathcal{K}_y(n)$ with an appropriate interpolation problem, allowing us to deduce \eqref{eq:Deriv}.

Function $\mathcal{F}$ in \eqref{eq:ImplicitFunc} is differentiable in all the variables $(\hat{P},\varepsilon)$. Considering the Jacobian matrix $J_{\hat{P}}(P,0):=\frac{\partial \mathcal{F}}{\partial \hat{P}}(P,0)$, a straightforward calculation yields the decomposition $J_{\hat{P}}(P,0) = \mathcal{H}^{\top}_{\left\{ \phi_n\right\}_{n=1}^{N_1}} D_{\hat{P}}(P,0)$, where $\mathcal{H}^{\top}_{\left\{ \phi_n\right\}_{n=1}^{N_1}}$ is defined in \eqref{eq:HermiteMatrix} and 
\begin{equation}\label{eq:J1Def}
\begin{array}{lll}
 &D_{\hat{P}}(P,0)  := \operatorname{diag}\left\{\begin{bmatrix}
    1 & 0\\
    0 & -\Delta y_n\phi_n
 \end{bmatrix} \right\}_{n=1}^{N_1}.
\end{array}
\end{equation}
Since the nodes $\left\{ \phi_n\right\}_{n=1}^{N_1}$ are distinct, due to the strict monotonicity of the eigenvalues $\left\{\lambda_n \right\}_{n=1}^{\infty}$ (which are real, since the Sturm-Liouville operator $\mathcal{A}$ is self-adjoint) and since $\Delta y_n \phi_n \neq 0$, in view of Assumption~\ref{assump:yassump} and of \eqref{eq:phinDef}, existence and uniqueness of the solution to the Hermite interpolation problem in \eqref{eq:HermiteMatrix} guarantee that both $\mathcal{H}^{\top}_{\left\{ \phi_n\right\}_{n=1}^{N_1}}$ and $ D_{\hat{P}}(P,0)$ are invertible.
By employing the implicit function theorem \cite{tao2006analysis}, we obtain the existence of $\varepsilon_*>0$ and of the continuously differentiable functions $\hat{P}(\varepsilon)$ such that $\hat{P}(0) = P$ and  \eqref{eq:IFT} holds for all $|\varepsilon| <\varepsilon_*$.

To compute the condition numbers, we differentiate the equality $\mathcal{F}(\hat{P}(\varepsilon), \varepsilon)=0$ for $|\varepsilon|<\varepsilon_*$ with respect to $\varepsilon$ and we substitute $\varepsilon=0$, thus obtaining 
\begin{equation}\label{eq:Deriv1}
\begin{array}{lll}
&J_{\hat{P}}(P,0) \cdot \operatorname{col}\left\{
\begin{bmatrix}
\frac{\mathrm{d}\hat{y}_{n} }{\mathrm{d} \varepsilon}(0) \\ \frac{\mathrm{d}\hat{\lambda}_{n} }{\mathrm{d} \varepsilon}(0)    
\end{bmatrix}
\right\}_{n=1}^{N_1} - \operatorname{col}\left\{\sum_{m=N_1+1}^{N_1+N_2}y_m\phi_m^k\right\}_{k=0}^{2N_1-1}=0\vspace{0.1cm} \\
&\Longrightarrow \operatorname{col}\left\{
\begin{bmatrix}
\frac{\mathrm{d}\hat{y}_{n} }{\mathrm{d} \varepsilon}(0) \\ \frac{\mathrm{d}\hat{\lambda}_{n} }{\mathrm{d} \varepsilon}(0)    
\end{bmatrix}
\right\}_{n=1}^{N_1} = D_{\hat{P}}^{-1}(P,0)\left( \mathcal{H}^{\top}_{\left\{ \phi_n\right\}_{n=1}^{N_1}}\right)^{-1}\sum_{m=N_1+1}^{N_1+N_2}y_m\operatorname{col}\left\{\phi_m^k\right\}_{k=0}^{2N_1-1}
\end{array}
\end{equation}
By the properties of the Hermite interpolation in \eqref{eq:HermiteMatrix}, and specifically in view of the construction of the column vectors $\mathcal{C}_{H_n}$ and $\mathcal{C}_{\tilde{H}_n}$, we have 
\begin{equation}\label{eq:inverseH}
\left( \mathcal{H}^{\top}_{\left\{ \phi_n\right\}_{n=1}^{N_1}}\right)^{-1} = \begin{bmatrix}
\mathcal{C}_{H_{1}} & \vert & \mathcal{C}_{\widetilde{H}_{1}} & \vert & \dots & \vert &\mathcal{C}_{H_{N_1}} & \vert & \mathcal{C}_{\widetilde{H}_{N_1}}
\end{bmatrix}^{\top},
\end{equation}
where we used the notation \eqref{eq:MathfrakC}.

For the row vectors $\mathcal{C}_{H_n}^{\top}$ and $\mathcal{C}_{\tilde{H}_n}^{\top}$, we have
\begin{equation*}
\begin{array}{lll}
&\begin{bmatrix}
\mathcal{C}_{H_n}^{\top}  \\[-4pt]  \rule{10mm}{0.4pt} \\[-1pt] \mathcal{C}_{\tilde{H}_n}^{\top}
\end{bmatrix}\operatorname{col}\left\{\phi_m^k\right\}_{k=0}^{2N_1-1} \overset{\eqref{eq:MathfrakC}}{=} \begin{bmatrix} H_n(\phi_m) \\ 
\tilde{H}_{n}(\phi_m)
\end{bmatrix}
\end{array}
\end{equation*}
and $D_{\hat{P}}^{-1}(P,0)  := \operatorname{diag}\left\{\begin{bmatrix}
    1 & 0\\
    0 & -(\Delta y_n\phi_n)^{-1}
 \end{bmatrix} \right\}_{n=1}^{N_1}$.
Combining \eqref{eq:inverseH} with \eqref{eq:Deriv1}, we conclude
\begin{equation*}
\begin{array}{lll}
\operatorname{col}\left\{
\begin{bmatrix}
\mathcal{K}_{y}(n) \\ \mathcal{K}_{\lambda}(n)  
\end{bmatrix}
\right\}_{n=1}^{N_1}& = \operatorname{col}\left\{
\begin{bmatrix}
\frac{\mathrm{d}\hat{y}_{n} }{\mathrm{d} \varepsilon}(0) \\ \frac{\mathrm{d}\hat{\lambda}_{n} }{\mathrm{d} \varepsilon}(0)    
\end{bmatrix}
\right\}_{n=1}^{N_1} \\
&= D_{\hat{P}}^{-1}(P,0)\sum_{m=N_1+1}^{N_1+N_2}y_m\left( \mathcal{H}^{\top}_{\left\{ \phi_n\right\}_{n=1}^{N_1}}\right)^{-1}\operatorname{col}\left\{\phi_m^k\right\}_{k=0}^{2N_1-1}\vspace{0.1cm}\\
& =\sum_{m=N_1+1}^{N_1+N_2}y_m \operatorname{col}\left\{ 
\begin{bmatrix}
 H_{n}(\phi_m)\\
-\frac{1}{\Delta y_{n}\phi_n } \widetilde{H}_{n}(\phi_m)
\end{bmatrix}
\right\}_{n=1}^{N_1}.
\end{array}
\end{equation*}

\subsection{Proof of Theorem~\ref{Thm:FirstOrdCond}}\label{Sec:ProofsAsymptoticAnalysis}

The proof relies on the explicit representation of the condition numbers $\mathcal{K}_{\lambda}(n)$ and $\mathcal{K}_y(n)$ in \eqref{eq:Deriv}.

Just for simplicity of the presentation, we consider $N_2=1$, meaning that the structured noise term $y_{\text{tail}}$ in \eqref{eq:MeasExpsum} contains a single mode. The result also holds in the case of an arbitrary $N_2\in \mathbb{N}$, as it follows from identical, but more tedious, arguments, which involve estimating each term of the summation in \eqref{eq:Deriv} separately.

By Theorem~\ref{thm:FirstOrdCondRep},
\begin{equation}\label{eq:Kyfirst}
\begin{array}{lll}
\left|\mathcal{K}_y(n) \right| = \left|y_{N_1+1} \right| \left|H_n(\phi_{N_1+1}) \right|\overset{\eqref{eq:Hermite}}{=}\left|y_{N_1+1} \right| \left|\left[ 1-2(\phi_{N_1+1}-\phi_{n})L_{n}'(\phi_{n})\right]  \right|L_{n}^2(\phi_{N_1+1}).
\end{array}
\end{equation}
Consider the product $\left(\phi_{N_1+1}-\phi_n \right)L'_n(\phi_n)$. Substituting \eqref{eq:phinPrime}, we obtain 
\begin{equation*}
\begin{array}{lll}
\left|\left(\phi_{N_1+1}-\phi_n \right)L'_n(\phi_n)\right| \leq  \sum_{k\in [N_1]\setminus \left\{n \right\}}\left|\frac{\phi_{N_1+1}-\phi_n}{\phi_n-\phi_k}\right| = \sum_{k=1}^{n-1}\frac{\phi_n-\phi_{N_1+1}}{\phi_k-\phi_n} + \sum_{k=n+1}^{N_1}\frac{\phi_n-\phi_{N_1+1}}{\phi_n-\phi_k},
\end{array}
\end{equation*}
where if the upper summation index exceeds the lower one in one of the sums on the right-hand side, we equate the sum to zero, following the convention stated below \eqref{eq:CalSdef}. Now,
\begin{equation*}
\begin{array}{lll}
\sum_{k=1}^{n-1}\frac{\phi_n-\phi_{N_1+1}}{\phi_k-\phi_n} = \sum_{k=1}^{n-1}\frac{e^{-\Delta \lambda_n}-e^{-\Delta \lambda_{N_1+1}}}{e^{-\Delta \lambda_k}-e^{-\Delta \lambda_n}} = \sum_{k=1}^{n-1} e^{-\Delta \left(\lambda_n -\lambda_k\right)} \frac{1-e^{-\Delta \left( \lambda_{N_1+1}-\lambda_n\right)}}{1-e^{-\Delta \left(\lambda_n-\lambda_k\right)}}.
\end{array}
\end{equation*}
Since for $k\in [n-1]$ we have
\begin{equation*}
\begin{array}{lll}
&e^{-\Delta (\lambda_n-\lambda_k)}\overset{\text{Prop.~\ref{prop:EigDiffBound}}}{\leq} e^{-\Delta \upsilon (n^2-k^2)}\leq  e^{-\Delta\upsilon },\quad 
1-e^{-\Delta \left(\lambda_{N_1+1}-\lambda_n \right) }\leq 1,\quad  1-e^{-\Delta (\lambda_n-\lambda_k)}\geq 1-e^{-\Delta \upsilon}, 
\end{array}
\end{equation*}
we conclude that
\begin{equation}\label{eq:FirstSumDeriv}
\begin{array}{lll}
\sum_{k=1}^{n-1}\frac{\phi_n-\phi_{N_1+1}}{\phi_k-\phi_n} \leq (n-1)\frac{e^{-\Delta \upsilon }}{1-e^{-\Delta \upsilon}}= (n-1)\frac{1}{e^{\Delta \upsilon}-1}.
\end{array}
\end{equation}
Similarly, 
\begin{equation*}
\begin{array}{lll}
\sum_{k=n+1}^{N_1}\frac{\phi_n-\phi_{N_1+1}}{\phi_n-\phi_k} = \sum_{k=n+1}^{N_1}\frac{e^{-\Delta \lambda_n}-e^{-\Delta \lambda_{N_1+1}}}{e^{-\Delta \lambda_n}-e^{-\Delta \lambda_k}} = \sum_{k=n+1}^{N_1}\frac{1-e^{-\Delta \left(\lambda_{N_1+1}-\lambda_n\right)}}{1-e^{-\Delta \left(\lambda_k-\lambda_n\right)}}
\end{array}
\end{equation*}
and, for $k \in [N_1]\setminus[n]$, we have 
\begin{equation*}
\begin{array}{lll}
&1-e^{-\Delta \left(\lambda_{N_1+1}-\lambda_n  \right)}\leq 1,\quad 1-e^{-\Delta \left(\lambda_k-\lambda_n \right)}\overset{\text{Prop.~\ref{prop:EigDiffBound}}}{\geq} 1-e^{-\Delta \upsilon \left(k^2-n^2 \right) }\geq 1-e^{-\Delta \upsilon (2n+1)}\geq 1-e^{-3\Delta \upsilon}.
\end{array}
\end{equation*}
Therefore,
\begin{equation}\label{eq:SecondSumDeriv}
\begin{array}{lll}
\sum_{k=n+1}^{N_1}\frac{\phi_n-\phi_{N_1+1}}{\phi_n-\phi_k} \leq (N_1-n)\frac{1}{1-e^{-3\Delta \upsilon}} = (N_1-n)\frac{e^{3\Delta \upsilon}}{e^{3\Delta \upsilon}-1}.
\end{array}
\end{equation}
Combining \eqref{eq:FirstSumDeriv} and \eqref{eq:SecondSumDeriv}, we conclude that
\begin{equation}\label{eq:AuxtermDeriv}
\begin{array}{llll}
\left|\left(\phi_{N_1+1}-\phi_n \right)L'_n(\phi_n)\right| &\leq (N_1-1)\left(\frac{1}{e^{\Delta \upsilon}-1}+\frac{e^{3\Delta \upsilon}}{e^{3\Delta \upsilon}-1} \right)=:(N_1-1)\Xi(\Delta).
\end{array}
\end{equation}
Then, we consider the regimes \eqref{Eq:Regimes} separately.

\underline{Regime 1:} 
Employing \eqref{eq:AuxtermDeriv} in \eqref{eq:Kyfirst} along with Lemma ~\ref{lem:Lsquared}, we have that, for $N_1\to \infty $,
\begin{equation}\label{eq:Kyfirst1}
\begin{array}{lll}
\max_{ n \in [\lfloor\eta N_1\rfloor]}\left|\mathcal{K}_y(n) \right| &= \max_{ n \in [\lfloor\eta N_1\rfloor]}\left|y_{N_1+1} \right| \left|\left[ 1-2(\phi_{N_1+1}-\phi_{n})L_{n}'(\phi_{n})\right]  \right|L_{n}^2(\phi_{N_1+1})\vspace{0.1cm}\\
&\leq \left|y_{N_1+1} \right|\left(1+2\Xi(\Delta) \right)(N_1-1)e^{-2\upsilon \Delta \Psi(\lfloor\eta N_1\rfloor; N_1) +O(1)}.
\end{array}
\end{equation}
Employing Proposition~\ref{Prop:VFrakAsymp}, $\upsilon\Psi(\lfloor\eta N_1\rfloor; N_1)\geq \Gamma(\eta)N_1^3$ for $N_1\gg1 $ and some $\Gamma(\eta)>0$.

Then, $\max_{n \in [\lfloor \eta N_1\rfloor]}\left|\mathcal{K}_{y}(n)\right|\leq \gamma_y(\eta, \Delta)  e^{ -\Gamma(\eta)\Delta N_1^3}$ follows from \eqref{eq:Kyfirst1} with 
\begin{equation*}
\begin{array}{lll}
\gamma_y(\eta,\Delta):= M_y^{(1)}\left( 1+2\Xi(\Delta)\right) \max_{N_1\in \mathbb{N}}\left[(N_1-1)e^{-\upsilon \Delta \Psi(\lfloor\eta N_1\rfloor; N_1) +O(1)} \right]<\infty,
\end{array}
\end{equation*}
where $\left|y_{N_1+1} \right| \leq M_y^{(1)}$ in view of Assumption~\ref{assump:yassump} and the exponential in the maximum decays super exponentially as $N_1\to \infty$, due to Proposition~\ref{Prop:VFrakAsymp} and the fact that we restrict to $\eta\in (0,1)$. Similarly, by Theorem~\ref{thm:FirstOrdCondRep} and Assumption~\ref{assump:yassump}, for $N_1\gg 1$ we obtain
\begin{equation}\label{eq:Klambdafirst}
\begin{array}{lll}
\max_{ n \in [\lfloor\eta N_1\rfloor]}\left|\mathcal{K}_{\lambda}(n) \right| &= \max_{ n \in [\lfloor\eta N_1\rfloor]}\frac{\left|y_{N_1+1} \right|}{\Delta |y_{n}| \phi_n }  \left|\tilde{H}_{n}(\phi_{N_1+1})\right|\\
&\overset{\eqref{eq:Hermite}}{\leq}\max_{n \in [\lfloor\eta N_1\rfloor]}\frac{M_y}{\Delta \phi_n} \left|\phi_{N_1+1}-\phi_{n} \right|L_{n}^2(\phi_{N_1+1})\\
&\leq \frac{M_y}{\Delta}e^{-2\upsilon \Delta\Psi(\lfloor\eta N_1\rfloor; N_1)+O(1)}\leq \gamma_{\lambda}(\eta, \Delta)e^{- \Gamma(\eta)\Delta N_1^3},
\end{array}
\end{equation}
where 
\begin{equation*}
\begin{array}{lll}
&\gamma_{\lambda}\left(\eta, \Delta \right) = \frac{M_y}{\Delta}\max_{N_1\in \mathbb{N}}e^{-\upsilon \Delta \Psi(\lfloor\eta N_1\rfloor; N_1) +O(1)}
\end{array}
\end{equation*}
and we used the facts that $\frac{\left|y_{N_1+1} \right|}{|y_{n}|} \leq M_y$ by Assumption~\ref{assump:yassump}  and that $\frac{\left| \phi_{N_1+1}-\phi_n\right|}{\phi_n} = 1-e^{-\Delta (\lambda_{N_1+1}-\lambda_n)}\leq 1$.

\underline{Regime 2:} When $\Delta \geq \Delta_*>0$, \eqref{eq:AuxtermDeriv} yields
\begin{equation}\label{eq:AuxtermDeriv2}
\begin{array}{lll}
\left|\left(\phi_{N_1+1}-\phi_n \right)L'_n(\phi_n)\right| &\leq (N_1-1)\max_{\Delta \geq \Delta_*}\Xi(\Delta)=(N_1-1)\Xi(\Delta_*),
\end{array}
\end{equation}
because $\Xi(\Delta)$ is monotonically decreasing. Then, by Lemma~\ref{lem:Lsquared} and \eqref{eq:Kyfirst}, for $\Delta \gg 0 $, we have
\begin{equation}\label{eq:Kysecond1}
\begin{array}{lll}
&\max_{ n \in [\lfloor\eta N_1\rfloor]}\left|\mathcal{K}_y(n) \right| \leq \left|y_{N_1+1} \right| 2N_1\left( 1+\Xi(\Delta_*)\right) e^{-2\upsilon \Delta \Psi(\lfloor\eta N_1\rfloor; N_1) +o(1)}.
\end{array}
\end{equation}
Now, given $\vartheta>0$, let $N_1(\vartheta)\in \mathbb{N}$ such that for all $N_1\geq N_1(\vartheta)$, we have $\upsilon \Psi\left(\lfloor \eta N_1 \rfloor; N_1  \right) \geq \vartheta$, whence $e^{-\upsilon\Delta  \Psi\left(\lfloor \eta N_1 \rfloor; N_1  \right) }\leq e^{-\vartheta \Delta }$. Such $N_1(\vartheta)$ exists in view of claim 3 in Proposition~\ref{Prop:VFrakAsymp}. We can define 
\begin{equation*}
\begin{array}{lll}
\gamma_y(\eta, \vartheta) =  2M_y^{(1)}(1+\Xi(\Delta_*))\max_{N_1 \geq N_1(\vartheta)}\left(N_1e^{-\upsilon \Delta_* \Psi\left(\lfloor \eta N_1 \rfloor; N_1  \right) +O(1)} \right)<\infty,
\end{array}
\end{equation*}
where we have replaced $o(1)$ with $O(1)$, whenever $N_1\geq N_1(\vartheta)$ and $\Delta \geq \Delta_*$, in view of Corollary~\ref{Cor:ThetaAssymp}, and we have employed claim 3 in Proposition~\ref{Prop:VFrakAsymp}, which guarantees that the maximum in $\gamma_y(\eta,\vartheta)$ is finite. Therefore, from \eqref{eq:Kysecond1}, we conclude that 
\begin{equation*}
\begin{array}{lll}
&\max_{ n \in [\lfloor\eta N_1\rfloor]}\left|\mathcal{K}_y(n) \right| \leq\gamma_y(\eta,\vartheta)e^{-\vartheta \Delta},\quad \Delta \gg 0.
\end{array}
\end{equation*}
Similarly,
\begin{equation}\label{eq:Klambdafirst111}
\begin{array}{lll}
\max_{n \in [\lfloor\eta N_1\rfloor]}\left|\mathcal{K}_{\lambda}(n) \right| &= \max_{n \in [\lfloor\eta N_1\rfloor]}\frac{\left|y_{N_1+1} \right|}{\Delta |y_{n}| \phi_n }  \left|\tilde{H}_{n}(\phi_{N_1+1})\right|\\
&\overset{\eqref{eq:Hermite}}{\leq}\max_{n \in [\lfloor\eta N_1\rfloor]}\frac{M_y}{\Delta \phi_n} \left|\phi_{N_1+1}-\phi_{n} \right|L_{n}^2(\phi_{N_1+1})\\
&\leq \frac{M_y}{\Delta}e^{-2\upsilon \Delta \Psi(\lfloor\eta N_1\rfloor; N_1)+O(1)}\leq \frac{\gamma_{\lambda}(\eta, \vartheta)}{\Delta}e^{- \vartheta \Delta},
\end{array}
\end{equation}
where 
\begin{equation*}
\begin{array}{lll}
\gamma_{\lambda}(\eta,\vartheta) = M_y \max_{N_1\geq N_1(\vartheta)}e^{-\upsilon \Delta_* \Psi\left(\lfloor \eta N_1 \rfloor; N_1  \right)+O(1)}.
\end{array}
\end{equation*}

\underline{Regime 3:} From \eqref{eq:AuxtermDeriv}, with $N_1 \Delta = T$, we can obtain
\begin{equation}\label{eq:AuxtermDeriv3}
\begin{array}{lll}
\left|\left(\phi_{N_1+1}-\phi_n \right)L'_n(\phi_n)\right| & \leq (N_1-1)\left(\frac{1}{e^{\Delta \upsilon}-1}+\frac{e^{3\Delta \upsilon}}{e^{3\Delta \upsilon}-1} \right)\\
&\leq \frac{N_1(N_1-1)}{\upsilon T}\left( \max_{N_1\in \mathbb{N}}\frac{\upsilon T}{N_1\left(e^{\upsilon T N_1^{-1}}-1 \right)}+ \frac{1}{3} \max_{N_1\in \mathbb{N}} \frac{3\upsilon Te^{3\upsilon T N_1^{-1}}}{N_1\left(e^{3\upsilon T N_1^{-1}}-1 \right)} \right)\\
&=: \frac{N_1(N_1-1)A_{\upsilon}}{\upsilon T}.
\end{array}
\end{equation}
Both maxima in \eqref{eq:AuxtermDeriv3} exist, because the limits as $N_1\to \infty$ of both terms appearing under the maxima exist and are finite (precisely, they are related to the derivative of an appropriate exponential function at zero). For the first term, setting $x:= \upsilon T N_1^{-1}$ yields $\lim_{x\to 0} \frac{x}{e^x-1}=1$, while for the second term, setting $x:= 3\upsilon T N_1^{-1}$ yields $\lim_{x\to 0} \frac{x e^x}{e^x-1}=1$. Hence, $A_{\upsilon}=1+\frac{1}{3} \cdot 1=\frac{4}{3}$.

Taking into account Lemma~\ref{lem:Lsquared} and \eqref{eq:AuxtermDeriv3}, we have
\begin{equation}\label{eq:Kysecond2}
\begin{array}{lll}
\max_{n \in [\lfloor\eta N_1\rfloor]}\left|\mathcal{K}_y(n) \right| &\leq \left|y_{N_1+1} \right|
\left(1+\frac{2}{\upsilon T}A_{\upsilon} \right)N_1(N_1-1)e^{-2\upsilon T\frac{\Psi(\lfloor \eta N_1\rfloor; N_1 )}{N_1}+O(N_1)}.
\end{array}
\end{equation}
Defining
\begin{equation*}
\begin{array}{lll}
\gamma_y(\eta, T): =  M_y^{(1)}\left(1+\frac{2}{\upsilon T}A_{\upsilon} \right) \max_{N_1\in \mathbb{N}}\left(N_1(N_1-1) e^{-\upsilon T\frac{\Psi(\lfloor \eta N_1 \rfloor; N_1 )}{N_1}+O(N_1)}\right),
\end{array}
\end{equation*}
which is finite in light of claim 3 in Proposition~\ref{Prop:VFrakAsymp} and the fact that $\eta \in (0,1)$,  
we conclude that 
\begin{equation*}
\begin{array}{lll}
\max_{n \in [\lfloor\eta N_1\rfloor]}\left|\mathcal{K}_y(n) \right| &\leq \gamma_y(\eta, T) e^{-\Gamma(\eta,T)N_1^2},\quad N_1\gg 1
\end{array}
\end{equation*}
for some $\Gamma(\eta,T)>0$, since $\Psi(\lfloor \eta N_1 \rfloor;N_1)\geq a(\eta)N_1^3$ for $N_1\gg 1$ in light of claim 3 in Proposition~\ref{Prop:VFrakAsymp}. Similar arguments yield
\begin{equation*}
\begin{array}{lll}
\max_{n \in [\lfloor\eta N_1\rfloor]}\left|\mathcal{K}_{\lambda}(n) \right| &= \max_{n \in [\lfloor\eta N_1\rfloor]}\frac{\left|y_{N_1+1} \right|}{\Delta |y_{n}| \phi_n }  \left|\tilde{H}_{n}(\phi_{N_1+1})\right|\\
&\overset{\eqref{eq:Hermite}}{\leq}\max_{n \in [\lfloor\eta N_1\rfloor]}\frac{M_y}{\Delta \phi_n} \left|\phi_{N_1+1}-\phi_{n} \right|L_{n}^2(\phi_{N_1+1})\\&\leq \frac{M_y}{\Delta}e^{-2\upsilon \Delta  \Psi(\lfloor \eta N_1 \rfloor; N_1 )+O(N_1)} \leq \gamma_{\lambda}\left(\eta,T \right)e^{-\Gamma(\eta,T)N_1^2},\quad N_1\gg 1,
\end{array}
\end{equation*}
where $\gamma_{\lambda}\left(\eta, T \right):=\frac{M_y}{T}\max_{N_1\in \mathbb{N}}\left(N_1e^{-\upsilon T\frac{\Psi(\lfloor \eta N_1\rfloor; N_1)}{N_1}+O(N_1)}\right)$. \qedhere

\subsection{Proof of Proposition~\ref{Prop:relSepUnif}}\label{Sec:ProofPropositionPhi}

The real sequence $\left\{\phi_n\right\}_{n=1}^{\infty}$ is strictly monotonically decreasing. Hence, given $n$, the smallest difference must be either $\phi_n-\phi_{n+1}$ or $\phi_{n-1}-\phi_n$ (where for $n=1$, only the former is needed). We have 
\begin{equation*}
\begin{array}{lll}
&\frac{\phi_n-\phi_{n+1}}{\phi_n} = 1- e^{-\Delta\left(\lambda_{n+1}-\lambda_n \right)}\geq 1-e^{-\Delta \upsilon (2n+1)}\geq 1 - e^{-3\Delta_* \upsilon},\quad n\geq 1,\\
&\frac{\phi_{n-1}-\phi_n}{\phi_n} = e^{\Delta \left(\lambda_{n}-\lambda_{n-1}  \right)}- 1\geq e^{\Delta\upsilon \left(2n-1 \right) }-1\geq e^{3\Delta_* \upsilon}-1,\quad n\geq 2,
\end{array}
\end{equation*}
where we rely on the lower bound $\lambda_m-\lambda_n \geq \upsilon (m^2-n^2)$ in \eqref{eq:EigDiff}, shown in Proposition~\ref{prop:EigDiffBound}.

The claim follows by choosing, for instance, $\tau = \frac{1}{3}\left(1-e^{-3\Delta_* \upsilon} \right)$.

\subsection{The explicit discrepancy representation - Lemma~\ref{lem:PolyDiscrep}}\label{Append:Pf8}

Given $N_1\in \mathbb{N}$ and $s=0,\dots,N_1+1$, we denote by $\mathcal{Q}_{N_1+1-s}^{N_1+1}$ the set of all increasing integer sequences (tuples) in $[N_1+1]$ of length $N_1+1-s$. For $s=N_1+1$, we use the convention that $\mathcal{Q}_{0}^{N_1+1} = \emptyset$. Given any $s=0,\dots, N_1$, we adopt the following notations:
\begin{equation}\label{eq:GammaParam1}
\begin{array}{lll}
& \gamma = \left(i_1,\dots,i_{N_1+1-s} \right)\in \mathcal{Q}_{N_1+1-s}^{N_1+1}, \\
&\beta = \left(1,j_1,\dots,j_{N_1-s} \right)\in \mathcal{Q}_{N_1+1-s}^{N_1+1},\\
& \mathrm{a} = i_1-1, \qquad \mathrm{b} = (j_1 -2) + (i_1-1),\\
   &\mathrm{k}_1 = i_2-i_1, \qquad \dots \qquad \mathrm{k}_{N_1-s}=i_{N_1+1-s}-i_1,\\
   &\mathrm{l}_1 = j_2-j_1, \qquad \dots \qquad \mathrm{l}_{N_1-s-1} = j_{N_1-s}-j_1.
\end{array}
\end{equation}
We denote by $\xi^c\in \mathcal{Q}^{N_1+1}_{s}$ the complementary sequence to $\xi = (i_1,\dots,i_{N_1+1-s})\in \mathcal{Q}^{N_1+1}_{N_1+1-s}$. With a slight abuse of notation, we treat the elements in $\mathcal{Q}^{N_1+1}_{N_1+1-s}$ as both ordered tuples and sets; hence, given a tuple $\xi \in \mathcal{Q}^{N_1+1}_{N_1+1-s}$, the expression $q\in \xi$  means that the integer $q$ appears in the tuple $\xi$.

Given a matrix $L\in \mathbb{R}^{(N_1+1)\times (N_1+1)}$, we denote by $\operatorname{adj}_{N_1+1-s}(L)$ its adjugate of order $N_1+1-s$; see \cite[Section 0.8.12]{horn2012matrix}. The adjugate $\operatorname{adj}_{N_1+1-s}(L)$ is indexed (in lexicographic order) by row-column elements $(\zeta,\xi)\in \mathcal{Q}^{N_1+1}_{N_1+1-s} \times \mathcal{Q}^{N_1+1}_{N_1+1-s}$ with 
\begin{equation}\label{eq:HighOrdAdj1}
\left[\operatorname{adj}_{N_1+1-s}(L) \right]_{\zeta, \xi} = (-1)^{\sum_{q\in \zeta} q+\sum_{q\in \xi}q}\cdot \operatorname{det}\left( L(\xi^c|\zeta^c) \right),
\end{equation}
where $L(\xi^c|\zeta^c)$ is the matrix obtained from $L$ by deleting rows indexed by $\xi$ and columns indexed by $\zeta$; by convention, when $s=N_1+1$, we have $\operatorname{adj}_{N_1+1-s}(L)=\operatorname{adj}_{0}(L)= \det(L)$.

Let $\left\{ \chi_n\right\}_{n=1}^{N_1}\subset \mathbb{C}$.
For a given $k \in \{0,\dots, N_1-1\}$, we denote by $\mathfrak{S}_{N_1-k}(\chi_1,\dots,\chi_{N_1})$ the \emph{elementary symmetric polynomial} of degree $N_1-k$, given by
\begin{equation}\label{eq:SymmPolDef}
    \mathfrak{S}_{N_1-k}(\chi_1,\dots,\chi_{N_1}) = \sum_{1\leq i_1<\dots < i_{N_1-k}\leq N_1}\chi_{i_1}\dots \chi_{i_{N_1-k}}.
    \end{equation}
We require several auxiliary  propositions.
\begin{proposition}\label{Prop:Symfuncdet1}
Let $m\in \mathbb{N}$ and $\left\{\chi_j \right\}_{j=1}^m\subset \mathbb{C}$. Then, for any $k\in [m]$,
\begin{equation}
\operatorname{det}\begin{bmatrix}
1 & \chi_1 & \dots & \chi_1^{k-1} & \chi_1^{k+1} & \dots & \chi_1^m \\
\vdots & \vdots  & \ddots & \vdots & \vdots & \ddots & \vdots \\
1 & \chi_m & \ldots & \chi_m^{k-1} & \chi_m^{k+1} & \dots & \chi_m^m
\end{bmatrix} = \mathfrak{S}_{m-k}(\chi_1,\dots,\chi_m) \left(\prod_{1\leq i<j\leq m}(\chi_j-\chi_i) \right),
\end{equation}
where $\mathfrak{S}_{m-k}(\chi_1,\dots,\chi_m)$ is defined in \eqref{eq:SymmPolDef}.
\end{proposition}
\begin{proof}
 Consider the Vandermonde determinant 
 \begin{equation}\label{eq:AuxVDMExp}
 \operatorname{det}\begin{bmatrix}
 1 & w & \ldots & w^{m-1} & w^{m}  \\
1 & \chi_1 & \ldots & \chi_1^{m-1} & \chi_1^{m}  \\
\vdots & \vdots  & \ddots & \vdots & \vdots  \\
1 & \chi_m & \ldots & \chi_m^{m-1} & \chi_m^{m} 
\end{bmatrix} = \left(\prod_{j=1}^m (\chi_j-w) \right)\left(\prod_{1\leq i<j\leq m}(\chi_j-\chi_i) \right).   
 \end{equation} 
Expanding the determinant via the first row, the result follows from equating the coefficient of $w^{k}$ on both sides of \eqref{eq:AuxVDMExp}.
\end{proof}
Recall that we denote $\phi_n = e^{-\Delta \lambda_n }$, $n\in [N_1]$. We now define the matrix
\begin{equation}\label{eq:Ddef1}
    D = y_{N_1+1} \begin{bmatrix}
0 & 0 & 0 & \dots & 0 \\
1 & \phi_{N_1+1} & \phi_{N_1+1}^2 & \dots & \phi_{N_1+1}^{N_1} \\
\phi_{N_1+1} & \phi_{N_1+1}^{2} & \phi_{N_1+1}^{3} & \dots & \phi_{N_1+1}^{N_1+1} \\
\phi_{N_1+1}^{2} & \phi_{N_1+1}^{3} & \phi_{N_1+1}^{4} & \dots & \phi_{N_1+1}^{N_1+2} \\
\vdots & \vdots & \vdots & \ddots & \vdots\\
\phi_{N_1+1}^{N_1-1} & \phi_{N_1+1}^{N_1} & \phi_{N_1+1}^{N_1+1} & \dots & \phi_{N_1+1}^{2N_1-1}
\end{bmatrix},
\end{equation}
such that matrix $\mathfrak{Q}(z)$ in \eqref{eq:qbar-def} can be written as $\mathfrak{Q}(z) = \mathfrak{P}(z) + \varepsilon D$, where $\mathfrak{P}(z)$ is the matrix in \eqref{eq:HomogPronyPol}.

\begin{proposition}\label{prop:Phimat1}
Consider the matrix $D$ in \eqref{eq:Ddef1}. When $s\in [N_1+1]\setminus \left\{1 \right\}$, we have $\operatorname{adj}_{N_1+1-s}(D)=0$. Furthermore, when $s=1$, let $\gamma,\beta \in \mathcal{Q}_{N_1}^{N_1+1}$, where we assume that $1\in \beta$, as in \eqref{eq:GammaParam1}. Then,
\begin{equation}\label{eq:AdjEntryS1}
\left[\operatorname{adj}_{N_1}(D) \right]_{\gamma, \beta} = (-1)^{\sum_{q\in \gamma}q+\sum_{q\in \beta}q}y_{N_1+1}\phi_{N_1+1}^{\beta^c+\gamma^c-3}.
\end{equation}
Conversely, when $s=1$ and $1 \not\in \beta$, then $\left[\operatorname{adj}_{N_1}(D) \right]_{\gamma, \beta} =0$.
\end{proposition}
\begin{proof}
The $j$-th row of matrix $D$, with $j\geq 2$, is equal to the second row multiplied by $\phi_{N_1+1}^{j-2}$.
Employing \eqref{eq:HighOrdAdj1}, with $L=D$, we see that for any $s\in [N_1]\setminus \left\{1 \right\}$ and $\zeta,\xi\in \mathcal{Q}^{N_1+1}_{N_1+1-s}$, one has $\operatorname{card}\left(\zeta^c \right) = \operatorname{card}\left(\xi^c \right)=s$, where $\operatorname{card}(A)$ denotes the cardinality of the set $A$. Therefore, the matrix $D\left(\xi^c|\zeta^c \right)$ is of rank $1$ and dimension $s\geq 2$, and hence $\left[\operatorname{adj}_{N_1+1-s}(D) \right]_{\zeta, \xi} =0$. When $s=N_1+1$, by definition, $\operatorname{adj}_0(D) = \operatorname{det}(D)=0$.

Finally, for $s=1$, let $\gamma,\beta \in \mathcal{Q}_{N_1}^{N_1+1}$. If $1\not\in\beta$, then $D(\beta^c|\gamma^c)$ contains a row of zeros, and hence $\det(D(\beta^c|\gamma^c))=0$. If, conversely, $1\in \beta$, by \eqref{eq:HighOrdAdj1} with $L=D$, $\zeta = \gamma$, $\xi=\beta$ and the fact that  $\operatorname{card}(\gamma)=\operatorname{card}(\beta)=N_1$, we have that $\operatorname{det}\left( D\left(\beta^c|\gamma^c \right)\right)$ is simply the $(\beta^c,\gamma^c)$ entry in $D$. This yields \eqref{eq:AdjEntryS1}.
\end{proof}

\begin{remark}
The proof of Proposition~\ref{prop:Phimat1} relies on $D$ in \eqref{eq:Ddef1} being a matrix of rank $1$, which is a consequence of the assumption that $N_2 = 1$ in \eqref{eq:MeasurementsProny}. For measurements of the form
\begin{equation*}
\begin{array}{lll}
y(t_k) &= \sum_{n=1}^{N_1}y_n\phi_n^k+\varepsilon y_{N_1+1}\phi_{N_1+1}^k+\varepsilon y_{N_1+2}\phi_{N_1+2}^k,\quad k=0,\dots, 2N_1-1,
\end{array}
\end{equation*}
whose $\varepsilon$-tail contains two elements, the corresponding matrix $D$ has the form
\begin{equation*}
D = y_{N_1+1} \begin{bmatrix}
0 & 0 & 0 & \dots & 0 \\
1 & \phi_{N_1+1} & \phi_{N_1+1}^2 & \dots & \phi_{N_1+1}^{N_1} \\
\phi_{N_1+1} & \phi_{N_1+1}^{2} & \phi_{N_1+1}^{3} & \dots & \phi_{N_1+1}^{N_1+1} \\
\phi_{N_1+1}^{2} & \phi_{N_1+1}^{3} & \phi_{N_1+1}^{4} & \dots & \phi_{N_1+1}^{N_1+2} \\
\vdots & \vdots & \vdots & \ddots & \vdots\\
\phi_{N_1+1}^{N_1-1} & \phi_{N_1+1}^{N_1} & \phi_{N_1+1}^{N_1+1} & \dots & \phi_{N_1+1}^{2N_1-1}
\end{bmatrix} +   y_{N_1+2}\begin{bmatrix}
0 & 0 & 0 & \dots & 0 \\
1 & \phi_{N_1+2} & \phi_{N_1+2}^2 & \dots & \phi_{N_1+2}^{N_1} \\
\phi_{N_1+2} & \phi_{N_1+2}^{2} & \phi_{N_1+2}^{3} & \dots & \phi_{N_1+2}^{N_1+1} \\
\phi_{N_1+2}^{2} & \phi_{N_1+2}^{3} & \phi_{N_1+2}^{4} & \dots & \phi_{N_1+2}^{N_1+2} \\
\vdots & \vdots & \vdots & \ddots & \vdots\\
\phi_{N_1+2}^{N_1-1} & \phi_{N_1+2}^{N_1} & \phi_{N_1+2}^{N_1+1} & \dots & \phi_{N_1+2}^{2N_1-1}
\end{bmatrix},
\end{equation*}
and is of rank $2$, at most. Thus, in the proof of Proposition~\ref{prop:Phimat1}, we would have that $\operatorname{adj}_{N_1+1-s}(D)=0$ only for $s\in[N_1+1]\setminus \left\{1,2 \right\}$. This fact would introduce additional terms into the presented analysis (see also Remark~\ref{rem:FullTail} below). For the general case of an $\varepsilon$-tail with $N_2$ large, all the adjugates $\operatorname{adj}_{N_1+1-s}(D)$ may be non-zero, which would lead to a combinatorial increase of terms in the analysis, as mentioned in Remark~\ref{rem:combExpl}. 
\end{remark}


We now fix $(n,N_1)\in \mathcal{N}_1$ and the noise intensity $\varepsilon>0$. Employing \cite[Lemma 2]{katz2024accuracy}, the homogenized perturbed Prony polynomial $\bar{q}(z)$ in \eqref{eq:qbar-def} can be written as
\begin{equation}\label{eq:CompundExpansion1}
\bar{q}\left(z\right) = \det(\mathfrak{Q}(z)) = \det(\mathfrak{P}(z)+\varepsilon D) = \sum_{s=0}^{N_1+1} \varepsilon^{s}\varsigma_{N_1+1-s}(z), 
\end{equation}
where 
\begin{equation}\label{eq:ThetaDefPronyErr}
\begin{array}{lll}
&\underline{s=0:} \quad \varsigma_{N_1+1-s}(z)=\varsigma_{N_1+1}(z) =\bar{p}(z),\vspace{0.1cm}\\
&\underline{s\in [N_1-1]:} \quad \varsigma_{N_1+1-s}(z) = \sum_{\gamma \in \mathcal{Q}_{N_1+1-s}^{N_1+1}}\sum_{\beta \in \mathcal{Q}^{N_1+1}_{N_1+1-s}: 1\in \beta}\left[\operatorname{adj}_{N_1+1-s}(D) \right]_{\gamma, \beta} \Gamma_{\beta,\gamma}(z), \vspace{0.1cm} \\
& \underline{s=N_1:} \quad \varsigma_{N_1+1-s}(z) = \varsigma_{1}(z) = \sum_{i=1}^{N_1+1}\left[\operatorname{adj}(D) \right]_{i, 1}z^{i-1}, \vspace{0.1cm}\\
&\underline{s=N_1+1:} \quad \varsigma_{N_1+1-s}(z) = \varsigma_{0}(z) = \operatorname{det}(D),
\end{array}
\end{equation}
and for $s\in[N_1-1]$, $\gamma,\beta \in \mathcal{Q}^{N_1+1}_{N_1+1-s}$ with $1\in \beta$, the polynomial $\Gamma_{\beta,\gamma}(z)$ in $z$ is given by
\begin{equation}\label{eq:OrigGamma}
\begin{array}{lll}
\Gamma_{\beta,\gamma}(z)= \operatorname{det}\begin{bmatrix}
z^\mathrm{a} & z^{\mathrm{a}+\mathrm{k}_1} & \dots & z^{\mathrm{a}+\mathrm{k}_{N_1-s}}\\
m_{\mathrm{b}} & m_{\mathrm{b}+\mathrm{k}_1}& \dots & m_{\mathrm{b}+\mathrm{k}_{N_1-s}}\\
\vdots & \vdots & \ddots & \vdots\\
m_{\mathrm{b}+\mathrm{l}_{N_1-s-1}} & m_{\mathrm{b}+\mathrm{l}_{N_1-s-1}+\mathrm{k}_1} & \dots & m_{\mathrm{b}+\mathrm{l}_{N_1-s-1}+\mathrm{k}_{N_1-s}}
\end{bmatrix},
\end{array}
\end{equation}
where we employ the notations \eqref{eq:GammaParam1} and 
\begin{equation*}
m_k  := \sum_{n=1}^{N_1}y_n\phi_n^k.
\end{equation*}
\begin{remark}\label{rem:FullTail}
From \eqref{eq:CompundExpansion1} and \eqref{eq:ThetaDefPronyErr}, it can be seen that considering an $\varepsilon$-tail of generic length $N_2$ would introduce higher powers of $\varepsilon$ in the expression of $\bar{q}(z)$, thereby adding higher order terms that need to be upper bounded in estimating the difference $\bar{p}(z)-\bar{q}(z)$.
\end{remark}

While the representation \eqref{eq:CompundExpansion1}-\eqref{eq:OrigGamma} was sufficient in \cite[Theorem 4]{katz2024accuracy}, here we require a more accurate presentation that is amenable to combinatorial analysis and is given in the following lemma, which is the main result of this section.

\begin{lemma}\label{lem:PolyDiscrep}
 The discrepancy $\bar{q}(z)-\bar{p}(z)$ can be written as
\begin{equation}\label{eq:CompundExpansionNew2}
\begin{array}{lll}
\bar{q}\left(z\right)-\bar{p}(z) &=\varepsilon  \sum_{\gamma \in \mathcal{Q}_{N_1}^{N_1+1}}\sum_{\beta \in \mathcal{Q}^{N_1+1}_{N_1}: 1\in \beta}\sum_{\omega \in \mathcal{Q}_{N_1-1}^{N_1}} \left\{(-1)^{\sum_{q\in \gamma}q+\sum_{q\in \beta}q}y_{N_1+1}\phi_{N_1+1}^{\beta^c+\gamma^c-3} \right.\\
&\hspace{10mm}\times \left(\prod_{s=1}^{N_1-1}y_{\omega_{s}} \right) z^{\mathrm{a}}\left(\prod_{j=1}^{N_1-1}\phi_{\omega_j}^{\mathrm{b}}\right)\left(\prod_{s=1}^{N_1-1}(\phi_{\omega_s}-z) \right)\\
&\hspace{10mm} \left. \times  \mathfrak{M}_{N_1+1-\gamma^c}(z,\phi_{\omega}) \mathfrak{N}_{N_1+1-\beta^c}(\phi_{\omega})\left(\prod_{1\leq s< t\leq N_1-1}(\phi_{\omega_t}-\phi_{\omega_s})^2 \right) \right\},
\end{array}
\end{equation}
where
\begin{equation}\label{eq:CompundExpansionNew3}
\begin{array}{lll}
&\omega = (\omega_1,\dots,\omega_{N_1-1})\in \mathcal{Q}^{N_1}_{N_1-1},\quad \phi_{\omega} = \left\{ \phi_{\omega_j}\right\}_{j=1}^{N_1-1},\\
&\mathfrak{M}_{N_1+1-\gamma^c}(z,\phi_{\omega}) = \begin{cases}
1, & \gamma^c=1,N_1+1,\\
\mathfrak{S}_{N_1+1-\gamma^c}(z,\phi_{\omega_1},\dots,\phi_{\omega_{N_1-1}}), &\text{otherwise},
\end{cases}\\
&\mathfrak{N}_{N_1+1-\beta^c}(\phi_{\omega}) = \begin{cases}
1,& \beta^c=2,N_1+1,\\
\mathfrak{S}_{N_1+1-\beta^c}(\phi_{\omega_1},\dots,\phi_{\omega_{N_1-1}}), & \text{otherwise},
\end{cases}
\end{array}
\end{equation}
and we identify the singletons $\gamma^c,\beta^c$ and $\omega^c$ with their elements. 
\end{lemma}
\begin{proof}
Considering \eqref{eq:CompundExpansion1}-\eqref{eq:OrigGamma}, we see that for all cases $s\in [N_1+1]\setminus\left\{1 \right\}$ in \eqref{eq:ThetaDefPronyErr}, the polynomials in $z$ are multiplied by coefficients of the form $\left[\operatorname{adj}_{N_1+1-s}(D)\right]_{\gamma,\beta}$ for some appropriate $\gamma,\beta \in \mathcal{Q}^{N_1+1}_{N_1+1-s}$. By Proposition~\ref{prop:Phimat1}, we have $\operatorname{adj}_{N_1+1-s}(D) =0 $ for such values of $s$, whence
\begin{equation}\label{eq:CompundExpansionNew1}
\begin{array}{lll}
\bar{q}\left(z\right)-\bar{p}(z)& = \varepsilon\sum_{\gamma \in \mathcal{Q}_{N_1}^{N_1+1}}\sum_{\beta \in \mathcal{Q}^{N_1+1}_{N_1}: 1\in \beta}\left[\operatorname{adj}_{N_1}(D) \right]_{\gamma, \beta} \Gamma_{\beta,\gamma}(z)\\
&= \varepsilon\sum_{\gamma \in \mathcal{Q}_{N_1}^{N_1+1}}\sum_{\beta \in \mathcal{Q}^{N_1+1}_{N_1}: 1\in \beta}(-1)^{\sum_{q\in \gamma}q+\sum_{q\in \beta}q}y_{N_1+1}\phi_{N_1+1}^{\beta^c+\gamma^c-3}\Gamma_{\beta,\gamma}(z).
\end{array}
\end{equation}
Based on \eqref{eq:OrigGamma}, the polynomial $\Gamma_{\beta,\gamma}(z)$ in the above expression is given by
\begin{equation}\label{eq:Gammabetagam}
 \Gamma_{\beta,\gamma}(z)= \sum_{\omega_1,\dots,\omega_{N_1-1}=1}^{N_1}\left(\prod_{\mu=1}^{N_1-1}y_{\omega_{\mu}} \right) z^{\mathrm{a}}\phi_{\omega_1}^{\mathrm{b}}\phi_{\omega_2}^{\mathrm{b}+\mathrm{l}_1}\dots \phi_{\omega_{N_1-1}}^{\mathrm{b}+\mathrm{l}_{N_1-2}}\operatorname{det} \begin{bmatrix}
1 & z^{\mathrm{k}_1} & \ldots & z^{\mathrm{k}_{N_1-2}} & z^{\mathrm{k}_{N_1-1}} \\
1 & \phi_{\omega_1}^{\mathrm{k}_1} & \ldots & \phi_{\omega_1}^{\mathrm{k}_{N_1-2}}  & \phi_{\omega_1}^{\mathrm{k}_{N_1-1}} \\
\vdots & \vdots & \ddots & \vdots & \vdots \\
1 & \phi_{\omega_{N_1-1}}^{\mathrm{k}_1}  & \ldots & \phi_{\omega_{N_1-1}}^{\mathrm{k}_{N_1-2}} & \phi_{\omega_{N_1-1}}^{\mathrm{k}_{N_1-1}}
\end{bmatrix}.
\end{equation}
The determinant in \eqref{eq:Gammabetagam} can be further expanded in light of \eqref{eq:GammaParam1} and Proposition~\ref{Prop:Symfuncdet1}, depending on the powers appearing therein. We consider the following cases:
\begin{itemize}
    \item \underline{Case 1:} Either $\gamma = (2,\dots,N_1+1)$, and hence $\gamma^c=1$, or $\gamma = (1,\dots,N_1)$, and hence $\gamma^c=N_1+1$. In both cases, we get $\mathrm{k}_{q}=q$, $q\in [N_1-1]$, in \eqref{eq:GammaParam1}.
    \item \underline{Case 2:} $\gamma = (1,2,\dots,\gamma^c-1,\gamma^c+1,\dots,N_1+1)$, with $2\leq \gamma^c\leq N_1$. Then, the values of $\mathrm{k}_q$, $q \in [N_1-1]$ in \eqref{eq:GammaParam1} are $(\mathrm{k}_1,\mathrm{k}_2,\dots,\mathrm{k}_{\gamma^c-2},\mathrm{k}_{\gamma^c-1},\dots,\mathrm{k}_{N_1-2},\mathrm{k}_{N_1-1})=(1,2,\dots,\gamma^c-2,\gamma^c,\dots,N_1-1,N_1)$.
\end{itemize}
In Case 1, we have a Vandermonde determinant
\begin{equation*}
\begin{array}{lll}
&\operatorname{det}\begin{bmatrix}
1 & z & \ldots & z^{N_1-2} & z^{N_1-1} \\
1 & \phi_{\omega_1} & \ldots & \phi_{\omega_1}^{N_1-2} & \phi_{\omega_1}^{N_1-1} \\
\vdots & \vdots & \ddots & \vdots & \vdots \\
1 & \phi_{\omega_{N_1-1}}  & \ldots & \phi_{\omega_{N_1-1}}^{N_1-2} & \phi_{\omega_{N_1-1}}^{N_1-1}
\end{bmatrix}=\left(\prod_{s=1}^{N_1-1}(\phi_{\omega_s}-z) \right) \left(\prod_{1\leq s< t\leq N_1-1}(\phi_{\omega_t}-\phi_{\omega_s}) \right).
\end{array}
\end{equation*}
In Case 2, we can apply Proposition~\ref{Prop:Symfuncdet1} with $m=N_1$ and $k=\gamma^c-1$ to obtain
\begin{equation*}
\begin{array}{lll}
\operatorname{det}
\begin{bmatrix}
1 & z & \ldots & z^{\mathrm{k}_{N_1-2}} & z^{\mathrm{k}_{N_1-1}} \\
1 & \phi_{\omega_1} & \ldots & \phi_{\omega_1}^{\mathrm{k}_{N_1-2}}  & \phi_{\omega_1}^{\mathrm{k}_{N_1-1}} \\
\vdots & \vdots & \ddots & \vdots & \vdots \\
1 & \phi_{\omega_{N_1-1}}  & \ldots & \phi_{\omega_{N_1-1}}^{\mathrm{k}_{N_1-2}} & \phi_{\omega_{N_1-1}}^{\mathrm{k}_{N_1-1}}
\end{bmatrix} &= \mathfrak{S}_{N_1+1-\gamma^c}(z,\phi_{\omega_1},\dots,\phi_{\omega_{N_1-1}}) \left(\prod_{s=1}^{N_1-1}(\phi_{\omega_s}-z) \right) \\
&\hspace{5mm}\times \left(\prod_{1\leq s< t\leq N_1-1}(\phi_{\omega_t}-\phi_{\omega_s}) \right).
\end{array}
\end{equation*}
These two cases are unified in the notation $\mathfrak{M}_{N_1+1-\gamma^c}(z,\phi_{\omega})$ in \eqref{eq:CompundExpansionNew3} and yield
\begin{equation}\label{eq:GammaExpr1}
\begin{array}{lll}
& \Gamma_{\beta,\gamma}(z)= \sum_{\omega_1,\dots,\omega_{N_1-1}=1}^{N_1}\left\{\left(\prod_{s=1}^{N_1-1}y_{\omega_{s}} \right) z^{\mathrm{a}}\phi_{\omega_1}^{\mathrm{b}}\phi_{\omega_2}^{\mathrm{b}+\mathrm{l}_1}\dots \phi_{\omega_{N_1-1}}^{\mathrm{b}+\mathrm{l}_{N_1-2}}\mathfrak{M}_{N_1+1-\gamma^c}(z,\phi_{\omega})\right.\\
 &\hspace{45mm}\left.  \times \left(\prod_{s=1}^{N_1-1}(\phi_{\omega_s}-z) \right) \left(\prod_{1\leq s< t\leq N_1-1}(\phi_{\omega_t}-\phi_{\omega_s}) \right)\right\}.
 \end{array}
\end{equation}
Consider all possible indices $\omega_1,\dots, \omega_{N_1-1}\in [N_1]$ in the summation \eqref{eq:GammaExpr1} and fix a permutation $\pi\in \mathrm{S}_{N_1-1}$, where $\mathrm{S}_{N_1-1}$ is the symmetric group on $N_1-1$ elements. Considering the action of $\pi$ on the indices $\left\{\omega_j \right\}_{j=1}^{N_1-1}$, we observe the following properties of \eqref{eq:GammaExpr1}:
\begin{enumerate}
\item If $\omega_i=\omega_j$ for some $i\neq j$, the corresponding summation term in \eqref{eq:GammaExpr1} equals zero.
\item  Terms $\left(\prod_{s=1}^{N_1-1}y_{\omega_{s}} \right)$,  $\left(\prod_{j=1}^{N_1-1}\phi_{\omega_j}^{\mathrm{b}}\right)$, $\left(\prod_{s=1}^{N_1-1}(\phi_{\omega_s}-z) \right)$ and $\mathfrak{M}_{N_1+1-\gamma^c}(z,\phi_{\omega})$ are invariant under $\pi$.
\item $\left(\prod_{1\leq s< t\leq N_1-1}(\phi_{\omega_t}-\phi_{\omega_s}) \right)$ is multiplied by $(-1)^{\operatorname{sgn}(\pi)}$ when applying $\pi$ to the indices $\left\{\omega_j \right\}_{j=1}^{N_1-1}$.
\end{enumerate}
Hence, rearranging the sum of terms in $\Gamma_{\beta,\gamma}(z)$ so that all permutations of a given set of distinct indices $\left\{\omega_j \right\}_{j=1}^{N_1-1}$ are summed together leads to
\begin{equation}\label{eq:GammaPreEnd}
\begin{array}{lll}
 \Gamma_{\beta,\gamma}(z)&= \sum_{\omega\in \mathcal{Q}_{N_1-1}^{N_1}}\left\{\left(\prod_{s=1}^{N_1-1}y_{\omega_{s}} \right) z^{\mathrm{a}}\left(\prod_{j=1}^{N_1-1}\phi_{\omega_j}^{\mathrm{b}}\right)\left(\prod_{s=1}^{N_1-1}(\phi_{\omega_s}-z) \right)\mathfrak{M}_{N_1+1-\gamma^c}(z,\phi_{\omega})\right.\\
 &\hspace{10mm}\left.  \times  \left(\prod_{1\leq s< t\leq N_1-1}(\phi_{\omega_t}-\phi_{\omega_s}) \right)\left( \sum_{\pi\in S_{N_1-1}}(-1)^{\operatorname{sgn}(\pi)}\phi_{\omega_{\pi(1)}}^{0}\phi_{\omega_{\pi(2)}}^{\mathrm{l}_1}\dots \phi_{\omega_{\pi(N_1-1)}}^{\mathrm{l}_{N_1-2}}\right)\right\}\vspace{0.1cm}\\
& = \sum_{\omega\in \mathcal{Q}_{N_1-1}^{N_1}}\left\{\left(\prod_{s=1}^{N_1-1}y_{\omega_{s}} \right) z^{\mathrm{a}}\left(\prod_{j=1}^{N_1-1}\phi_{\omega_j}^{\mathrm{b}}\right)\left(\prod_{s=1}^{N_1-1}(\phi_{\omega_s}-z) \right)\mathfrak{M}_{N_1+1-\gamma^c}(z,\phi_{\omega})\right.\\
 &\hspace{10mm}\left.  \times  \left(\prod_{1\leq s< t\leq N_1-1}(\phi_{\omega_t}-\phi_{\omega_s}) \right) \operatorname{det} \begin{bmatrix}
1 & \phi_{\omega_1}^{\mathrm{l}_1} & \ldots & \phi_{\omega_1}^{\mathrm{l}_{N_1-2}} \\
\vdots & \vdots & \ddots & \vdots \\
1 & \phi_{\omega_{N_1-1}}^{\mathrm{l}_1}  & \ldots & \phi_{\omega_{N_1-1}}^{\mathrm{l}_{N_1-1}}
\end{bmatrix}\right\}.
\end{array}
\end{equation}
Also the determinant in \eqref{eq:GammaPreEnd} can be further expanded in light of \eqref{eq:GammaParam1} and Proposition~\ref{Prop:Symfuncdet1},
distinguishing between the two cases:
\begin{itemize}
    \item \underline{Case 1:} Either $\beta = (1,3,\dots,N_1+1)$ and hence $\beta^c=2$, or $\beta = (1,2,\dots,N_1)$, and hence $\beta^c=N_1+1$. In both cases, we get $\mathrm{l}_{q}=q$, $q\in [N_1-2]$, in \eqref{eq:GammaParam1}.
    \item \underline{Case 2:} $\beta = (1,2,\dots,\beta^c-1,\beta^c+1,\dots,N_1+1)$, with $ \beta^c\in \left\{3,\dots,N_1 \right\}$.
    Then, $\mathrm{l}_q$, $q \in [N_1-2]$ in \eqref{eq:GammaParam1} are $(\mathrm{l}_1,\mathrm{l}_2,\dots,\mathrm{l}_{\beta^c-3},\mathrm{l}_{\beta^c-2},\dots,\mathrm{l}_{N_1-3},\mathrm{l}_{N_1-2})=(1,2,\dots,\beta^c-3,\beta^c-1,\dots,N_1-2,N_1-1)$.
\end{itemize}
In both cases, the determinant is equal to $\mathfrak{N}_{N_1+1-\beta^c}(\phi_{\omega})$ in \eqref{eq:CompundExpansionNew3}, whence substitution back into \eqref{eq:GammaPreEnd} and then substitution of the resulting $\Gamma_{\beta,\gamma}(z)$ into \eqref{eq:CompundExpansionNew1} finish the proof of the lemma. 
\end{proof}

\subsection{Proof of Theorem~\ref{thm:unifbdd} and of Corollary~\ref{cor:phinPrreal}}\label{Append:ProofUnifBdd}

Consider $\phi_n=e^{-\Delta\lambda_{n}}$, $n\in \mathbb{N}$, and $(n,N_1)\in \mathcal{N}_1$, i.e., here we set $\eta = 1$ in \eqref{eq:setN}. Recalling Proposition~\ref{Prop:relSepUnif}, let $\tau_n := \tau\phi_n$  and consider $z\in \partial B(\phi_n,\tau_n)\subset \mathbb{C}$, where $B(\phi_n,\tau_n)$ is a disk of radius $\tau_n$ centered at $\phi_n$. We aim to bound uniformly (in $n$ and in $z$) the ratio $\frac{\left|\bar{q}(z)-\bar{p}(z) \right|}{\left|\bar{p}(z) \right|}$ as follows: we define
\begin{equation}\label{def:RatioTerm}
\begin{array}{lll}
\max_{n\in [N_1]}\mathcal{R}_n:=\max_{n\in [N_1]}\max_{z\in \partial B(\phi_n,\tau_n)}\frac{\left|\bar{q}(z)-\bar{p}(z) \right|}{\left|\bar{p}(z) \right|}
\end{array}
\end{equation}
and we show that there exists some $\varepsilon_0>0$ such that $\max_{n\in [N_1]}\mathcal{R}_n<1$ for all $\varepsilon \in (0,\varepsilon_0]$ and all $N_1\in \mathbb{N}$, in each of the regimes in \eqref{Eq:Regimes} separately. Invoking Rouche's theorem \cite[Chapter 4]{stein2010complex}, we will then obtain the result in Theorem~\ref{thm:unifbdd}.

Employing the representation of the unperturbed Prony polynomial in \eqref{eq:HomogPronyPol}, according to \cite[Equation (15)]{katz2024accuracy}, we have 
\begin{equation}\label{eq:PbarExplicitExpr1}
\bar{p}(z) = (-1)^{N_1} \left(\prod_{k=1}^{N_1} y_k\right) \cdot \left(\prod_{1\leq s<t\leq N_1}\left(\phi_{t}-\phi_s \right)^2\right)\cdot \left(\prod_{s=1}^{N_1}(\phi_s-z)\right).
\end{equation}
By Lemma~\ref{lem:PolyDiscrep}, the term to be maximized in \eqref{def:RatioTerm} has the upper bound
\begin{equation*}
\begin{array}{lll}
&\frac{\left|\bar{q}(z)-\bar{p}(z) \right|}{\left|\bar{p}(z) \right|} \leq \varepsilon \sum_{\gamma \in \mathcal{Q}^{N_1+1}_{N_1}} \sum_{\beta \in \mathcal{Q}^{N_1+1}_{N_1}: 1\in \beta}\sum_{\omega \in \mathcal{Q}_{N_1-1}^{N_1}}\Sigma_{\gamma,\beta,\omega}(z),\vspace{0.1cm}\\
& \Sigma_{\gamma,\beta,\omega}(z): = \frac{|y_{N_1+1}|\left(\prod_{s=1}^{N_1-1}|y_{\omega_{s}}| \right)}{\left(\prod_{k=1}^{N_1} |y_k|\right)}\cdot \frac{\left(\prod_{s=1}^{N_1-1}|\phi_{\omega_s}-z| \right)}{\left(\prod_{s=1}^{N_1}|\phi_s-z| \right)}\cdot \frac{\left(\prod_{1\leq s< t\leq N_1-1}(\phi_{\omega_t}-\phi_{\omega_s})^2 \right)}{\left(\prod_{1\leq s<t\leq N_1}\left|\phi_{t}-\phi_s \right|^2\right)}\vspace{0.1cm}\\
&\hspace{16mm}\times \phi_{N_1+1}^{\beta^c+\gamma^c-3} |z|^{\mathrm{a}}\left(\prod_{j=1}^{N_1-1}\phi_{\omega_j}^{\mathrm{b}}\right)\mathfrak{M}_{N_1+1-\gamma^c}(|z|,\phi_{\omega})\mathfrak{N}_{N_1+1-\beta^c}(\phi_{\omega}).
\end{array}
\end{equation*}
Recalling that $\omega^c \in [N_1]$ denotes the complement of $\omega\in \mathcal{Q}^{N_1}_{N_1-1}$ and canceling out the elements in the first line of $\Sigma_{\gamma,\beta,\omega}(z)$, we have
\begin{equation}\label{eq:Ssum}
\begin{array}{lll}
&\frac{\left|\bar{q}(z)-\bar{p}(z) \right|}{\left|\bar{p}(z) \right|} \leq  \varepsilon  \sum_{\gamma \in \mathcal{Q}_{N_1}^{N_1+1}}\sum_{\beta \in \mathcal{Q}^{N_1+1}_{N_1}: 1\in \beta} \mathfrak{s}(z,\gamma^c,\beta^c) ,\vspace{0.1cm}\\
&\mathfrak{s}(z,\gamma^c,\beta^c) = \sum_{\omega \in \mathcal{Q}^{N_1}_{N_1-1}}\frac{ |y_{N_1+1}|}{|y_{\omega^c}|} \phi_{N_1+1}^{\beta^c+\gamma^c-3}\frac{\prod_{s\in[N_1],s\neq \omega^c}\phi_s^{\mathrm{b}}}{\prod_{s\in [N_1],s\neq \omega^c}|\phi_s-\phi_{\omega^c}|^2}\frac{|z|^{\mathrm{a}} \mathfrak{M}_{N_1+1-\gamma^c}(|z|,\phi_{\omega}) \mathfrak{N}_{N_1+1-\beta^c}(\phi_{\omega})}{|z-\phi_{\omega^c}|}.
\end{array}
\end{equation}
We proceed with analyzing the behavior of $\mathfrak{s}\left(z,\gamma^c,\beta^c\right)$ for $z\in \partial B(\phi_n,\tau_n)$.

To this aim, we begin by proving the next auxiliary proposition, which relies on the classical MacLaurin inequality \cite{maclaurin1730iv} (see also \cite{tao2023maclaurin}) and on the definition of imaginary error function $\mathcal{E}(x)$.

    \begin{lemma}[MacLaurin inequality]\label{Lem:MacIneq}
    Let $\left\{ \chi_n\right\}_{n=1}^{N_1}\subset \mathbb{R}$ be a set of non-negative real numbers. The elementary symmetric means, defined as
    \begin{equation}\label{eq:SymmMeans}
    \begin{array}{lll}
    \mathcal{M}_k(\chi_1,\dots,\chi_{N_1}) = \frac{1}{\binom{N_1}{k}}\mathfrak{S}_k(\chi_1,\dots,\chi_{N_1}),\quad k\in [N_1],
    \end{array}
    \end{equation}
    obey the inequality
    \begin{equation*}
      \left(\mathcal{M}_{\ell}(\chi_1,\dots,\chi_{N_1}) \right)^{\frac{1}{\ell}}  \leq \left(\mathcal{M}_{k}(\chi_1,\dots,\chi_{N_1}) \right)^{\frac{1}{k}}
    \end{equation*}
    for $1\leq k \leq \ell \leq N_1$. 
    \end{lemma}
    
The imaginary error function $\mathcal{E}(x)$ \cite[Chapter 7]{abramowitz1965handbook} is defined as
    \begin{equation}\label{eq:DefErfi}
    \begin{array}{lll}
    \mathcal{E}(x) := \frac{2}{\sqrt{\pi}} \int_0^x e^{t^2}\mathrm{d}t,\quad x\in [0,\infty), 
    \end{array}
    \end{equation}
    and is known to satisfy the asymptotic behavior
    \begin{equation}\label{eq:efiasymp}
    \begin{array}{lll}
    e^{-x^2}\mathcal{E}(x) = \frac{1}{\sqrt{\pi}x}+\frac{1}{2\sqrt{\pi}x^3}+O\left(\frac{1}{x^4} \right),\quad x\to \infty.
    \end{array}
    \end{equation}

\begin{proposition}\label{prop:frakxfraky}
Employing the notations \eqref{eq:GammaParam1}, let $\gamma^c\in [N_1]\setminus \left\{1\right\}$ and $\beta^c\in [N_1]\setminus \left\{1,2 \right\}$. Let $(n,N_1)\in \mathcal{N}_{1}$ and $\omega\in \mathcal{Q}^{N_1}_{N_1-1}$, and introduce 
\begin{equation}\label{eq:FrakeUpper}
\begin{array}{lll}
&\mathfrak{m}(\gamma^c,\omega^c):= \frac{\phi_{N_1+1}^{\gamma^c-1}}{\phi_n\prod_{s\in [N_1],s\neq \omega^c}\phi_s} \mathfrak{M}_{N_1+1-\gamma^c}(\phi_n,\phi_{\omega}) = \sum_{\substack{B\subseteq \mathcal{J}(\omega^c): \\ \operatorname{card}(B)=N_1+1-\gamma^c}} \frac{\phi_{N_1+1}^{\gamma^c-1}\left(\prod_{\chi\in B}\chi \right)}{\phi_n\prod_{s\in [N_1],s\neq \omega^c}\phi_s},\vspace{0.1cm}\\
&\mathfrak{n}(\beta^c,\omega^c) := \frac{\phi_{N_1+1}^{\beta^c-2}}{\prod_{s\in [N_1], s\neq \omega^c}\phi_s} \mathfrak{N}_{N_1+1-\beta^c}(\phi_{\omega}) = \sum_{\substack{B\subseteq \left\{\phi_{\omega_j} \right\}_{j=1}^{N_1-1}: \\ \operatorname{card}(B)=N_1+1-\beta^c}} \frac{\phi_{N_1+1}^{\beta^c-2}\left(\prod_{\chi\in B}\chi \right)}{\prod_{s\in [N_1], s\neq \omega^c}\phi_s}\vspace{0.1cm}\\
&\mathcal{J}(\omega^c) :=  \left(\phi_n,\phi_{\omega_1},\dots, \phi_{\omega_{N_1-1}} \right).
\end{array}
\end{equation}   
Note that here we abuse the notations in the definition of $\mathcal{J}\left(\omega^c \right)$ by considering it to be both an ordered $N_1$-tuple, which may contain two appearances of $\phi_n$, depending on $\left\{\phi_{\omega_j} \right\}_{j=1}^{N_1-1}$, and a regular set. Then, for the regimes in \eqref{Eq:Regimes}, we have the following upper bounds.
\begin{itemize}
    \item \underline{Regimes 1 and 2:} 
    \begin{equation}\label{eq:frakyfrakxReg1}
    \begin{array}{lll}
        \mathfrak{m}(\gamma^c,\omega^c)&\leq \binom{N_1}{\gamma^c-1}e^{-\upsilon \Delta (2N_1+1)(\gamma^c-1)},\vspace{0.1cm}\\
        \mathfrak{n}(\beta^c,\omega^c)&\leq \binom{N_1-1}{\beta^c-2}e^{-\upsilon \Delta (2N_1+1)(\beta^c-2)}.
    \end{array}
    \end{equation}

    \item \underline{Regime 3:}
\begin{equation}\label{eq:frakyfrakxReg3}
    \begin{array}{lll}
        \mathfrak{m}(\gamma^c,\omega^c)&\leq \binom{N_1}{\gamma^c-1}\left(\frac{ \upsilon T+1}{\upsilon T N_1} \right)^{\gamma^c-1},\\
        \mathfrak{n}(\beta^c,\omega^c)&\leq \binom{N_1-1}{\beta^c-2}\left(\frac{\upsilon T +1}{\upsilon T N_1} \right)^{\beta^c-2}.
    \end{array}
    \end{equation}
\end{itemize}
\end{proposition}
\begin{proof}
    Consider $\mathfrak{m}(\gamma^c,\omega^c)$ with $\gamma^c\in [N_1]\setminus \left\{ 1\right\}$. Fixing any $B\subseteq \mathcal{J}(\omega^c)$ with $\operatorname{card}(B)=N_1+1-\gamma^c$, we denote $B^c := \mathcal{J}(\omega^c)\setminus B =: \left\{\phi_{r_1(B)},\dots, \phi_{r_{\gamma^c-1}(B)} \right\}$, where $\phi_{r_i(B)}=e^{-\Delta\lambda_{r_i(B)}}$, $i=1,\dots,\gamma^c-1$.
    Then, we note that $\prod_{\chi\in B}\chi = \frac{\phi_n \phi_{\omega_1}\dots \phi_{\omega_{N_1-1}}}{\phi_{r_1(B)}\dots \phi_{r_{\gamma^c-1}(B)}}$, where the product at the numerator includes all the elements in $\mathcal{J}(\omega^c)$, and $\prod_{s\in [N_1],s\neq \omega^c}\phi_s = \phi_{\omega_1}\dots \phi_{\omega_{N_1-1}}$. Hence, we have
    \begin{equation*}    
    \frac{\phi_{N_1+1}^{\gamma^c-1}\left(\prod_{\chi\in B}\chi \right)}{\phi_n\prod_{s\in [N_1],s\neq \omega^c}\phi_s} = \phi_{N_1+1}^{\gamma^c-1} \frac{\phi_n}{\phi_n} \frac{\phi_{\omega_1}\dots \phi_{\omega_{N_1-1}}}{\prod_{s\in [N_1],s\neq \omega^c}\phi_s} \frac{1}{\phi_{r_1(B)}\dots \phi_{r_{\gamma^c-1}(B)}}=\frac{\phi_{N_1+1}^{\gamma^c-1}}{\phi_{r_1(B)}\dots \phi_{r_{\gamma^c-1}(B)}}
    \end{equation*}
    and thus
    \begin{equation*}
    \begin{array}{lll}
    \mathfrak{m}(\gamma^c,\omega^c) &= \sum_{\substack{B\subseteq \mathcal{J}(\omega^c): \\ \operatorname{card}(B)=N_1+1-\gamma^c}}\frac{\phi_{N_1+1}^{\gamma^c-1}}{\phi_{r_1(B)}\dots \phi_{r_{\gamma^c-1}(B)}} = \sum_{\substack{B^c\subseteq \mathcal{J}(\omega^c): \\ \operatorname{card}(B^c)=\gamma^c-1}} \prod_{\phi_r \in B^c} \frac{\phi_{N_1+1}}{\phi_r} \\
    &\overset{\eqref{eq:SymmPolDef}}{=} \mathfrak{S}_{\gamma^c-1}\left(\left\{ \frac{\phi_{N_1+1}}{\phi_r}\right\}_{\phi_r\in \mathcal{J}(\omega^c)}  \right)  \overset{\eqref{eq:SymmMeans} }{=} \binom{N_1}{\gamma^c-1}\mathcal{M}_{\gamma^c-1}\left(\left\{ \frac{\phi_{N_1+1}}{\phi_r}\right\}_{\phi_r\in \mathcal{J}(\omega^c)}  \right) . 
    \end{array}
    \end{equation*}
    Since the elements of $\left\{ \frac{\phi_{N_1+1}}{\phi_r}\right\}_{\phi_r\in \mathcal{J}(\omega^c)}$ are positive, we apply Lemma~\ref{Lem:MacIneq} to obtain
    \begin{equation}\label{eq:xFrakDeltaN1Const}
    \begin{array}{lll}
    \mathfrak{m}(\gamma^c,\omega^c) &\leq \binom{N_1}{\gamma^c-1}\left[\mathcal{M}_{1}\left(\left\{ \frac{\phi_{N_1+1}}{\phi_r}\right\}_{\phi_r\in \mathcal{J}(\omega^c)} \right)\right]^{\gamma^c-1}.
    \end{array}
    \end{equation}
    Considering $\mathfrak{n}(\beta^c,\omega^c)$ with $\beta^c\in [N_1]\setminus \left\{1,2 \right\}$, 
    similar arguments yield 
    \begin{equation}\label{eq:yFrakDeltaN1Const}
    \begin{array}{lll}
    \mathfrak{n}(\beta^c,\omega^c) &\leq \binom{N_1-1}{\beta^c-2}\left[\mathcal{M}_{1}\left(\left\{\frac{\phi_{N_1+1}}{\phi_{\omega_j}} \right\}_{j=1}^{N_1-1} \right)\right]^{\beta^c-2}.
    \end{array}
    \end{equation}
    By \eqref{eq:xFrakDeltaN1Const} and \eqref{eq:yFrakDeltaN1Const}, we proceed with upper-bounding $\mathcal{M}_{1}\left(\left\{ \frac{\phi_{N_1+1}}{\phi_r}\right\}_{\phi_r\in \mathcal{J}(\omega^c)} \right)$ and $\mathcal{M}_{1}\left(\left\{\frac{\phi_{N_1+1}}{\phi_{\omega_j}} \right\}_{j=1}^{N_1-1} \right)$ in each of the regimes.
    
    \underline{Regimes 1 and 2:} In both cases, by applying the definition \eqref{eq:SymmMeans}, the definition $\phi_n=e^{-\Delta\lambda_{n}}$, as well as the lower bound $\lambda_m-\lambda_n \geq \upsilon (m^2-n^2)$ in \eqref{eq:EigDiff}, proven in Proposition~\ref{prop:EigDiffBound}, 
    \begin{equation*}
    \begin{array}{lll}
    \mathcal{M}_{1}\left(\left\{ \frac{\phi_{N_1+1}}{\phi_r}\right\}_{\phi_r\in \mathcal{J}(\omega^c)} \right) &=  \frac{1}{N_1} \sum_{\phi_r \in \mathcal{J}(\omega^c)} \frac{\phi_{N_1+1}}{\phi_r} = \frac{1}{N_1}\sum_{\phi_r \in \mathcal{J}(\omega^c)}e^{-\Delta(\lambda_{N_1+1}-\lambda_r)} \\
    & \leq \frac{1}{N_1}\sum_{\phi_r \in \mathcal{J}(\omega^c)}e^{-\upsilon \Delta ((N_1+1)^2-r^2)}\leq \frac{1}{N_1}\sum_{\phi_r \in \mathcal{J}(\omega^c)}e^{-\upsilon \Delta (2N_1+1)} = e^{-\upsilon \Delta (2N_1+1)},\\[4mm]
    \mathcal{M}_{1}\left(\left\{\frac{\phi_{N_1+1}}{\phi_{\omega_j}}  \right\}_{j=1}^{N_1-1} \right) &=  \frac{1}{N_1-1} \sum_{j=1}^{N_1-1} \frac{\phi_{N_1+1}}{\phi_{\omega_j}} = \frac{1}{N_1-1}\sum_{j=1}^{N_1-1}e^{-\Delta (\lambda_{N_1+1}-\lambda_{\omega_j})}\leq e^{-\upsilon \Delta (2N_1+1)}.
    \end{array}
    \end{equation*}
    Substituting the latter into  \eqref{eq:xFrakDeltaN1Const} and \eqref{eq:yFrakDeltaN1Const} yields the result for Regimes 1 and 2.
    
    \underline{Regime 3:} Since in this case $e^{-2\upsilon \Delta (2N_1+1)} = e^{-2\upsilon T \left(2+\frac{1}{N_1} \right)}$ does not tend to zero as $N_1\to \infty$, we need to employ  sharper estimates. For $(n,N_1)\in \mathcal{N}_1$, considering $\Delta = \frac{T}{N_1}$, we have 
    \begin{equation}\label{eq:CalM1Reg3}
    \begin{array}{lll}
    \mathcal{M}_1\left(\left\{ \frac{\phi_{N_1+1}}{\phi_r}\right\}_{\phi_r\in \mathcal{J}(\omega^c)} \right) &=  \frac{1}{N_1} \sum_{\phi_r \in \mathcal{J}(\omega^c)} \frac{\phi_{N_1+1}}{\phi_r}= \frac{1}{N_1}\sum_{\phi_r \in \mathcal{J}(\omega^c)}e^{-\Delta(\lambda_{N_1+1}-\lambda_r)}\\
    &\leq \frac{1}{N_1}\sum_{\phi_r\in \mathcal{J}(\omega^c)}e^{-\upsilon \frac{T}{N_1}\left((N_1+1)^2-r^2 \right)}\vspace{0.1cm}\\
    & = \frac{1}{N_1} e^{-\upsilon T \frac{(N_1+1)^2}{N_1}}\left(\sum_{r=1}^{N_1}e^{\upsilon T N_1 \left(\frac{r}{N_1} \right)^2}+e^{\upsilon T N_1 \left(\frac{n}{N_1} \right)^2}-e^{\upsilon T N_1 \left(\frac{\omega^c}{N_1} \right)^2} \right)\\
    &\leq e^{-\upsilon T \frac{(N_1+1)^2}{N_1}}\sum_{r=1}^{N_1}\frac{1}{N_1}e^{\upsilon T N_1 \left(\frac{r}{N_1} \right)^2}+\frac{1}{N_1} e^{-\upsilon T N_1 \left(\frac{N_1+1-n }{N_1} \right)^2}\\
    &\leq e^{-\upsilon T \frac{(N_1+1)^2}{N_1}}\int_0^{1+\frac{1}{N_1}}e^{\upsilon T N_1 x^2}\mathrm{d}x +\frac{1}{N_1}e^{-\upsilon T N_1 \left(\frac{N_1+1-n }{N_1} \right)^2}\\
    &\overset{\eqref{eq:DefErfi}}{=} \frac{1}{2} \sqrt{\frac{ \pi}{N_1\upsilon T }} e^{-\upsilon T \frac{(N_1+1)^2}{N_1}}\mathcal{E}\left( \sqrt{\frac{\upsilon T }{N_1}}(N_1+1)\right)+ \frac{1}{N_1} e^{-\upsilon T N_1 \left(\frac{N_1+1-n }{N_1} \right)^2},
    \end{array}
    \end{equation}
    where the third inequality follows from the fact that $x\mapsto e^{\upsilon T N_1 x^2}$ is strictly increasing on $x\in[0,\infty)$, whereas the equality immediately after follows from a change of variables in integration (from $x$ to $z:= \sqrt{\upsilon T N_1} x$).
    Then, for $N_1 \gg 1$, resorting to the asymptotic behavior in \eqref{eq:efiasymp} yields
    \begin{equation*}    
    \mathcal{M}_1\left(\left\{ \frac{\phi_{N_1+1}}{\phi_r}\right\}_{\phi_r\in \mathcal{J}(\omega^c)} \right) \leq \frac{1}{\upsilon T} \frac{1}{N_1+1}+\frac{1}{N_1}e^{-\frac{\upsilon T}{N_1} }\leq \frac{\upsilon T + 1}{\upsilon T N_1}.
    \end{equation*}
    Similarly, 
    \begin{equation*}
    \begin{array}{lll}
    \mathcal{M}_{1}\left(\left\{\frac{\phi_{N_1+1}}{\phi_{\omega_j}}  \right\}_{j=1}^{N_1-1} \right) &= \frac{1}{N_1-1}\sum_{j=1}^{N_1-1}e^{-\Delta (\lambda_{N_1+1}-\lambda_{\omega_j})}\\
    & \leq \frac{1}{N_1-1}\left(\sum_{j=1}^{N_1}e^{-\upsilon \frac{T}{N_1}((N_1+1)^2-j^2)} - e^{-\upsilon \frac{T}{N_1}\left((N_1+1)^2-(\omega^c)^2\right)} \right)\\
    &\leq \frac{N_1}{N_1-1}e^{-\upsilon T \frac{(N_1+1)^2}{N_1}}\sum_{j=1}^{N_1}\frac{1}{N_1}e^{\upsilon T N_1 \left(\frac{j}{N_1} \right)^2 }.
    \end{array}
    \end{equation*}
    For $N_1\gg 1$, proceeding as before, we have
    \begin{equation*}
        \mathcal{M}_{1}\left(\left\{\frac{\phi_{N_1+1}}{\phi_{\omega_j}}  \right\}_{j=1}^{N_1-1} \right)  \leq \frac{\upsilon T+ 1}{\upsilon T N_1}.
    \end{equation*}
    Again, substitution into \eqref{eq:xFrakDeltaN1Const} and \eqref{eq:yFrakDeltaN1Const} yields the result for Regime 3.
\end{proof}

We now provide a finite upper bound for the ratio $\frac{\left|\bar{q}(z)-\bar{p}(z) \right|}{\left|\bar{p}(z) \right|}$ in the limit, in all the considered regimes.

\begin{lemma}\label{lem:UppBdlimN1}
Consider
\begin{equation}\label{eq:PHIdef}
\Phi(\Delta, N_1):=\max_{n \in [N_1]}\left( \sum_{\gamma \in \mathcal{Q}_{N_1}^{N_1+1}}\sum_{\beta \in \mathcal{Q}^{N_1+1}_{N_1}: 1\in \beta} \max_{z\in \partial B(\phi_n,\tau_n)}\mathfrak{s}(z,\gamma^c,\beta^c)\right).
\end{equation}
For the regimes in \eqref{Eq:Regimes}, the following bounds hold.
\begin{itemize}
    \item \underline{Regime 1}: Let $ \Delta_*>0$. Then,
    \begin{equation}\label{eq:limN1}
\limsup_{N_1\to \infty} \Phi(\Delta, N_1) \leq M_y \frac{1+\tau}{\tau},
\end{equation}
uniformly in $\Delta \in [\Delta_*,\infty)$.

\item \underline{Regime 2:} 
\begin{equation}\label{eq:RelErrDeltaInfty}
\limsup_{\Delta \to \infty} \Phi(\Delta, N_1) \leq M_y\frac{1+\tau}{\tau}
\end{equation}
uniformly in  $N_1\geq 2$.

\item \underline{Regime 3:} Let $T>0$ and $\Delta N_1 =T$. Then, there exists a positive, monotonically decreasing function $\mathcal{C}(T)$ such that 
\begin{equation}\label{eq:limN1DeltaConst}
\limsup_{N_1\to \infty} \Phi(\Delta, N_1) \leq \mathcal{C}(T).
\end{equation}
\end{itemize}
\end{lemma}
\begin{proof}
    We split the proof into two parts.
    First, we show that the considered quantities are upper bounded by the product between a regime-specific constant and a limit, which is assumed to be finite. Second, we prove that such a limit is indeed finite and we provide an upper bound for it in each regime.
    
    \underline{Part 1:} Given $\omega \in \mathcal{Q}^{N_1}_{N_1-1}$, let us further denote  one summand in $\mathfrak{s}(z,\gamma^c,\beta^c)$, defined in \eqref{eq:Ssum}, as 
    \begin{equation}\label{eq:dfracDef}
    \begin{array}{lll}
    &\mathfrak{f}(z,\gamma^c,\beta^c,\omega^c):= \frac{ |y_{N_1+1}|}{|y_{\omega^c}|} \phi_{N_1+1}^{\beta^c+\gamma^c-3}\frac{\prod_{s\neq \omega^c}\phi_s^{\mathrm{b}}}{\prod_{s\neq \omega^c}|\phi_s-\phi_{\omega^c}|^2}\frac{|z|^{\mathrm{a}} \mathfrak{M}_{N_1+1-\gamma^c}(|z|,\phi_{\omega}) \mathfrak{N}_{N_1+1-\beta^c}(\phi_{\omega})}{|z-\phi_{\omega^c}|}.
    \end{array}
    \end{equation}
    Then, we can write all the summands in $\mathfrak{s}(z,\gamma^c,\beta^c)$ as proportional to
    \begin{equation}\label{eq:mathfrakf12}  
    \mathfrak{f}(z,1,2,\omega^c) = \frac{ |y_{N_1+1}|}{|y_{\omega^c}|} \frac{\prod_{s\neq \omega^c}\phi_s^{2}}{\prod_{s\neq \omega^c}|\phi_s-\phi_{\omega^c}|^2}\frac{|z|}{|z-\phi_{\omega^c}|},
    \end{equation}
    where, in the case $\gamma^c=1$ and $\beta^c=2$, \eqref{eq:GammaParam1} and \eqref{eq:CompundExpansionNew3} yield
    $\gamma = (2,\dots,N_1+1)$,
    $\beta = (1,3,\dots,N_1+1)$, $\mathrm{a} = 1$, $\mathrm{b}=2$ and $\mathfrak{M}_{N_1+1-\gamma^c}(|z|,\phi_{\omega})  = \mathfrak{N}_{N_1+1-\beta^c}(\phi_{\omega})  = 1$.
    
    In fact, the following equalities hold:
    \begin{equation}\label{eq:dfrakrelations}
    \hspace{-4mm}
    \begin{array}{lll}
    &\mathrm{I.}\quad  \mathfrak{f}(z,N_1+1,2,\omega^c) = \frac{\phi_{N_1+1}^{N_1}}{|z|\prod_{s\neq 
    \omega^c}\phi_s}\mathfrak{f}(z,1,2,\omega^c),\\
    &\mathrm{II.}\quad \mathfrak{f}(z,1,N_1+1,\omega^c)  = \frac{\phi_{N_1+1}^{N_1-1}}{\prod_{s\neq \omega^c}\phi_s}\mathfrak{f}(z,1,2,\omega^c),\\
    &\mathrm{III.}\quad \mathfrak{f}(z,N_1+1,N_1+1,\omega^c)  = \frac{\phi_{N_1+1}}{|z|}\left(\frac{\phi_{N_1+1}^{N_1-1}}{\prod_{s\neq \omega^c}\phi_s}\right)^2\mathfrak{f}(z,1,2,\omega^c),\\
    &\mathrm{IV.}\quad \mathfrak{f}(z,\gamma^c,2,\omega^c) = \frac{\mathfrak{M}_{N_1+1-\gamma^c}(|z|,\phi_{\omega})\phi_{N_1+1}^{\gamma^c-1}}{|z|\prod_{s\neq \omega^c}\phi_s}\mathfrak{f}(z,1,2,\omega^c), \quad \gamma^c\in [N_1]\setminus \left\{1 \right\},\\
    &\mathrm{V.}\quad \mathfrak{f}(z,\gamma^c,N_1+1,\omega^c) = \frac{\phi_{N_1+1}^{N_1-1}}{\prod_{s\neq \omega^c}\phi_s}\mathfrak{f}(z,\gamma^c,2,\omega^c), \quad \gamma^c \in [N_1]\setminus \left\{1\right\},\\
    &\mathrm{VI.}\quad \mathfrak{f}(z,1,\beta^c,\omega^c) = \frac{\mathfrak{N}_{N_1+1-\beta^c}(\phi_{\omega})\phi_{N_1+1}^{\beta^c-2}}{\prod_{s\neq \omega^c}\phi_s}\mathfrak{f}(z,1,2,\omega^c),\quad \beta^c \in [N_2]\setminus \left\{1,2 \right\},\\
    &\mathrm{VII.}\quad \mathfrak{f}(z,N_1+1,\beta^c,\omega^c) = \frac{1}{|z|}\frac{\phi_{N_1+1}^{N_1}}{\prod_{s\neq \omega^c}\phi_s}\mathfrak{f}(z,1,\beta^c,\omega^c),\quad \beta^c\in [N_1]\setminus \left\{1,2 \right\},\\
    &\mathrm{VIII.}\quad \mathfrak{f}(z,\gamma^c,\beta^c,\omega^c) = \frac{\mathfrak{M}_{N_1+1-\gamma^c}(|z|,\phi_{\omega})\phi_{N_1+1}^{\gamma^c-1}}{|z|\prod_{s\neq \omega^c}\phi_s}\\
    &\hspace{40mm}\times \frac{\mathfrak{N}_{N_1+1-\beta^c}(\phi_{\omega})\phi_{N_1+1}^{\beta^c-2}}{\prod_{s\neq \omega^c}\phi_s}\mathfrak{f}(z,1,2,\omega^c), \quad \gamma^c\in [N_1]\setminus \left\{1\right\},\ \beta^c\in [N_1]\setminus \left\{1,2 \right\}.
    \end{array}
    \end{equation}
    We now show that, for $\Phi(\Delta, N_1)$ defined in \eqref{eq:PHIdef}, \eqref{eq:dfrakrelations} implies 
    \begin{equation}\label{eq:Pasrt1Des}
    \limsup_{*} \Phi(\Delta, N_1) \leq C_* \cdot \limsup_{*}\max_{n\in [N_1 ]}\max_{z\in\partial B(\phi_n,\tau_n)}\mathfrak{s}(z,1,2) < \infty,
    \end{equation}
    provided that the limit on the right hand side is finite; at this stage of the proof, we assume that this is the case. Here, the subscripts $*$ in the limits depend on the considered regime, and $C_*\geq 0$ is a constant that varies depending on the regime.
    
    For $z\in \partial B(\phi_n,\tau_n)$, we have $|z| = |\phi_n + (z - \phi_n)| \geq \phi_n - |z-\phi_n| = \phi_n - \tau \phi_n = (1-\tau)\phi_n$, where we used the reverse triangle inequality. Since, in addition, $\left\{\lambda_n\right\}_{n=1}^{\infty}$ is a monotonically increasing sequence, the equalities I., II. and III. in \eqref{eq:dfrakrelations} imply
    \begin{equation*}
    \begin{array}{lll}
    &\mathfrak{f}(z,N_1+1,2,\omega^c) \leq \frac{1}{1-\tau} \frac{\phi_{N_1+1}}{\phi_n} \frac{\phi_{N_1+1}^{N_1-1}}{\prod_{s\neq 
    \omega^c}\phi_s} \mathfrak{f}(z,1,2,\omega^c) \leq  \frac{1}{1-\tau}e^{-\Delta (\lambda_{N_1+1}-\lambda_n)-\Delta\sum_{s=2}^{N_1}(\lambda_{N_1+1}-
    \lambda_s)}\mathfrak{f}(z,1,2,\omega^c),\\
    &\mathfrak{f}(z,1,N_1+1,\omega^c) \leq \frac{\phi_{N_1+1}^{N_1-1}}{\prod_{s\neq \omega^c}\phi_s} \mathfrak{f}(z,1,2,\omega^c) \leq e^{-\Delta \sum_{s=2}^{N_1}(\lambda_{N_1+1}-\lambda_s)}\mathfrak{f}(z,1,2,\omega^c),\\
    &\mathfrak{f}(z,N_1+1,N_1+1,\omega^c) \leq \frac{1}{1-\tau} \frac{\phi_{N_1+1}}{\phi_n}\left(\frac{\phi_{N_1+1}^{N_1-1}}{\prod_{s\neq \omega^c}\phi_s}\right)^2 \mathfrak{f}(z,1,2,\omega^c) \\ & \hspace{36mm} \leq \frac{1}{1-\tau}e^{-\Delta(\lambda_{N_1+1}-\lambda_n)} e^{-2\Delta \sum_{s=2}^{N_1}(\lambda_{N_1+1}-\lambda_s)}\mathfrak{f}(z,1,2,\omega^c),
    \end{array}
    \end{equation*}
    and hence, summing over all possible $\omega^c$,
    \begin{equation}\label{eq:MultConst}
    \begin{array}{lll}
    &\mathfrak{s}(z,N_1+1,2) \leq  \frac{1}{1-\tau}e^{-\Delta (\lambda_{N_1+1}-\lambda_n)-\Delta\sum_{s=2}^{N_1}(\lambda_{N_1+1}-
    \lambda_s)}\mathfrak{s}(z,1,2),\\
    &\mathfrak{s}(z,1,N_1+1)\leq e^{-\Delta \sum_{s=2}^{N_1}(\lambda_{N_1+1}-\lambda_s)}\mathfrak{s}(z,1,2),\\
    &\mathfrak{s}(z,N_1+1,N_1+1)\leq \frac{1}{1-\tau}e^{-\Delta(\lambda_{N_1+1}-\lambda_n)} e^{-2\Delta \sum_{s=2}^{N_1}(\lambda_{N_1+1}-\lambda_s)}\mathfrak{s}(z,1,2).
    \end{array}
    \end{equation}
    Consider the terms multiplying $\mathfrak{s}(z,1,2)$ on the right-hand sides of \eqref{eq:MultConst}. In Regimes 1 and 2, all the terms converge to zero, since they all contain the term $e^{-\Delta (\lambda_{N_1+1}-\lambda_{N_1})}\leq e^{-\upsilon \Delta (2N_1+1)}$, which converges to zero in both regimes. In Regime 3, since $\Delta = \frac{T}{N_1}$, we employ
    the fact that $\sum_{j=1}^{N_1}j^2 = \frac{N_1(N_1+1)(2N_1+1)}{6}$ to obtain 
    \begin{equation*}
    \begin{array}{lll}
    e^{-\Delta \sum_{s=2}^{N_1}\left(\lambda_{N_1+1}-\lambda_s \right)}\leq e^{-\frac{\upsilon T}{N_1}  \sum_{s=2}^{N_1}\left((N_1+1)^2-s^2 \right)} \leq e^{-\upsilon T \frac{N_1+1}{N_1}\frac{4N_1^2-N_1-6}{6}}, 
    \end{array}
    \end{equation*}
    which tends to zero as $N_1 \to \infty$. Hence, all terms multiplying $\mathfrak{s}(z,1,2)$ on the right-hand side of \eqref{eq:MultConst} tend to zero in all of the regimes. We conclude that the limits of the terms on the left-hand side of \eqref{eq:MultConst} are zero, uniformly in $n\in [N_1]$ and $z\in \partial B(\phi_n,\tau_n)$.
    
    For case IV. in \eqref{eq:dfrakrelations}, the inequalities $(1-\tau)\phi_n\leq |z| \leq (1+\tau)\phi_n$ (where the upper bound follows from the triangle inequality) imply that
    \begin{equation}\label{eq:dFrakRelgammac2}
    \begin{array}{lll}
    \mathfrak{f}(z,\gamma^c,2,\omega^c) &= \frac{\mathfrak{M}_{N_1+1-\gamma^c}(|z|,\phi_{\omega})\phi_{N_1+1}^{\gamma^c-1}}{|z|\prod_{s\neq \omega^c}\phi_s}\mathfrak{f}(z,1,2,\omega^c) \leq \frac{1+\tau}{1-\tau}\mathfrak{m}(\gamma^c,\omega^c)\mathfrak{f}(z,1,2,\omega^c),
    \end{array}
    \end{equation}
    where $\mathfrak{m}(\gamma^c,\omega^c)$ was defined in Proposition~\ref{prop:frakxfraky}. Employing this proposition, we sum both sides of \eqref{eq:dFrakRelgammac2} over $\omega \in \mathcal{Q}^{N_1}_{N_1-1}$, take the maximum over $z\in \partial B(\phi_n,\tau_n)$ and then sum over $\gamma^c\in [N_1]\setminus\left\{1 \right\}$ to obtain
    \begin{equation*}
    \begin{array}{lll}
    \sum_{\gamma^c\in [N_1]\setminus \left\{1 \right\}}\max_{z\in \partial B(\phi_n,\tau_n)}\mathfrak{s}(z,\gamma^c,2) \leq \frac{1+\tau}{1-\tau}\left(\sum_{\gamma^c=2}^{N_1} \binom{N_1}{\gamma^c-1}q^{\gamma^c-1}\right)\max_{z\in \partial B(\phi_n,\tau_n)}\mathfrak{s}(z,1,2),
    \end{array}
    \end{equation*}
    where $q\in \left\{e^{-\upsilon \Delta (2N_1+1)},\frac{\upsilon T +1}{\upsilon T N_1} \right\}$, depending on the regime (see Proposition~\ref{prop:frakxfraky}). By the binomial formula,
    \begin{equation*}
    \begin{array}{lll}
    \sum_{\gamma^c=2}^{N_1} \binom{N_1}{\gamma^c-1}q^{\gamma^c-1} = \left( 1+q\right)^{N_1} - 1 - q^{N_1}.
    \end{array}
    \end{equation*}

    In Regime 1,
    \begin{equation*}
    \lim_{N_1 \to \infty} \left( \left(1+e^{-\upsilon \Delta (2N_1+1)}\right)^{N_1}-1-e^{-\upsilon \Delta N_1 (2N_1+1)} \right)=0
    \end{equation*}
    uniformly in $\Delta$, because
    \begin{equation*}
    \lim_{N_1 \to \infty} \left(1+e^{-\upsilon \Delta (2N_1+1)}\right)^{N_1} = e^{\lim_{N_1 \to \infty} N_1 \ln(1+e^{-\upsilon \Delta (2N_1+1)})}=e^0=1.
    \end{equation*}

    In Regime 2,
    \begin{equation*}
    \lim_{\Delta \to \infty} \left( \left(1+e^{-\upsilon \Delta (2N_1+1)}\right)^{N_1}-1-e^{-\upsilon \Delta N_1 (2N_1+1)} \right)=0
    \end{equation*}
    uniformly in $N_1$, because $0 \leq e^{-\upsilon \Delta N_1 (2N_1+1)} \leq e^{- 3\upsilon \Delta} \to 0$ as $\Delta \to \infty$, while, setting $\aleph = \upsilon \Delta (2N_1+1)$, so that $\aleph \to \infty$ as $\Delta \to \infty$, $\left(1+e^{-\upsilon \Delta (2N_1+1)}\right)^{N_1}-1 = \left\{\left[\left(1+e^{-\aleph}\right)^{\aleph}\right]^{\frac{N_1}{\upsilon(2N_1+1)}}\right\}^{\frac{1}{\Delta}}-1 \leq M^{\frac{1}{\Delta}}-1$ for some $M \geq 1$, and $\lim_{\Delta \to \infty} M^{\frac{1}{\Delta}}-1=0$.

    In Regime 3, 
    \begin{equation*}
    \lim_{N_1 \to \infty} \left( \left(1+\frac{\upsilon T +1}{\upsilon T N_1}\right)^{N_1}-1-\left(\frac{\upsilon T +1}{\upsilon T N_1}\right)^{N_1} \right)= e^{1+
    \frac{1}{\upsilon T}}-1,
    \end{equation*}    
    because, setting $x := \frac{1}{N_1}$ and $\alpha := \frac{\upsilon T +1}{\upsilon T}$, $\lim_{x\to 0^+}\left(1+ \alpha x \right)^{\frac{1}{x}} = e^\alpha$, while $\left(\frac{\alpha}{N_1}\right)^{N_1} \leq \left(\frac{1}{2}\right)^{N_1}$ for sufficiently large $N_1$, and $\lim_{N_1 \to \infty} \left(\frac{1}{2}\right)^{N_1}=0$.

    Therefore, to sum up,
    \begin{equation*}
    \begin{array}{lll}
    &\limsup_{*}\max_{n\in [N_1]}\sum_{\gamma^c\in [N_1]\setminus \left\{1 \right\}}\max_{z\in \partial B(\phi_n,\tau_n)}\mathfrak{s}(z,\gamma^c,2) \\
    &\hspace{10mm}\leq \limsup_{*}\max_{n\in [N_1]}\max_{z\in \partial B(\phi_n,\tau_n)}\mathfrak{s}(z,1,2) \times \begin{cases}
    0, & \text{Regimes 1 and 2},\\
    \frac{1+\tau}{1-\tau}\left(e^{1+
    \frac{1}{\upsilon T}}-1 \right), & \text {Regime 3}. 
    \end{cases}
    \end{array}
    \end{equation*}
    
    For case V. in \eqref{eq:dfrakrelations}, since $\frac{\phi_{N_1+1}^{N_1-1}}{\prod_{s\neq \omega^c}\phi_s}\leq e^{-\Delta \sum_{s=2}^{N_1}(\lambda_{N_1+1}-\lambda_s)}$ tends to zero in all regimes, employing IV. yields 
    \begin{equation*}
    \begin{array}{lll}
    &\limsup_{*}\max_{n\in [N_1]}\sum_{\gamma^c\in [N_1]\setminus \left\{1 \right\}}\max_{z\in \partial B(\phi_n,\tau_n)}\mathfrak{s}(z,\gamma^c,N_1+1) = 0.
    \end{array}
    \end{equation*}
    
    The cases VI. and VII. in \eqref{eq:dfrakrelations} are analogous to IV. and V. and yield, via Proposition~\ref{prop:frakxfraky},
    \begin{equation*}
    \begin{array}{lll}
    &\limsup_{*}\max_{n\in [N_1]}\sum_{\beta^c\in [N_1]\setminus \left\{1,2 \right\}}\max_{z\in \partial B(\phi_n,\tau_n)}\mathfrak{s}(z,1,\beta^c) \\
    &\hspace{10mm}\leq \limsup_{*}\max_{n\in [N_1]}\max_{z\in \partial B(\phi_n,\tau_n)}\mathfrak{s}(z,1,2) \times \begin{cases}
    0,& \text{Regimes 1 and 2,}\\
    \left(e^{1+\frac{1}{\upsilon T}}-1 \right),& \text {Regime 3,} 
    \end{cases}\\
    &\limsup_{*}\max_{n\in [\lfloor \eta N_1 \rfloor]}\sum_{\beta^c\in [N_1]\setminus \left\{1,2 \right\}}\max_{z\in \partial B(\phi_n,\tau_n)}\mathfrak{s}(z,N_1+1,\beta^c) = 0.
    \end{array}
    \end{equation*}
    
    Finally, in case VIII., by arguments similar to cases IV. and VI., we have 
    \begin{equation*}
    \begin{array}{lll}
    \mathfrak{f}(z,\gamma^c,2,\omega^c) &\leq \frac{1+\tau}{1-\tau}\mathfrak{m}(\gamma^c,\omega^c)\mathfrak{n}(\beta^c,\omega^c)\mathfrak{f}(z,1,2,\omega^c),
    \end{array}
    \end{equation*}
     which yields 
    \begin{equation*}
    \begin{array}{lll}
    &\limsup_{*}\max_{n\in [\lfloor \eta N_1 \rfloor]}\sum_{\gamma^c \in [N_1]\setminus\left\{1 \right\}}\sum_{\beta^c\in [N_1]\setminus \left\{1,2 \right\}}\max_{z\in \partial B(\phi_n,\tau_n)}\mathfrak{s}(z,1,\beta^c) \\
    &\hspace{10mm}\leq \limsup_{*}\max_{n\in [\lfloor \eta N_1 \rfloor]}\max_{z\in \partial B(\phi_n,\tau_n)}\mathfrak{s}(z,1,2) \times \begin{cases}
    0, & \text{Regimes 1 and 2},\\
    \frac{1+\tau}{1-\tau}\left(e^{1+\frac{1}{\upsilon T}}-1 \right)^2, & \text {Regime 3}. 
    \end{cases}
    \end{array}
    \end{equation*}
    Combining all cases together, we obtain \eqref{eq:Pasrt1Des} with 
    \begin{equation}\label{eq:Cstar}
    \begin{array}{lll}
    C_* = \begin{cases}
        1, & \text{Regimes 1 and 2,}\\
        \frac{1}{1-\tau}\left(e^{1
        +\frac{1}{\upsilon T}}-1 \right)\left(1-\tau+e^{1+\frac{1}{\upsilon T}}(1+\tau) \right)+1, & \text{Regime 3.}
    \end{cases}\  
    \end{array}
    \end{equation}
    
    \underline{Part 2:} We now estimate $\limsup_{*}\max_{n\in [N_1]}\max_{z\in \partial B(\phi_n,\tau_n)}\mathfrak{s}(z,1,2) $, where the subscript $*$ depends on the regime.
    In \eqref{eq:mathfrakf12}, $\frac{|y_{N_1+1}|}{|y_{\omega^c}|}\leq M_y$ due to Assumption~\ref{assump:yassump} and $\frac{|z|}{|z-\phi_{\omega^c}|}\leq \frac{1+\tau}{\tau}$ because $|z| \leq (1+\tau) \phi_n$ and $|z-\phi_{\omega^c}| = |(\phi_n -\phi_{\omega^c})+(z- \phi_n)| \geq 2 \tau \phi_n - \tau \phi_n = \tau \phi_n$, where we used the reverse triangle inequality and Proposition~\ref{Prop:relSepUnif}. Hence, it suffices to upper bound $\frac{\prod_{s\neq \omega^c}\phi_s^{2}}{\prod_{s\neq \omega^c}|\phi_s-\phi_{\omega^c}|^2}$. By arguments similar to Proposition~\ref{prop:ComponentBounds0}, we have 
    \begin{equation}\label{eq:omegacoptions}
    \begin{array}{lll}
    \omega^c=1 & \Longrightarrow \ln \left(\frac{\prod_{s\in [N_1],s\neq 1}\phi_s^{2}}{\prod_{s\in [N_1],s\neq 1}|\phi_s-\phi_{\omega^c}|^2}\right)&=-2\Delta \sum_{s=2}^{N_1}\left(\lambda_s-\lambda_1 \right)+2\theta^3_{1,\Delta,N_1},\\
    \omega^c\neq 1,N_1 & \Longrightarrow \ln \left(\frac{\prod_{s\in [N_1],s\neq \omega^c}\phi_s^{2}}{\prod_{s\in [N_1],s\neq \omega^c}|\phi_s-\phi_{\omega^c}|^2}\right)&=-2\Delta \sum_{s=\omega^c+1}^{N_1}\left(\lambda_s-\lambda_{\omega^c} \right)+2\theta^3_{\omega^c,\Delta,N_1}+2\theta^2_{\omega^c,\Delta},\\
    \omega^c=N_1 & \Longrightarrow \ln \left(\frac{\prod_{s\in [N_1],s\neq N_1}\phi_s^{2}}{\prod_{s\in [N_1],s\neq N_1}|\phi_s-\phi_{\omega^c}|^2}\right)&=2\theta^2_{N_1,\Delta}.
    \end{array}
    \end{equation}
    Therefore, summing over all possible $\omega^c$,
    \begin{equation*}
    \begin{array}{lll}
    \mathfrak{s}(z,1,2) \leq M_y \frac{1+\tau}{\tau} \left(e^{2\theta^2_{N_1,\Delta}} + e^{-2\Delta \sum_{s=2}^{N_1}(\lambda_s-\lambda_1)+2\theta^3_{1,\Delta,N_1}}+\sum_{\omega^c=2}^{N_1-1}e^{-2\Delta \sum_{s=\omega^c+1}^{N_1}(\lambda_s-\lambda_{\omega^c})+2\theta^3_{\omega^c,\Delta,N_1}+2\theta^2_{\omega^c,\Delta}}\right).
    \end{array}
    \end{equation*}
    We now consider the three regimes separately.
    
    \underline{Regime 1:} 
    Since for any $\omega^c\in [N_1-1]$, in view of Proposition~\ref{prop:EigDiffBound}, 
    \begin{equation*}
    -2\Delta \sum_{s=\omega^c+1}^{N_1}(\lambda_s-\lambda_{\omega^c})\leq -2\upsilon \Delta \left( N_1^2 - (N_1-1)^2 \right) = -2\upsilon \Delta (2N_1-1)
    \end{equation*}
    and, by Corollary~\ref{Cor:ThetaAssymp}, statement (i),  $\theta^2_{n,\Delta}$ and $\theta^3_{n,\Delta,N_1}$ are $O(1)$, uniformly in $n\in [N_1]$ and $\Delta \in [\Delta_*,\infty)$, we obtain 
    \begin{equation*}
    \max_{n\in [N_1 ]}\max_{z\in \partial B(\phi_n,\tau_n)}\mathfrak{s}(z,1,2)\leq M_y \frac{1+\tau}{\tau}\left(e^{2\theta^2_{N_1,\Delta}}+(N_1-1)e^{-2\upsilon \Delta (2N_1-1) + O(1)} \right),
    \end{equation*}
    whence $\limsup_{N_1\to \infty}\max_{n\in [N_1 ]}\max_{z\in \partial B(\phi_n,\tau_n)}\mathfrak{s}(z,1,2) \leq M_y \frac{1+\tau}{\tau}$, because $\theta^2_{N_1,\Delta}$ tends to zero, uniformly in $\Delta \in [\Delta_*,\infty)$, in view of \eqref{eq:thetabounds}. In combination with \eqref{eq:Pasrt1Des} and \eqref{eq:Cstar}, this yields \eqref{eq:limN1}.
    
    \underline{Regime 2:} By Corollary~\ref{Cor:ThetaAssymp}, statement (ii), $\theta^1_{n,\Delta,N_1}$, $\theta^2_{n,\Delta}$ and $\theta^3_{n,\Delta,N_1}$ are $O(1)$, uniformly in $(n,N_1)\in \mathcal{N}_1$, and, if $\Delta \to \infty$, then $\theta^1_{n,\Delta,N_1}$, $\theta^2_{n,\Delta}$ and $\theta^3_{n,\Delta,N_1}$ are $o(1)$, uniformly in $n,N_1$. Hence, since $\Delta \geq \Delta_*$, in this case we have  
    \begin{equation*}
    \begin{array}{lll}
    (N_1-1)e^{-2\upsilon \Delta (2N_1-1) + O(1)} &\leq (N_1-1)e^{-\upsilon \Delta_* (2N_1-1) + O(1)} e^{-\upsilon \Delta }\\
    &\leq \max_{N_1\in \mathbb{N} }\left[(N_1-1)e^{-\upsilon \Delta_* (2N_1-1) + O(1)} \right] e^{-\upsilon \Delta },
    \end{array}
    \end{equation*}
    with the maximum being finite. Therefore, 
    \begin{equation*}
    \begin{array}{lll}
    &\max_{n\in [N_1 ]}\max_{z\in \partial B(\phi_n,\tau_n)}\mathfrak{s}(z,1,2)\\
    &\hspace{20mm}\leq M_y \frac{1+\tau}{\tau}\left(e^{2\theta^2_{N_1,\Delta}}+\max_{N_1\in \mathbb{N}}\left[(N_1-1)e^{-\upsilon \Delta_* (2N_1-1) + O(1)} \right]e^{-\upsilon \Delta }\right),
    \end{array}
    \end{equation*}
    whence $\limsup_{\Delta \to \infty}\max_{n\in [N_1 ]}\max_{z\in \partial B(\phi_n,\tau_n)}\mathfrak{s}(z,1,2) \leq M_y \frac{1+\tau}{\tau}$, uniformly in $N_1\in \mathbb{N}\setminus \left\{1 \right\}$. In combination with \eqref{eq:Pasrt1Des} and \eqref{eq:Cstar}, this yields \eqref{eq:RelErrDeltaInfty}.
    
    \underline{Regime 3:}  
    Fix $\varphi \in (0,1)$.  Consider first $\omega^c \leq \lfloor \varphi N_1 \rfloor$, $N_1\gg1$. Then, $\theta^3_{\omega^c,\Delta,N_1}$ and $\theta^2_{\omega^c,\Delta}$ are both $O(N_1)$, by Corollary~\ref{Cor:ThetaAssymp}, statement (iii). Therefore, for $\omega^c \in [\lfloor \varphi N_1 \rfloor]$, in view of Proposition~\ref{prop:EigDiffBound},
    \begin{equation}\label{eq:auxsum}
    \begin{array}{lll}
    -2\Delta \sum_{s=\omega^c+1}^{N_1}\left(\lambda_s-\lambda_{\omega^c} \right) \leq -2\Delta \sum_{s=\omega^c+1}^{N_1} \upsilon \left(s^2-(\omega^c)^2 \right) \overset{\eqref{eq:CalSdef}}{=} -2\Delta \upsilon \Psi(\omega^c; N_1)\leq-2\upsilon T\frac{\Psi(\lfloor \varphi N_1 \rfloor; N_1 )}{N_1},
    \end{array}
    \end{equation}
    since $\Psi(n;N_1)$ is monotonically decreasing in $n$ in view of Proposition~\ref{Prop:VFrakAsymp}, which also ensures that, for $N_1 \gg 1$, $\Psi(\lfloor \varphi N_1 \rfloor; N_1 ) \geq a(\varphi)N_1^3$. Hence, we conclude that 
    \begin{equation*}
    \begin{array}{lll}
    &\limsup_{N_1\to \infty }\max_{n\in [N_1]}\left(\max_{z\in \partial B(\phi_n,\tau_n)}\sum_{\omega^c \in [\lfloor \varphi N_1 \rfloor] }\frac{\prod_{s\neq \omega^c}\phi_s^{2}}{\prod_{s\neq \omega^c}|\phi_s-\phi_{\omega^c}|^2}\right)\\
    &\hspace{40mm}\leq \lim_{N_1\to \infty }\lfloor \varphi N_1 \rfloor e^{-2\upsilon T\frac{\Psi(\lfloor \varphi N_1 \rfloor; N_1 )}{N_1}+O(N_1)} = 0.
    \end{array}
    \end{equation*}
    It remains to consider the case $\omega^c\in [N_1]\setminus [\lfloor \varphi N_1 \rfloor ]$, where \eqref{eq:auxsum} does not necessarily yield an exponentially decaying term.
    In this case, the proof of Corollary~\ref{Cor:ThetaAssymp}, statement (iii), shows that $\theta^2_{\omega^c,\Delta}$ and $\theta^3_{\omega^c,\Delta,N_1}$ are bounded for $\omega^c\in [N_1]\setminus [\lfloor \varphi N_1 \rfloor ]$, uniformly in $N_1$, with the bounds
    \begin{equation}\label{eq:ThetaTailBounds}
    \begin{array}{lll}
     0\leq \theta^2_{\omega^c,\Delta} \leq \mathcal{G}_{0,\infty}(\varphi \upsilon T),\quad 
     0\leq \theta^3_{\omega^c,\Delta,N_1}\leq \mathcal{G}_{0,\infty}(\varphi \upsilon T).
     \end{array}
    \end{equation}
    Therefore, for $N_1\gg 1$ we have 
    \begin{equation}\label{eq:Step1gammacabetac2DeltaN1Const}
    \begin{array}{lll}
   \sum_{\lfloor \varphi N_1 \rfloor < \omega^c \leq N_1} \frac{\prod_{s\neq \omega^c}\phi_s^{2}}{\prod_{s\neq \omega^c}|\phi_s-\phi_{\omega^c}|^2} &\leq e^{4\mathcal{G}_{0,\infty}(\varphi \upsilon T)}\left(1+ \sum_{\lfloor \varphi N_1 \rfloor <\omega^c < N_1 }e^{-2\Delta \sum_{s=\omega^c+1}^{N_1}(\lambda_s-\lambda_{\omega^c})} \right)\\
    &\leq e^{4\mathcal{G}_{0,\infty}(\varphi \upsilon T)}\left(1+ \sum_{\lfloor \varphi N_1 \rfloor <\omega^c < N_1 }e^{-\frac{2\upsilon T}{N_1} \sum_{s=\omega^c+1}^{N_1} \left( s^2-\left(\omega^c\right)^2 \right)} \right)\\
    & \leq e^{4\mathcal{G}_{0,\infty}(\varphi \upsilon T)}\left(1+ e^{-2\upsilon T N_1}\sum_{\lfloor \varphi N_1 \rfloor <\omega^c < N_1 }e^{2\upsilon T \frac{\left(\omega^c\right)^2}{N_1}} \right),
    \end{array}
    \end{equation}
    where the first inequality employs  \eqref{eq:omegacoptions} and \eqref{eq:ThetaTailBounds}, the second follows from Proposition~\ref{prop:EigDiffBound}, and the third relies on the upper bound $e^{-\frac{2\upsilon T}{N_1} \sum_{s=\omega^c+1}^{N_1} \left( s^2-\left(\omega^c\right)^2 \right)} \leq e^{-\frac{2\upsilon T}{N_1} \left( {N_1}^2-\left(\omega^c\right)^2 \right)} = e^{-2\upsilon T N_1} e^{\frac{2\upsilon T}{N_1} \left(\omega^c\right)^2}$.
    
    Recall the imaginary error function $\mathcal{E}(x)$ given in \eqref{eq:DefErfi}. Since function $x\mapsto e^{2\upsilon T\frac{x^2}{N_1}}$ is strictly increasing on $x\in [0,\infty)$, and we can employ a change of variables ($t:=\sqrt{\frac{2\upsilon T}{N_1}}x$) in the integration, we obtain
    \begin{equation*}
    \begin{array}{lll}
    \sum_{\lfloor \varphi N_1 \rfloor <\omega^c < N_1 }e^{2\upsilon T \frac{\left(\omega^c\right)^2}{N_1}} \leq \int_{\lfloor \varphi N_1 \rfloor}^{N_1}e^{2\upsilon T \frac{x^2}{N_1}}\mathrm{d}x \leq e^{-2\upsilon T N_1}\int_0^{N_1}e^{2\upsilon T \frac{x^2}{N_1}}\mathrm{d}x \overset{\eqref{eq:DefErfi}}{=} \frac{1}{2}\sqrt{\frac{\pi N_1}{2\upsilon T}}\mathcal{E}(\sqrt{2\upsilon T N_1}).
    \end{array}
    \end{equation*}
    Then, the asymptotic behavior in \eqref{eq:efiasymp} yields, for $N_1\gg 1$,
    \begin{equation}\label{eq:Step1gammacabetac2DeltaN1Const1}
    \begin{array}{lll}
    e^{-2\upsilon TN_1} \sum_{\lfloor \varphi N_1 \rfloor <\omega^c < N_1 }e^{2\upsilon T \frac{\left(\omega^c\right)^2}{N_1}} \leq \frac{1}{2}\sqrt{\frac{\pi N_1}{2\upsilon T}}e^{-2\upsilon TN_1}\mathcal{E}(\sqrt{2\upsilon T N_1})\leq\frac{1}{2\upsilon T}.
    \end{array}
    \end{equation}
    
    Combining \eqref{eq:Step1gammacabetac2DeltaN1Const} and \eqref{eq:Step1gammacabetac2DeltaN1Const1}, we obtain
    \begin{equation*}
    \begin{array}{lll}
    &\limsup_{N_1\to \infty} \max_{n\in [N_1 ]}\left(\max_{z\in \partial B(\phi_n,\tau_n)}\sum_{\lfloor \eta N_1 \rfloor < \omega^c 
    \leq N_1}\frac{ |y_{N_1+1}|}{|y_{\omega^c}|}\frac{|z|}{\left|z-\phi_{\omega^c} \right|}\frac{\prod_{s\neq \omega^c}\phi_s^{2}}{\prod_{s\neq \omega^c}|\phi_s-\phi_{\omega^c}|^2}\right)\\
    &\hspace{50mm}\leq M_y\frac{1+\tau}{\tau}e^{4\mathcal{G}_{0,\infty}(\eta \upsilon T)}\left(1+ \frac{1}{2\upsilon T} \right).
    \end{array}
    \end{equation*}
    In combination with \eqref{eq:Pasrt1Des} and \eqref{eq:Cstar}, this yields \eqref{eq:limN1DeltaConst}, where
    \begin{equation*}
    \begin{array}{lll}
    \mathcal{C}(T) = \left[\frac{1}{1-\tau}\left(e^{1+ \frac{1}{\upsilon T}}-1 \right)\left(1-\tau+e^{1+\frac{1}{\upsilon T}}(1+\tau) \right)+1 \right] M_y \frac{1+\tau}{\tau}e^{4\mathcal{G}_{0,\infty}(\eta \upsilon T)}\left(1+ \frac{1}{2\upsilon T} \right)
    \end{array}
    \end{equation*}
    is a positive function (since $0 < \tau < \frac{1}{2}$ by Proposition~\ref{Prop:relSepUnif}), monotonically decreasing in $T$.
    \end{proof}
    
    

\begin{remark}
In Lemma~\ref{lem:UppBdlimN1}, the lower bound $N_1\geq 2$ in Regime 2 was chosen for brevity of notations only. Similar arguments can be employed with $N_1=1$, yielding a slightly different upper bound in \eqref{eq:RelErrDeltaInfty}.
\end{remark}

\begin{remark}
The similarity between \eqref{eq:limN1} and \eqref{eq:RelErrDeltaInfty} is not surprising. As the proof of Lemma~\ref{lem:UppBdlimN1} shows, the bound on the limit depends on the asymptotic behavior of the product $\Delta N_1$, and we have $\Delta N_1 \to \infty $ in both Regime 1 and Regime 2.
\end{remark}

\begin{corollary}\label{Cor:UnifBoundThreeRegimes}
Consider $\Phi(\Delta, N_1)$ defined in \eqref{eq:PHIdef}.
For the regimes in \eqref{Eq:Regimes}, the following statements hold.
\begin{itemize}
\item \underline{Regimes 1 and 2:} Let $\Delta_*>0$. There exists $\mathcal{C}_{\Delta_*}>0$ such that $\Phi(\Delta,N_1)\leq \mathcal{C}_{\Delta_*}$ on $\Delta \in [\Delta_*,\infty)$ and $N_1\in \mathbb{N}\setminus \left\{1 \right\}$.

\item \underline{Regime 3:} Given $T>0$, there exists $\mathcal{C}_{T}>0$ such that $\Phi(\Delta,N_1)\leq \mathcal{C}_{T}$ on
\begin{equation*}
\left\{(\Delta,N_1)\in (0,\infty)\times \mathbb{N} \ ; \ \Delta N_1 =T \right\}.
\end{equation*}
\end{itemize}
\end{corollary}
\begin{proof}
\underline{Regimes 1 and 2:} For Regime 1, Lemma~\ref{lem:UppBdlimN1} shows that $\limsup_{N_1\to \infty}\Phi(\Delta,N_1)$ is bounded uniformly in $\Delta \in [\Delta_*,\infty)$, which leaves only finitely many cases (for finite values of $N_1$) to be bounded uniformly in $\Delta \in [\Delta_*,\infty)$. On the other hand, for a fixed $N_1$, the arguments in the proof of Lemma~\ref{lem:UppBdlimN1} demonstrate that any upper bound on $\Phi(\Delta_*,N_1)$ is also an upper bound on $\Phi(\Delta,N_1)$, since $e^{-\Delta x} \leq e^{-\Delta_* x}$ for all positive $x$. Combining these facts together yields the result.
For Regime 2, similar arguments hold.

\underline{Regime 3:} The claim follows directly from \eqref{eq:limN1DeltaConst} and from arguments similar to those in the other regimes. 
\end{proof}


\begin{proof}[Proof of Theorem~\ref{thm:unifbdd}]
The proof is similar for all regimes. Hence, for simplicity, we only consider Regimes 1 and 2. Given $\mathcal{C}_{\Delta_*}$ in  Corollary~\ref{Cor:UnifBoundThreeRegimes}, we can choose $\varepsilon_{\Delta_*}>0$ such that $\mathcal{C}_{\Delta_*} \cdot \varepsilon_{\Delta_*}<1$. Then, recalling \eqref{def:RatioTerm} and \eqref{eq:Ssum}, for any $\varepsilon \in [0,\varepsilon_{\Delta_*}]$ we obtain 
\begin{equation*}
\begin{array}{lll}
&\max_{ n \in [N_1]}\mathcal{R}_n  =\max_{n \in [N_1] }\max_{z\in \partial B(\phi_n,\rho_n)}\frac{\left|\bar{q}(z)-\bar{p}(z) \right|}{\left|\bar{p}(z) \right|}\leq  \varepsilon_{\Delta_*}\Phi(\Delta,N_1)\leq \varepsilon_{\Delta_*} \cdot \mathcal{C}_{\Delta_* }<1.
\end{array}
\end{equation*}
Hence, for any $(n,N_1)\in \mathcal{N}_{1}$ and $z\in \partial B(\phi_n,\tau_n)$, we have $|\bar{q}(z)-\bar{p}(z)|<|\bar{p}(z)|\leq |\bar{p}(z)|+|\bar{q}(z)|$. By Rouche's theorem \cite[Chapter 4]{stein2010complex}, we conclude that the perturbed Prony polynomial $\bar{q}(z)$, which corresponds to $\varepsilon$, has a unique root $\hat{\phi}_n^{Pr}$ such that $\left|\hat{\phi}_n^{Pr}-\phi_n \right|\leq \tau_n = \tau \phi_n$.
\end{proof}

\begin{proof}[Proof of Corollary~\ref{cor:phinPrreal}]
Given the disk $B(\phi_n,\tau_n)$, Rouche's theorem \cite[Chapter 4]{stein2010complex} guarantees that $\bar p(z)$ and $\bar q(z)$ have the same number of roots, counted with their multiplicity, within the disk. Therefore, since for any choice of $\tau$ according to Proposition~\ref{Prop:relSepUnif}, the real roots $\phi_n$ of the polynomial $\bar p(z)$ are sufficiently separated and only one of them is contained within the disk $B(\phi_n,\tau_n)$, also $\bar q(z)$ has a unique root within the disk. The polynomial $\bar q(z)$ has real coefficients, and hence its roots are either real or complex conjugate. Since the disk $B(\phi_n,\tau_n)$ is symmetric around the real axis and must contain a single root of $\bar q(z)$, this root is necessarily real. Since $n \in [N_1]$ is arbitrary, the result follows.
\end{proof}





\subsection{Proof of Theorem~\ref{thm:unifbdd11}}\label{Append:Pf13}

We fix some $\Delta_*>0 $ and $\eta\in (0,1)$, and we aim to upper bound the asymptotic behavior of the condition numbers $\left|\mathcal{K}_{\lambda}^{Pr}(n)\right|$ associated with recovering $\left\{ \lambda_n \right\}_{n=1}^{\lfloor \eta N_1 \rfloor}$ via Prony's method, as well as the discrepancies $\left|\phi_n -\hat{\phi}_n^{Pr} \right|$ for $n\in [N_1]$.  We further assume that $\varepsilon \in [0,\varepsilon_{\Delta_*}]$, for Regimes 1 and 2, and that $\varepsilon\in [0,\varepsilon_{T}]$, for Regime 3, so that the conclusions of Theorem~\ref{thm:unifbdd} hold. For brevity, we subsequently denote $\hat{\phi}_n^{Pr}$  and $\hat{\lambda}_n^{Pr}$ by $\hat{\phi}_n$ and $\hat{\lambda}_n$, respectively.

\begin{lemma}\label{lem:P1P2N1infy}
Let $(n,N_1)\in \mathcal{N}_{1}$ and denote
\begin{equation}\label{eq:Thetadef}
\begin{array}{lll}
&\Theta_{n,N_1} := (-1)^{N_1+1}\left(\prod_{k=1}^{N_1} y_k\right)  \left(\prod_{1\leq s<t\leq N_1}\left(\phi_{t}-\phi_s \right)^2\right) \left(\prod_{\substack{s \in [N_1]\\s\neq n}}( \phi_s-\hat{\phi}_n) \right),\\
&\mathfrak{M}^{(n)}_{N_1+1-\gamma^c}\left(\hat{\phi}_n \right) := \mathfrak{M}_{N_1+1-\gamma^c}\left(\hat{\phi}_n, \phi_1,\dots,\phi_{n-1},\phi_{n+1},\dots, \phi_{N_1}\right),\\
&\mathfrak{N}^{(n)}_{N_1+1-\beta^c} := \mathfrak{N}_{N_1+1-\beta^c}\left(\phi_1,\dots , \phi_{n-1},\phi_{n+1},\dots, \phi_{N_1} \right),
\end{array}
\end{equation}
where $\mathfrak{M}_{N_1+1-\gamma^c}$ and $\mathfrak{N}_{N_1+1-\beta^c}$ are defined as in \eqref{eq:CompundExpansionNew3}, and
\begin{equation}\label{eq:CalACAl}
\begin{array}{lll}
\mathcal{P}_{n,N_1}^{(1)} &:= 1 - \varepsilon  \sum_{\gamma \in \mathcal{Q}_{N_1}^{N_1+1}}\sum_{\beta \in \mathcal{Q}^{N_1+1}_{N_1}: 1\in \beta} \sum_{\substack{\omega \in \mathcal{Q}^{N_1}_{N_1-1}\\ \omega^c\neq n}} \left\{(-1)^{\sum_{q\in \gamma}q+\sum_{q\in \beta}q} \,\, \frac{y_{N_1+1}}{\Theta_{n,N_1}}\phi_{N_1+1}^{\beta^c+\gamma^c-3}\right.\\
&\left.\times \left(\prod_{s=1}^{N_1-1}y_{\omega_{s}} \right) \hat{\phi}_n^{\mathrm{a}}\left(\prod_{j=1}^{N_1-1}\phi_{\omega_j}^{\mathrm{b}}\right) \left(\prod_{\substack{s\in [N_1-1],\\\omega_s\neq n}}( \phi_{\omega_s}-\hat{\phi}_n) \right)\right.
\\
&\left.\times \mathfrak{M}_{N_1+1-\gamma^c}(\hat{\phi}_n,\phi_{\omega})\mathfrak{N}_{N_1+1-\beta^c}(\phi_{\omega}) \left(\prod_{1\leq s< t\leq N_1-1}(\phi_{\omega_t}-\phi_{\omega_s})^2 \right) \right\}\\
&=:1-\varepsilon \overline{\mathcal{P}}^{(1)}_{n,N_1} ,\\[2mm]
\mathcal{P}_{n,N_1}^{(2)} &:=  \sum_{\gamma \in \mathcal{Q}_{N_1}^{N_1+1}}\sum_{\beta \in \mathcal{Q}^{N_1+1}_{N_1}: 1\in \beta} \left\{(-1)^{\sum_{q\in \gamma}q+\sum_{q\in \beta}q} \, \, \frac{y_{N_1+1}}{\Theta_{n,N_1}}\phi_{N_1+1}^{\beta^c+\gamma^c-3} \right.\\
& \left. \times \left(\prod_{\substack{s \in [N_1]\\s\neq n}}y_{s} \right) \hat{\phi}_n^{\mathrm{a}}\left( \prod_{\substack{s \in [N_1]\\s\neq n}}\phi_s^{\mathrm{b}}\right) \left(\prod_{\substack{s \in [N_1]\\s\neq n}} ( \phi_{s}-\hat{\phi}_n)\right)\mathfrak{M}^{(n)}_{N_1+1-\gamma^c}\left( \hat{\phi}_n\right)\mathfrak{N}^{(n)}_{N_1+1-\beta^c}\right.
\\
&\left.\times \left(\prod_{\substack{1\leq s< t\leq N_1\\s,t\neq n}}(\phi_{t}-\phi_{s})^2 \right) \right\},
\end{array}
\end{equation}
where we adopt the notation in \eqref{eq:GammaParam1}. Then,
\begin{equation}\label{eq:PProduct}
\left|\hat{\phi}_n-\phi_n \right|\left|\mathcal{P}_{n,N_1}^{(1)} \right|= \varepsilon \left|\mathcal{P}_{n,N_1}^{(2)} \right|.
\end{equation}
\end{lemma}

\begin{proof}

    We substitute $z= \hat{\phi}_n$ into both sides of \eqref{eq:CompundExpansionNew2}, take into account that $\bar{q}(\hat{\phi}_n)=0$ and write the unperturbed Prony polynomial $\bar{p}(\hat{\phi}_n)$ in the form \eqref{eq:PbarExplicitExpr1}, thereby obtaining
    \begin{equation}\label{eq:FirstEq}
    \begin{array}{lll}
    &(-1)^{N_1+1}\left(\prod_{k=1}^{N_1} y_k\right) \left(\prod_{1\leq s<t\leq N_1}\left(\phi_{t}-\phi_s \right)^2\right)\cdot \left(\prod_{s=1}^{N_1}( \phi_s-\hat{\phi}_n) \right)\\
    &\hspace{20mm}=\varepsilon  \sum_{\gamma \in \mathcal{Q}_{N_1}^{N_1+1}}\sum_{\beta \in \mathcal{Q}^{N_1+1}_{N_1}: 1\in \beta}\sum_{\omega\in \mathcal{Q}_{N_1-1}^{N_1}} \left\{(-1)^{\sum_{q\in \gamma}q+\sum_{q\in \beta}q}y_{N_1+1} \right.\\
    &\hspace{30mm}\times \phi_{N_1+1}^{\beta^c+\gamma^c-3}\left(\prod_{s=1}^{N_1-1}y_{\omega_{s}} \right) \hat{\phi}_n^{\mathrm{a}}\left(\prod_{j=1}^{N_1-1}\phi_{\omega_j}^{\mathrm{b}}\right)\left(\prod_{s=1}^{N_1-1}( \phi_{\omega_s}-\hat{\phi}_n) \right)\\
    &\hspace{30mm}\left. \times \mathfrak{M}_{N_1+1-\gamma^c}(\hat{\phi}_n,\phi_{\omega}) \mathfrak{N}_{N_1+1-\beta^c}(\phi_{\omega})\left(\prod_{1\leq s< t\leq N_1-1}(\phi_{\omega_t}-\phi_{\omega_s})^2 \right) \right\}.
    \end{array}
    \end{equation}
    Isolating $\phi_n-\hat{\phi}_n$ on the left-hand side of \eqref{eq:FirstEq}, we have
    \begin{equation}\label{eq:FirstEq1}
    \begin{array}{lll}
    \phi_n-\hat{\phi}_n&=\varepsilon  \sum_{\gamma \in \mathcal{Q}_{N_1}^{N_1+1}}\sum_{\beta \in \mathcal{Q}^{N_1+1}_{N_1}: 1\in \beta}\sum_{\omega\in \mathcal{Q}_{N_1-1}^{N_1}} \left\{(-1)^{\sum_{q\in \gamma}q+\sum_{q\in \beta}q}\frac{y_{N_1+1}}{\Theta_{n,N_1}} \right.\\
    &\hspace{10mm}\times \phi_{N_1+1}^{\beta^c+\gamma^c-3}\left(\prod_{s=1}^{N_1-1}y_{\omega_{s}} \right) \hat{\phi}_n^{\mathrm{a}}\left(\prod_{j=1}^{N_1-1}\phi_{\omega_j}^{\mathrm{b}}\right)\left(\prod_{s=1}^{N_1-1}( \phi_{\omega_s}-\hat{\phi}_n) \right)\\
    &\hspace{10mm}\left. \times \mathfrak{M}_{N_1+1-\gamma^c}( \hat{\phi}_n,\phi_{\omega}) \mathfrak{N}_{N_1+1-\beta^c}(\phi_{\omega})\left(\prod_{1\leq s< t\leq N_1-1}(\phi_{\omega_t}-\phi_{\omega_s})^2 \right) \right\}.
    \end{array}
    \end{equation}
    Next, considering the right-hand side of \eqref{eq:FirstEq1} and the product $\prod_{s=1}^{N_1-1}( \phi_{\omega_s}-\hat{\phi}_n)$ appearing therein, we distinguish between two cases. If $\omega^c=n$, then the product does not include the term $\phi_n-\hat{\phi}_n$; in the complementary case, whenever $\omega^c \neq n$, the product includes the term $\phi_n-\hat{\phi}_n$. All terms belonging to the second case are moved to the left-hand side. Taking the absolute value of both sides of the obtained equality then yields \eqref{eq:PProduct}. 
\end{proof}

From \eqref{eq:PProduct}, we see that 
\begin{equation}\label{eq:maxratioterm}
\begin{array}{lll}
\max_{ n \in [N_1]}\left|\phi_n -\hat{\phi}_n \right| = \varepsilon\cdot \max_{ n \in [N_1]}\left( \left|\mathcal{P}_{n,N_1}^{(1)} \right|^{-1}\left|\mathcal{P}_{n,N_1}^{(2)} \right|\right) =: \varepsilon \cdot \max_{n\in [N_1] }\left|\mathcal{K}_{\lambda}^{Pr}(n)\right| ,
\end{array}
\end{equation}
provided $\left|\mathcal{P}_{n,N_1}^{(1)} \right|\neq 0$. We bound $\left|\mathcal{P}_{n,N_1}^{(1)} \right|$ and $\left|\mathcal{P}_{n,N_1}^{(2)} \right|$ in the next lemma.
\begin{lemma}\label{eq:CAlP12}
Let $\Delta_*>0$  be fixed. Then,
\begin{enumerate}

    \item Fix any $\kappa \in (0,1)$. For Regimes 1 and 2, there exists $\varepsilon_{\Delta_*,\kappa}\in (0,\varepsilon_{\Delta_* }]$ such that, for $\varepsilon \in [0,\varepsilon_{\Delta_*,\kappa}]$,
     \begin{equation}\label{eq:P1Callower}
    \begin{array}{lll}
    \min_{n\in [N_1 ]}\left|\mathcal{P}_{n,N_1}^{(1)}\right| = \min_{n\in [N_1]} \mathcal{P}_{n,N_1}^{(1)} \geq \kappa
    \end{array}
    \end{equation}
    holds for all $\Delta \in [\Delta_*,\infty)$ and $N_1\in \mathbb{N}$.

    For Regime 3, there exists $\varepsilon_{T,\kappa}\in (0,\varepsilon_{T }]$ such that, for $\varepsilon \in [0,\varepsilon_{T,\kappa}] $, \eqref{eq:P1Callower} holds for all $\Delta \in [\Delta_*,\infty)$ and $N_1\in \mathbb{N}$.

    \item For some constants $\mathcal{D}(\Delta_*)>0$ and $\mathcal{D}(T)>0$, $\mathcal{P}^{(2)}_{n,N_1}$ satisfies 

    \begin{equation*}
    \begin{array}{lll}
    & \left|\mathcal{P}_{n,N_1}^{(2)} \right| \leq \frac{\phi_n\left( \prod_{s\neq n}\phi_s^2\right)}{\prod_{s\neq n}|\phi_{s}-\phi_{n}|^2 } \cdot \begin{cases}
\mathcal{D}(\Delta_*),& \text{Regimes 1 and 2},\\
\mathcal{D}(T), & \text{Regime 3}.
\end{cases}
    \end{array}
    \end{equation*}

\end{enumerate}
\end{lemma}

\begin{proof}
    Both claims follow from arguments similar to the proof of Lemma~\ref{lem:UppBdlimN1}. For brevity, we only provide an explicit proof of the second claim, which provides an upper bound for $\mathcal{P}^{(2)}_{n,N_1}$; analogous arguments can be used to provide an upper bound for $\overline{\mathcal{P}}^{(1)}_{n,N_1}$ such that $\mathcal{P}_{n,N_1}^{(1)} = 1-\varepsilon \overline{\mathcal{P}}^{(1)}_{n,N_1}$, and hence, for small enough $\varepsilon$, a lower bound for $\mathcal{P}_{n,N_1}^{(1)}$.
    
    Applying an absolute value to $\mathcal{P}^{(2)}_{n,N_1}$, employing the triangle inequality, canceling out terms with $\Theta_{n,N_1}$ and employing Assumption~\ref{assump:yassump}, we have 
    \begin{equation*}
    \begin{array}{lll}
    |\mathcal{P}^{(2)}_{n,N_1}| \leq M_y \sum_{\gamma\in \mathcal{Q}^{N_1+1}_{N_1}}\sum_{\beta \in \mathcal{Q}_{N_1}^{N_1+1}: 1\in \beta}\mathfrak{p}^{(2)}(\gamma^c,\beta^c),\vspace{0.1cm}\\
    \mathfrak{p}^{(2)}(\gamma^c,\beta^c):=\frac{\phi_{N_1+1}^{\beta^c+\gamma^c-3}|\hat{\phi}_n|^{\mathrm{a}}\left( \prod_{s\neq n}\phi_s^{\mathrm{b}}\right)\mathfrak{M}^{(n)}_{N_1+1-\gamma^c}\left(\left|\hat{\phi}_n \right|\right)\mathfrak{N}^{(n)}_{N_1+1-\beta^c}}{\prod_{s\neq n}|\phi_{s}-\phi_{n}|^2 }.
    \end{array}
    \end{equation*}
    Proceeding as in Lemma~\ref{lem:UppBdlimN1}, it can be shown that 
    \begin{equation*}
    \begin{array}{lll}
    &\mathfrak{p}^{(2)}(N_1+1,2) = \frac{\phi_{N_1+1}^{N_1}}{|\hat{\phi}_n|\prod_{s\neq 
    n}\phi_s}\mathfrak{p}^{(2)}(1,2),\ \quad \mathfrak{p}^{(2)}(1,N_1+1)  = \frac{\phi_{N_1+1}^{N_1-1}}{\prod_{s\neq n}\phi_s}\mathfrak{p}^{(2)}(1,2),\\
    &\mathfrak{p}^{(2)}(N_1+1,N_1+1)  = \frac{\phi_{N_1+1}}{|\hat{\phi}_n|}\left(\frac{\phi_{N_1+1}^{N_1-1}}{\prod_{s\neq n}\phi_s}\right)^2\mathfrak{p}^{(2)}(1,2),\\
    &\mathfrak{p}^{(2)}(\gamma^c,2) = \frac{\mathfrak{M}^{(n)}_{N_1+1-\gamma^c}\left(|\hat{\phi}_n |\right)\phi_{N_1+1}^{\gamma^c-1}}{|\hat{\phi}_n|\prod_{s\neq n}\phi_s}\mathfrak{p}^{(2)}(1,2), \quad \gamma^c\in [N_1]\setminus \left\{1 \right\},\\
    &\mathfrak{p}^{(2)}(\gamma^c,N_1+1) = \frac{\phi_{N_1+1}^{N_1-1}}{\prod_{s\neq n}\phi_s}\mathfrak{p}^{(2)}(\gamma^c,2), \quad \gamma^c \in [N_1]\setminus \left\{1\right\},\\
    &\mathfrak{p}^{(2)}(1,\beta^c) = \frac{\mathfrak{N}^{(n)}_{N_1+1-\beta^c}\phi_{N_1+1}^{\beta^c-2}}{\prod_{s\neq n}\phi_s}\mathfrak{p}^{(2)}(1,2),\quad \beta^c \in [N_2]\setminus \left\{1,2 \right\},\\
    &\mathfrak{p}^{(2)}(N_1+1,\beta^c) = \frac{1}{|\hat{\phi}_n|}\frac{\phi_{N_1+1}^{N_1}}{\prod_{s\neq n}\phi_s}\mathfrak{p}^{(2)}(1,\beta^c),\quad \beta^c\in [N_1]\setminus \left\{1,2 \right\},\\
    &\mathfrak{p}^{(2)}(\gamma^c,\beta^c) = \frac{\mathfrak{M}^{(n)}_{N_1+1-\gamma^c}\left(|\hat{\phi}_n |\right)\phi_{N_1+1}^{\gamma^c-1}}{|\hat{\phi}_n|\prod_{s\neq n}\phi_s}\frac{\mathfrak{N}^{(n)}_{N_1+1-\beta^c}\phi_{N_1+1}^{\beta^c-2}}{\prod_{s\neq n}\phi_s}\mathfrak{p}^{(2)
    }(1,2), \quad \gamma^c\in [N_1]\setminus \left\{1\right\},\ \beta^c\in [N_1]\setminus \left\{1,2 \right\}.
    \end{array}
    \end{equation*}
    Then, the arguments employed in the proof of Lemma~\ref{lem:UppBdlimN1} (including Proposition~\ref{prop:frakxfraky}) yield
    \begin{equation}\label{eq:p212bd}
    \begin{array}{lll}
    |\mathcal{P}^{(2)}_{n,N_1}| \leq M_y \Omega^{(2)}_{n,N_1}\mathfrak{p}^{(2)}(1,2) = M_y \Omega^{(2)}_{n,N_1} \frac{|\hat{\phi}_n|\left( \prod_{s\neq n}\phi_s^2\right)}{\prod_{s\neq n}|\phi_{s}-\phi_{n}|^2 }\leq M_y(1+\tau)\Omega^{(2)}_{n,N_1} \frac{\phi_n \left( \prod_{s\neq n}\phi_s^2\right)}{\prod_{s\neq n}|\phi_{s}-\phi_{n}|^2 },
    \end{array}
    \end{equation}
    where
    \begin{equation*}
    \begin{array}{lll}
    \Omega^{(2)}_{n,N_1} =& 1 +\frac{\phi_{N_1+1}^{N_1}}{|\hat{\phi}_n|\prod_{s\neq 
    n}\phi_s}+\frac{\phi_{N_1+1}^{N_1-1}}{\prod_{s\neq n}\phi_s}+\frac{\phi_{N_1+1}}{|\hat{\phi}_n|}\left(\frac{\phi_{N_1+1}^{N_1-1}}{\prod_{s\neq n}\phi_s}\right)^2+\left(1+\frac{\phi_{N_1+1}^{N_1-1}}{\prod_{s\neq n}\phi_s} \right)\frac{1+\tau}{1-\tau}\left(\sum_{\gamma^c=2}^{N_1} \binom{N_1}{\gamma^c-1}q^{\gamma^c-1}\right)\\
    &+\left(1+ \frac{1}{|\hat{\phi}_n|}\frac{\phi_{N_1+1}^{N_1}}{\prod_{s\neq n}\phi_s} \right)\left(\sum_{\beta^c=3}^{N_1-1}\binom{N_1-1}{\beta^c-2}q^{\beta^c-2} \right) +\frac{1+\tau}{1-\tau}\left(\sum_{\gamma^c=2}^{N_1} \binom{N_1}{\gamma^c-1}q^{\gamma^c-1}\right)\left(\sum_{\beta^c=3}^{N_1-1}\binom{N_1}{\beta^c-2}q^{\beta^c-2} \right)
    \end{array}
    \end{equation*}
    and $q\in \left\{e^{-\upsilon \Delta (2N_1+1)},\frac{\upsilon T+1}{\upsilon TN_1} \right\}$, depending on the regime (see Proposition~\ref{prop:frakxfraky}). As in the proofs of Lemma~\ref{lem:UppBdlimN1} and Corollary~\ref{Cor:UnifBoundThreeRegimes}, it can be shown that $\Omega^{(2)}_{n,N_1}$ is bounded in all of the regimes, uniformly in the corresponding parameters, meaning that 
    \begin{equation*}
    M_y(1+\tau) \Omega^{(2)}_{n,N_1} \leq \begin{cases}
    \mathcal{D}(\Delta_*), &\text{Regimes 1 and 2},\\
    \mathcal{D}(T), & \text{Regime 3},
    \end{cases}
    \end{equation*}
    for some constants $\mathcal{D}(\Delta_*)>0$ and $\mathcal{D}(T)>0$. Note that the final inequality in \eqref{eq:p212bd}, as well as the upper bound on $\Omega^{(2)}_{n,N_1}$, employ the fact that $(1-\tau)\phi_n \leq \hat{\phi}_n\leq (1+\tau)\phi_n$, which follows from Theorem~\ref{thm:unifbdd}.
\end{proof}

\begin{proof}[Proof of Theorem~\ref{thm:unifbdd11}]
Recalling Lemma~\ref{eq:CAlP12}, let $\kappa \in (0,1)$ and let $\varepsilon_*= \varepsilon_{\Delta_*,\kappa}$ in Regimes 1 and 2, while $\varepsilon_* = \varepsilon_{T,\kappa}$ in Regime 3.

Then, in view of \eqref{eq:maxratioterm} and of Lemma~\ref{eq:CAlP12}, the inequality $|\hat{\phi}_n-\phi_n| \leq \varepsilon \mathcal{D}_{\phi} \,\, \frac{\phi_n \left( \prod_{s\neq n}\phi_s^2\right)}{\prod_{s\neq n}|\phi_{s}-\phi_{n}|^2 }$ in \eqref{eq:Basic estimates} holds with
\begin{equation*}
 \mathcal{D}_{\phi} = \kappa^{-1}\begin{cases}
\mathcal{D}(\Delta_*), & \text{Regimes 1 and 2},\\
\mathcal{D}(T), & \text{Regime 3}.
\end{cases}
\end{equation*}
Concerning the inequality $|\hat{\lambda}_n-\lambda_n| \leq \varepsilon \mathcal{D}_{\lambda} \,\, \frac{1}{\Delta} \frac{\left( \prod_{s\neq n}\phi_s^2\right)}{\prod_{s\neq n}|\phi_{s}-\phi_{n}|^2 }$ in \eqref{eq:Basic estimates}, let $\psi(z) = -\Delta^{-1} \ln(z)$, $z>0$. Then, by Lagrange's theorem, there exists $\xi$ in the interval between $\phi_n$ and $\hat \phi_n$ such that, given $\psi'(\xi) = -\frac{1}{\Delta \xi}$, we have
\begin{equation}\label{eq:lambdadiffN1}
\begin{array}{lll}
\left|\hat{\lambda}_n-\lambda_n \right| &= \left|\psi\left(\hat{\phi}_n \right)- \psi\left(\phi_n \right) \right| = \frac{1}{\Delta \xi}\left|\hat{\phi}_n - \phi_n\right|\overset{\eqref{eq:Basic estimates}}{\leq} \varepsilon \frac{ \mathcal{D}_{\phi}}{\xi} \frac{1}{\Delta} \frac{\phi_n \left( \prod_{s\neq n}\phi_s^2\right)}{\prod_{s\neq n}|\phi_{s}-\phi_{n}|^2 }\leq \varepsilon \frac{\mathcal{D}_{\phi}}{1-\tau} \frac{1}{\Delta} \frac{ \left( \prod_{s\neq n}\phi_s^2\right)}{\prod_{s\neq n}|\phi_{s}-\phi_{n}|^2 },
\end{array}
\end{equation}
where we use the fact that $\xi\geq (1-\tau)\phi_n$, whence $\xi^{-1}\phi_n\leq (1-\tau)^{-1}$. Therefore, the inequality for $|\hat{\lambda}_n-\lambda_n|$ in \eqref{eq:Basic estimates} follows, with $\mathcal{D}_{\lambda} = \frac{\mathcal{D}_{\phi}}{1-\tau}$. 

Let us now consider the asymptotic behavior.
By arguments of Lemma~\ref{lem:UppBdlimN1}, and in particular \eqref{eq:omegacoptions},
\begin{equation}\label{eq:fracAsymp}
\begin{array}{lll}
\frac{\prod_{s\neq n}\phi_s^2}{\prod_{s\neq n |\phi_s-\phi_n|^2}} &= e^{-2\Delta \sum_{s=n+1}^{N_1}\left(\lambda_s-\lambda_n \right)}  \cdot \begin{cases}
e^{2\theta^3_{1,\Delta ,N_1}}, & n=1\\
e^{2\theta^3_{n,\Delta,N_1}+2\theta^2_{n,\Delta}}, & n\neq 1,N_1\\
e^{2\theta^2_{N_1},\Delta}, & n=N_1
\end{cases}\\
& \leq e^{-2\upsilon \Delta \Psi(n; N_1)}  \cdot \begin{cases}
e^{2\theta^3_{1,\Delta ,N_1}},& n=1\\
e^{2\theta^3_{n,\Delta,N_1}+2\theta^2_{n,\Delta}},& n\neq 1,N_1\\
e^{2\theta^2_{N_1},\Delta},& n=N_1
\end{cases}
\end{array}
\end{equation}
where the last inequality follows from Proposition~\ref{prop:EigDiffBound} and \eqref{eq:CalSdef}, and we recall that $\Psi(N_1; N_1)=0$, by convention.
We assess the asymptotic behavior separately for the three regimes.

\smallskip
\underline{Regime 1:} By Corollary~\ref{Cor:ThetaAssymp}, $\theta^2_{n,\Delta}$ and $\theta^3_{n,\Delta,N_1}$ are $O(1)$, uniformly in $(n,N_1)\in \mathcal{N}_1$ and $\Delta \in [\Delta_*,\infty)$. Therefore, we have 
\begin{equation*}
\begin{array}{lll}
\left|\hat{\phi}_n-\phi_n \right|\leq \varepsilon \mathcal{D}_{\phi}\phi_n e^{-2\upsilon \Delta \Psi(n; N_1)+O(1)} = \varepsilon \mathcal{D}_{\phi}e^{- \Delta \left(2\upsilon\Psi(n; N_1)+\lambda_n\right)+O(1)},\quad (n,N_1)\in \mathcal{N}_1.
\end{array}
\end{equation*}
In particular, employing Proposition~\ref{Prop:VFrakAsymp}, when $(n,N_1)\in \mathcal{N}_{\eta}$ with $N_1\gg 1$, we obtain the first inequality in \eqref{eq:Asymptoticphi1}:
\begin{equation}\label{eq:Reg1FinThm}
\begin{array}{lll}
\left|\hat{\phi}_n-\phi_n \right|\leq  \varepsilon \mathcal{D}_{\phi}\left(\max_{N_1\in \mathbb{N}}e^{- \Delta_* \upsilon\Psi\left(\lfloor \eta N_1 \rfloor; N_1  \right)+O(1)}\right)e^{-\Delta \Gamma^{Pr}(\eta)N_1^3} =: \varepsilon \mathcal{D}^{(1)}_{\phi}e^{-\Delta \Gamma^{Pr}(\eta)N_1^3},
\end{array}
\end{equation}
for some $\Gamma^{Pr}(\eta)>0$. Here, in particular, we used the facts that $\Psi(n; N_1)\geq \Psi(\lfloor \eta N_1 \rfloor; N_1) > 0$ and $\Psi(\lfloor \eta N_1 \rfloor; N_1) \geq a(\eta) N_1^3$ by Proposition~\ref{Prop:VFrakAsymp}, that $\Delta \geq \Delta_*$, and that $\lambda_n \geq \lambda_1$ can be absorbed into $O(1)$.

For $(n,N_1)\in \mathcal{N}_1\setminus \mathcal{N}_{\eta}$, on the other hand, we have that $2\upsilon \Psi(n; N_1) + \lambda_n \geq \lambda_n \geq \lambda_{\lfloor \eta N_1 \rfloor}$ and hence we get the second inequality in \eqref{eq:Asymptoticphi1}:
\begin{equation}\label{eq:Reg1FinThm1}
\begin{array}{lll}
\left|\hat{\phi}_n-\phi_n \right|\leq \varepsilon \mathcal{D}_{\phi}e^{- \Delta \lambda_{\lfloor \eta N_1 \rfloor }+O(1)}\leq \varepsilon \mathcal{D}_{\phi}\left(\max_{N_1\in \mathbb{N}}e^{-\frac{1}{2}\Delta_* \lambda_{\lfloor \eta N_1 \rfloor}+O(1)} \right)e^{-\Delta \Gamma^{Pr}_1(\eta)N_1^2}=: \varepsilon \mathcal{D}_{\phi}^{(2)}e^{-\Delta \Gamma^{Pr}_1(\eta)N_1^2}
\end{array}
\end{equation}
for some $\Gamma^{Pr}_1(\eta)>0$, because $\lambda_{\lfloor \eta N_1 \rfloor }\propto N_1^2$ as $N_1\to \infty$, in view of Proposition~\ref{prop:EigDiffBound}.

Finally, when $\eta \in (0,1)$, \eqref{eq:lambdadiffN1} yields, for $(n,N_1)\in \mathcal{N}_{\eta}$ and $N_1\gg 1$, the third inequality in \eqref{eq:Asymptoticphi1}:
\begin{equation*}
\begin{array}{lll}
\left|\hat{\lambda}_n-\lambda_n \right| &\leq \varepsilon \mathcal{D}_{\lambda}\frac{1}{\Delta}\left(\max_{N_1\in \mathbb{N}}e^{- \Delta_*  \upsilon\Psi\left(\lfloor \eta N_1 \rfloor; N_1  \right)+O(1)}\right)e^{-\Delta \Gamma^{Pr}(\eta)N_1^3}=: \varepsilon \frac{\mathcal{D}^{(1)}_{\lambda}}{\Delta}e^{-\Delta \Gamma^{Pr}(\eta)N_1^3}.
\end{array}
\end{equation*}

\underline{Regime 2:} By Corollary~\ref{Cor:ThetaAssymp}, $\theta^2_{n,\Delta}$ and $\theta^3_{n,\Delta,N_1}$ are $O(1)$, uniformly in $(n,N_1)\in \mathcal{N}_1$ and $\Delta \in [\Delta_*,\infty)$. Choose $N_1(\vartheta)$ such that, for all $N_1\geq N_1(\vartheta)$,
\begin{equation*}
\begin{array}{lll}
\min \left(\upsilon \Psi\left( \lfloor \eta N_1\rfloor; N_1 \right), \frac{1}{2}\lambda_{\lfloor \eta N_1\rfloor} \right)\geq \vartheta.
\end{array}
\end{equation*}
Then, for $(n,N_1)\in \mathcal{N}_\eta$ and $N_1\geq N_1(\vartheta)$
\begin{equation*}
\begin{array}{lll}
\left|\hat{\phi}_n-\phi_n \right|\leq  \varepsilon \mathcal{D}_{\phi}\left(\max_{N_1\in \mathbb{N}}e^{- \Delta_* \upsilon\Psi\left(\lfloor \eta N_1 \rfloor; N_1  \right)+O(1)}\right)e^{-\upsilon \Delta \Psi\left(\lfloor \eta N_1 \rfloor; N_1  \right)} \leq \varepsilon \mathcal{D}^{(1)}_{\phi}e^{-\vartheta \Delta},\quad \Delta\gg 0.
\end{array}
\end{equation*}
Similarly, for $(n,N_1)\in \mathcal{N}_1\setminus \mathcal{N}_{\eta}$ with $N_1\geq N_1(\vartheta)$
\begin{equation*}
\begin{array}{lll}
\left|\hat{\phi}_n-\phi_n \right|\leq \varepsilon \mathcal{D}_{\phi}e^{- \Delta \lambda_{\lfloor \eta N_1 \rfloor }+O(1)}\leq \varepsilon \mathcal{D}_{\phi}\left(\max_{N_1\in \mathbb{N}}e^{-\frac{1}{2} \Delta_* \lambda_{\lfloor \eta N_1 \rfloor}+O(1)} \right)e^{-\frac{1}{2}\Delta \lambda_{\lfloor \eta N_1 \rfloor }}\leq  \varepsilon \mathcal{D}_{\phi}^{(2)}e^{-\vartheta \Delta },\quad \Delta \gg 0.
\end{array}
\end{equation*}
Hence, the first inequality in \eqref{eq:Asymptoticphi11} holds with $\mathcal{D}^{(3)}_{\phi} = \min \left( \mathcal{D}^{(1)}_{\phi},\mathcal{D}^{(2)}_{\phi}\right)$. To get the second inequality in \eqref{eq:Asymptoticphi11}, we employ \eqref{eq:lambdadiffN1} to obtain, for $(n,N_1)\in \mathcal{N}_{\eta}$ with $N_1 \geq N_1(\vartheta)$,
\begin{equation*}
\begin{array}{lll}
\left|\hat{\lambda}_n-\lambda_n \right| &\leq \varepsilon \mathcal{D}_{\lambda}\frac{1}{\Delta}\left(\max_{N_1\in \mathbb{N}}e^{- \Delta_* \upsilon\Psi\left(\lfloor \eta N_1 \rfloor; N_1  \right)+O(1)}\right)e^{-\upsilon \Delta \Psi\left(\lfloor \eta N_1 \rfloor; N_1 \right)}\leq \varepsilon \frac{\mathcal{D}^{(1)}_{\lambda}}{\Delta}e^{-\vartheta \Delta},\quad \Delta \gg 0.
\end{array}
\end{equation*}

\underline{Regime 3:} Let $\Delta N_1=T>0$. Consider first $n\in [\lfloor \eta N_1 \rfloor ]$, $N_1\gg1$. Then, Corollary~\ref{Cor:ThetaAssymp} shows that $\theta^2_{n,\Delta}$ and $\theta^3_{n,\Delta,N_1}$ are $O(N_1)$, uniformly in $N_1\in \mathbb{N}$. Hence, for $(n,N_1)\in \mathcal{N}_{\eta}$ with $N_1\gg 1$, we obtain the first inequality in \eqref{eq:Asymptoticphi111}:
\begin{equation}\label{eq:Reg3etaN1}
\begin{array}{lll}
\left|\hat{\phi}_n-\phi_n \right|&\leq  \varepsilon \mathcal{D}_{\phi}e^{- \frac{T}{N_1} \left(2\upsilon\Psi(n; N_1)+\lambda_n\right)+O(N_1)} \\
&\leq \varepsilon \mathcal{D}_{\phi}\left(\max_{N_1\in \mathbb{N}}e^{-T \frac{\upsilon \Psi\left(\lfloor \eta N_1 \rfloor; N_1 \right)}{N_1}+O(N_1)} \right)e^{-\Gamma^{Pr}_2(T,\eta)N_1^2} =: \varepsilon \mathcal{D}_{\phi}^{(4)} e^{-\Gamma^{Pr}_2(T,\eta)N_1^2}
\end{array}
\end{equation}
for some $\Gamma^{Pr}_2(T,\eta)>0$. Here we used Proposition~\ref{Prop:VFrakAsymp}, which shows that $\frac{\Psi\left(\lfloor \eta N_1 \rfloor; N_1 \right)}{N_1}\propto N_1^2$ for $N_1\gg 1$; hence, the maximum in \eqref{eq:Reg3etaN1} is finite. For $(n,N_1)\in \mathcal{N}_1\setminus \mathcal{N}_{\eta}$,  Corollary~\ref{Cor:ThetaAssymp} shows that $\theta^2_{n,\Delta}$ and $\theta^3_{n,\Delta,N_1}$ are $O(1)$, uniformly in $N_1\in \mathbb{N}$, and hence we obtain the second inequality in \eqref{eq:Asymptoticphi111}:
\begin{equation}\label{eq:Reg3notetaN1}
\begin{array}{lll}
\left|\hat{\phi}_n-\phi_n \right|&\leq  \varepsilon \mathcal{D}_{\phi}e^{- \frac{T}{N_1} \left(2\upsilon\Psi(n; N_1)+\lambda_n\right)+O(1)} \\
&\leq \varepsilon \mathcal{D}_{\phi}\left(\max_{N_1\in \mathbb{N}}e^{-T \frac{\lambda_{\lfloor \eta N_1 \rfloor}}{2N_1}+O(1)} \right)e^{-\Gamma^{Pr}_3(T,\eta)N_1}=:\varepsilon \mathcal{D}_{\phi}^{(5)} e^{-\Gamma^{Pr}_3(T,\eta)N_1}
\end{array}
\end{equation}
for some $\Gamma^{Pr}_3(T,\eta)>0$. Here we used the fact that $\frac{\lambda_{\lfloor \eta N_1\rfloor }}{N_1}\geq \frac{\upsilon \left(\lfloor \eta N_1 \rfloor ^2 -1 \right)+\lambda_1}{N_1}\propto N_1$ as $N_1 \to \infty$, whence the maximum in \eqref{eq:Reg3notetaN1} is finite. Finally, for $(n,N_1)\in \mathcal{N}_{\eta}$ with $N_1\gg 1$, by \eqref{eq:lambdadiffN1}, we obtain the third inequality in \eqref{eq:Asymptoticphi111}:
\begin{equation*}
\begin{array}{lll}
\left|\hat{\lambda}_n-\lambda_n \right| &\leq \varepsilon \mathcal{D}_{\lambda}\frac{1}{\Delta}\left(\max_{N_1\in \mathbb{N}}e^{- T \frac{\upsilon\Psi\left(\lfloor \eta N_1 \rfloor; N_1  \right)}{N_1}+O(N_1)}\right)e^{-\Gamma^{Pr}_2(T,\eta)N_1^2}=:\varepsilon \frac{\mathcal{D}_{\lambda}^{(2)}}{\Delta}e^{-\Gamma^{Pr}_2(T,\eta)N_1^2}.
\end{array}
\end{equation*}
This completes the proof of Theorem~\ref{thm:unifbdd11}.
\end{proof}

\subsection{Proof of Theorem~\ref{thm:PronyAmplitudeCondition} }\label{Sec:PronyAmplitProof}
 Recall the expression $\operatorname{col}\left\{ \hat{y}_n^{Pr} - y_n\right\}_{n=1}^{N_1} = \hat{V}^{-1}\cdot \operatorname{col} \left\{\varepsilon y_{N_1+1}e^{-\Delta \lambda_{N_1+1} k} \right\}_{k=0}^{N_1-1} - \left(I_{N_1}-\hat{V}^{-1}V \right)\cdot \operatorname{col}\left\{y_k \right\}_{k=1}^{N_1}$ in \eqref{eq:PronyYdiff}-\eqref{eq:VDMmatDef}. The proof of Theorem~\ref{thm:PronyAmplitudeCondition} is based on upper bounding the two terms  
\begin{equation}\label{eq:twoterms}
\begin{array}{lll}
\left[\hat{V}^{-1}\cdot \operatorname{col} \left\{\varepsilon y_{N_1+1}e^{-\Delta \lambda_{N_1+1}k} \right\}_{k=0}^{N_1-1}\right]_n,\quad \left[\left(I_{N_1}-\hat{V}^{-1}V \right)\cdot \operatorname{col}\left\{y_k \right\}_{k=1}^{N_1}\right]_n, \quad (n,N_1)\in \mathcal{N}_{\eta},
\end{array}
\end{equation}
separately, so that an upper bound for $\left|\hat{y}_n^{Pr} - y_n\right|$ can be obtained by summing those two upper bounds, and then showing that each of them has the required asymptotic behavior in all the regimes in \eqref{Eq:Regimes}. Initially, we consider $\varepsilon \in (0,\varepsilon_*]$ where $\varepsilon_*>0$ is given in Theorem~\ref{thm:unifbdd11}. The value of $\varepsilon_*>0$ will be further constrained below. For brevity, we subsequently denote $\hat{\phi}_n^{Pr}$ by $\hat{\phi}_n$.

We begin by considering the first term in \eqref{eq:twoterms}. Introduce
\begin{equation}\label{eq:firstterm}
\begin{array}{lll}
 \mathcal{I}_n^{(1)}&: = \left|\left[\hat{V}^{-1}\cdot \operatorname{col} \left\{\varepsilon y_{N_1+1}e^{-\lambda_{N_1+1} k \Delta} \right\}_{k=0}^{N_1-1}\right]_n \right|=\varepsilon |y_{N_1+1}|  \prod_{s\neq n} \frac{\left|\phi_{N_1+1}-\hat{\phi}_s \right|}{\left|\hat{\phi}_s - \hat{\phi}_n \right|},\quad (n,N_1)\in \mathcal{N}_{\eta}.
\end{array}
\end{equation}
where the second equality is obtained by performing the multiplication and using the explicit form of $\hat{V}^{-1}$ in \eqref{eq:VDMmatDef} and its relation to Lagrange interpolation. In fact, since $e^{-\lambda_{N_1+1} k \Delta} = \phi_{N_1+1}^k$, we have
\begin{equation*}
\begin{array}{lll}
\left[\hat{V}^{-1}\cdot \operatorname{col} \left\{\varepsilon y_{N_1+1}e^{-\lambda_{N_1+1} k \Delta} \right\}_{k=0}^{N_1-1}\right]_n & = \varepsilon y_{N_1+1} \sum_{j=0}^{N_1-1} \hat l_{n,j} \phi_{N_1+1}^j = \varepsilon y_{N_1+1} \hat L_n(\phi_{N_1+1})\\ &= \varepsilon y_{N_1+1} \prod_{s \neq n} \frac{\phi_{N_1+1}-\hat{\phi}_s}{\hat{\phi}_s-\hat{\phi}_n}.
\end{array}
\end{equation*}

\begin{proposition}\label{prop:RatioConstituents}
There exist constants $\mathcal{C}_i>0$, $i=1,2$, such that, in all the regimes in \eqref{Eq:Regimes}, 
\begin{equation*}
\begin{array}{lll}
&\frac{\prod_{s\neq n}\phi_s^2}{\prod_{s\neq n}\left|\phi_s - \phi_n \right|^2} \leq \mathcal{C}_1,\quad (n,N_1)\in \mathcal{N}_{1}, \vspace{0.1cm}\\
&\frac{|\phi_{N_1+1}-\hat{\phi}_n|}{\left|\phi_{N_1+1}-\phi_n\right|} \leq 1+\mathcal{C}_2 \varepsilon ,\quad (n,N_1)\in \mathcal{N}_1,\\
& \frac{|\phi_s - \hat{\phi}_s |}{|\phi_s - \phi_n |} \leq \mathcal{C}_2 \varepsilon,\quad (s,N_1),(n,N_1)\in \mathcal{N}_1,\ s\neq n.
\end{array}
\end{equation*}
\end{proposition}
\begin{proof}
The expression $\frac{\prod_{s\neq n}\phi_s^2}{\prod_{s\neq n}\left|\phi_s - \phi_n \right|^2}$ can be upper bounded as in \eqref{eq:fracAsymp} and then arguments akin to those in the proof of Theorem~\ref{thm:unifbdd11} can be applied. In Regimes 1 and 2, by Corollary~\ref{Cor:ThetaAssymp}, $\theta^2_{n,\Delta}$ and $\theta^3_{n,\Delta,N_1}$ are $O(1)$, uniformly in $(n,N_1)\in \mathcal{N}_1$ and $\Delta \in [\Delta_*,\infty)$; moreover, by Proposition~\ref{Prop:VFrakAsymp}, $\Psi(n;N_1)\geq 0$. Hence, the expression is bounded. In Regime 3, we need to distinguish between two cases. If $(n,N_1)\in \mathcal{N}_\phi$, by Corollary~\ref{Cor:ThetaAssymp}, $\theta^2_{n,\Delta}$ and $\theta^3_{n,\Delta,N_1}$ are $O(N_1)$; however, boundedness of the expression can be ensured since $\frac{\Psi\left(\lfloor \eta N_1 \rfloor; N_1 \right)}{N_1}\propto N_1^2$ for $N_1\gg 1$ by Proposition~\ref{Prop:VFrakAsymp}. Conversely, if $(n,N_1)\in \mathcal{N}_1\setminus \mathcal{N}_{\eta}$,  Corollary~\ref{Cor:ThetaAssymp} shows that $\theta^2_{n,\Delta}$ and $\theta^3_{n,\Delta,N_1}$ are $O(1)$, and hence boundedness follows since $\Psi(n;N_1)\geq 0$. We denote by $\mathcal{C}_1$ the common upper bound.

Then, the second inequality is obtained because, using the triangle inequality,
\begin{equation*}
\begin{array}{lll}
\frac{|\phi_{N_1+1}-\hat{\phi}_n|}{\left|\phi_{N_1+1}-\phi_n\right|}&\leq 1+\frac{\left|\phi_n - \hat{\phi}_n \right|}{\left|\phi_{N_1+1}-\phi_n \right|} = 1+ \frac{\left|\phi_n - \hat{\phi}_n \right|}{\phi_n} \frac{\phi_n}{\left|\phi_{N_1+1}-\phi_n \right|}\overset{\eqref{eq:TauSeparationDelta}}{\leq}1+\frac{1}{2\tau}\frac{\left|\phi_n - \hat{\phi}_n \right|}{\phi_n}
\\
&\overset{\eqref{eq:Basic estimates}}{\leq } 1 + \varepsilon \frac{\mathcal{D}_{\phi}}{2\tau} \frac{ \prod_{s\neq n}\phi_s^2}{\prod_{s\neq n}|\phi_{s}-\phi_{n}|^2 }\leq 1+\varepsilon\mathcal{C}_2
\end{array}
\end{equation*}
for $\mathcal{C}_2 = \frac{\mathcal{D}_{\phi}\mathcal{C}_1}{2\tau}$. Similarly, for  $(s,N_1),(n,N_1)\in \mathcal{N}_1$ with $s\neq n$, we have
\begin{equation*}
\begin{array}{lll}
\frac{|\phi_s - \hat{\phi}_s |}{|\phi_s - \phi_n |} = \frac{|\phi_s - \hat{\phi}_s |}{\phi_s} \frac{\phi_s}{|\phi_s - \phi_n |}\overset{\eqref{eq:TauSeparationDelta}}{\leq} \frac{1}{2\tau}\frac{|\phi_s - \hat{\phi}_s |}{\phi_s}\overset{\eqref{eq:Basic estimates}}{ \leq} \varepsilon \frac{\mathcal{D}_{\phi}}{2\tau}\frac{ \prod_{j\neq s}\phi_j^2}{\prod_{j\neq s}|\phi_{j}-\phi_{s}|^2 }\leq \varepsilon \mathcal{C}_2,
\end{array}
\end{equation*}
which concludes the proof.
\end{proof}
We now further constrain $\varepsilon_*>0$ by assuming that $\varepsilon_* \mathcal{C}_2< \frac{1}{4}$. The value $\frac{1}{4}$ can be replaced by any number in $\left(0,\frac{1}{2}\right)$ and is chosen explicitly just to simplify the subsequent derivations. 

Employing Proposition~\ref{prop:RatioConstituents}, we can bound the numerator and denominator of the fraction in \eqref{eq:firstterm} as follows:
\begin{equation}\label{eq:SwitchEstimates}
\begin{array}{lll}
&\prod_{s\neq n}\left|\phi_{N_1+1}-\hat{\phi}_s \right| \leq \left(1+\varepsilon\mathcal{C}_2 \right)^{N_1-1}\prod_{s\neq n}\left|\phi_{N_1+1}-\phi_s \right| \leq \left(\frac{5}{4} \right)^{N_1-1}\prod_{s\neq n}\left|\phi_{N_1+1}-\phi_s \right|,\vspace{0.1cm}\\
&\frac{\left|\hat{\phi}_s - \hat{\phi}_n \right|}{\left|\phi_s - \phi_n \right|}\geq 1 - \frac{\left|\phi_s - \hat{\phi}_s \right|}{|\phi_s- \phi_n|}- \frac{\left|\phi_n - \hat{\phi}_n \right|}{|\phi_s- \phi_n|}\geq 1-2\varepsilon\mathcal{C}_2 \geq \frac{1}{2},\quad s\neq n\\
&\hspace{55mm} \Longrightarrow \prod_{s\neq n}\left|\hat{\phi}_s - \hat{\phi}_n \right|\geq \frac{1}{2^{N_1-1}}\prod_{s\neq n}\left|\phi_s - \phi_n \right|.
\end{array}
\end{equation}
Therefore,
\begin{equation*}
\begin{array}{lll}
\mathcal{I}_n^{(1)} &\leq \varepsilon \left|y_{N_1+1} \right| \left(\frac{5}{2} \right)^{N_1-1}\prod_{s\neq n}\frac{\left|\phi_{N_1+1}-\phi_s \right|}{\left|\phi_s - \phi_n \right|}= \varepsilon \left|y_{N_1+1} \right| \left(\frac{5}{2} \right)^{N_1-1}\left|L_n\left(\phi_{N_1+1} \right) \right|.
\end{array}
\end{equation*}
We now assess the asymptotic behavior of $\mathcal{I}_n^{(1)}$. Recalling the proof of Lemma~\ref{lem:Lsquared}, and specifically \eqref{eq:LegPol1Useful} and \eqref{eq:LegPol2Useful}, as well as Corollary~\ref{Cor:ThetaAssymp}, we consider the regimes in \eqref{Eq:Regimes} separately, taking into account the fact that $n\in [\lfloor \eta N_1 \rfloor]$ and Assumption~\ref{assump:yassump}, and we resort to arguments partially similar to those in the proof of Theorem~\ref{thm:unifbdd11}.

\begin{itemize}
    \item \underline{Regime 1:} Let $\Delta_*>0$ be fixed and $\Delta\geq \Delta_*$. Then,
    \begin{equation*}
    \begin{array}{lll}
    \mathcal{I}_n^{(1)} &\leq \varepsilon \left|y_{N_1+1} \right| \left(\frac{5}{2} \right)^{N_1-1} e^{-\upsilon \Delta_* \Psi(\lfloor\eta N_1\rfloor; N_1)+O(1)}\\
    &\leq \varepsilon M_y^{(1)}\max_{N_1\in \mathbb{N}}\left( \left(\frac{5}{2} \right)^{N_1-1} e^{-\frac{\upsilon}{2} \Delta_* \Psi(\lfloor\eta N_1\rfloor; N_1)+O(1)} \right)e^{-\frac{\upsilon}{2} \Delta_* \Psi(\lfloor\eta N_1\rfloor; N_1)}.
    \end{array}
    \end{equation*}
    Invoking Proposition~\ref{Prop:VFrakAsymp}, we see that $\frac{\upsilon}{2} \Delta_* \Psi(\lfloor\eta N_1\rfloor; N_1) \propto N_1^3 $ for $N_1 \gg 1$.

    \item \underline{Regime 2:} Let $\Delta_*>0$ and $\vartheta>0$, and choose $N_1^{(1)}(\vartheta)$ such that for all $N_1\geq N_1^{(1)}(\vartheta)$
    \begin{equation*}
    \frac{\upsilon}{2}\Psi\left(\lfloor \eta N_1 \rfloor; N_1 \right) \geq \vartheta,
    \end{equation*}
    where we invoke Proposition~\ref{Prop:VFrakAsymp} to guarantee the existence of $N_1^{(1)}(\vartheta)$. Then, for $N_1\geq N_1^{(1)}(\vartheta)$
    \begin{equation*}
    \begin{array}{lll}
    \mathcal{I}_n^{(1)} &\leq \varepsilon \left|y_{N_1+1} \right| \left(\frac{5}{2} \right)^{N_1-1} e^{-\upsilon \Delta \Psi(\lfloor\eta N_1\rfloor; N_1)+o(1)}\\
    &\leq \varepsilon M_y^{(1)}\max_{N_1\geq N_1^{(1)}(\vartheta)}\left( \left(\frac{5}{2} \right)^{N_1-1} e^{-\frac{\upsilon}{2} \Delta_* \Psi(\lfloor\eta N_1\rfloor; N_1)+o(1)} \right)e^{-\vartheta \Delta}.
    \end{array}
    \end{equation*}

    \item \underline{Regime 3:} Let $N_1\Delta \equiv T>0$. Then, 
    \begin{equation*}
    \begin{array}{lll}
    \mathcal{I}_n^{(1)} &\leq \varepsilon M_y^{(1)} \left(\frac{5}{2} \right)^{N_1-1} e^{-\upsilon T \frac{\Psi(\lfloor\eta N_1\rfloor; N_1) }{N_1}+O(N_1)}\\
    &=\varepsilon M_y^{(1)} e^{-\upsilon T \frac{\Psi(\lfloor\eta N_1\rfloor; N_1) }{N_1}+O(N_1)}.
    \end{array}
    \end{equation*}
    Invoking Proposition~\ref{Prop:VFrakAsymp} again, we have $\upsilon T \frac{\Psi(\lfloor\eta N_1\rfloor; N_1) }{N_1} \propto N_1^2 $  for $N_1 \gg 1$.  
\end{itemize}
Therefore, $\max_{n\in [\lfloor \eta N_1 \rfloor]}\mathcal{I}^{(1)}_n$ exhibits the same asymptotic behavior as $\max_{n\in [\lfloor \eta N_1 \rfloor]}\left|\mathcal{K}_y(n)\right|$, obtained in Theorem~\ref{Thm:FirstOrdCond}, in all regimes \eqref{Eq:Regimes}.

Our next step is to prove analogous estimates on the second term in \eqref{eq:twoterms}. In view of Assumption~\ref{assump:yassump}, $|y_n| \leq \ell$. Then, the triangle inequality yields
\begin{equation*}
\begin{array}{lll}
\left|\left[\left(I_{N_1}-\hat{V}^{-1}V \right)\cdot \operatorname{col}\left\{y_k \right\}_{k=1}^{N_1}\right]_n \right| &\leq \ell \left(\left|1 -\left[\hat{V}^{-1}V \right]_{n,n} \right|+ \sum_{k\neq n}\left|\left[\hat{V}^{-1}V \right] _{n,k}\right| \right)\\
&\leq \ell N_1 \max\left(\left|1 -\left[\hat{V}^{-1}V \right]_{n,n} \right|,\max_{k\neq n}\left|\left[\hat{V}^{-1}V \right] _{n,k}\right|  \right)=:\mathcal{I}_n^2.
\end{array}
\end{equation*}
Given $n\in [\lfloor \eta N_1 \rfloor]$, we upper bound $N_1\left|1 -\left[\hat{V}^{-1}V \right]_{n,n} \right|$ and $N_1\left|\left[\hat{V}^{-1}V \right] _{n,k}\right|$, $(k,N_1)\in \mathcal{N}_{1}$, $k\neq n$, and we show that each of these terms exhibits the desired asymptotic behavior in all regimes \eqref{Eq:Regimes}. 

We begin by noting that, from \eqref{eq:VDMmatDef},
\begin{equation*}
\begin{array}{lll}
V^{-1} \begin{bmatrix}
    \phi_j^0\\
    \phi_j^1\\
    \phi_j^2\\
    \vdots\\
    \phi_j^{N_1-1}    
\end{bmatrix} = 
\begin{bmatrix}
    L_1(\phi_j)\\
    L_2(\phi_j)\\
    \vdots\\
    L_{j-1}(\phi_j)\\
    L_{j}(\phi_j)\\
    L_{j+1}(\phi_j)\\
    \vdots\\
    L_{N_1-1}(\phi_j)    
\end{bmatrix} = \begin{bmatrix}
    0\\
    0\\
    \vdots\\
    0\\
    1\\
    0\\
    \vdots\\
    0    
\end{bmatrix} = \mathrm{e}_{j}
\Longrightarrow V^{-1} \underbrace{\begin{bmatrix}
    \phi_1^0 & \phi_2^0 & \phi_3^0 & \dots & \phi_{N_1}^0\\
    \phi_1^1 & \phi_2^1 & \phi_3^1 & \dots & \phi_{N_1}^1\\
    \phi_1^2 & \phi_2^2 & \phi_3^2 & \dots & \phi_{N_1}^2\\
    \vdots & \vdots & \vdots & \ddots & \vdots\\
    \phi_1^{N_1-1} & \phi_2^{N_1-1} & \phi_3^{N_1-1} & \dots & \phi_{N_1}^{N_1-1}
\end{bmatrix}}_{V} = I_{N_1}.
\end{array} \end{equation*}

Then, $\left[\hat{V}^{-1}V \right]_{n,k} = \sum_{j=0}^{N_1-1} \hat l_{n,j} \phi_k^j = \hat L_n(\phi_k) = \prod_{s \neq n} \frac{\phi_k-\hat \phi_s}{\hat \phi_n - \hat \phi_s}$.

Consider first the off-diagonal element with $(k,N_1)\in \mathcal{N}_1$, $k\neq n$. Recall that we consider $n\in [\lfloor \eta N_1 \rfloor]$, while $k \in [N_1]$. Employing arguments similar to \eqref{eq:SwitchEstimates}, we have that, for some $\mathcal{C}_3>1$,
\begin{equation*}
\begin{array}{lll}
N_1\left| \left[\hat{V}^{-1}V \right]_{n,k} \right| &= N_1\prod_{s\neq n}\frac{\left| \phi_k-\hat{\phi}_s\right|}{\left|\hat{\phi}_n - \hat{\phi}_s \right|} \leq N_1\mathcal{C}_3^{N_1}\left|\phi_k - \hat{\phi}_k \right| \frac{\prod_{s\neq n,k}\left|\phi_k - \phi_s \right|}{\prod_{s\neq n}\left|\phi_s - \phi_n \right|}\\
&\overset{\eqref{eq:Basic estimates}}{\leq }\varepsilon \mathcal{D}_{\phi}N_1\mathcal{C}_3^{N_1}\frac{\phi_k \prod_{s\neq k}\phi_s^2}{\prod_{s\neq k}\left|\phi_s - \phi_k \right|^2} \frac{\prod_{s\neq n,k}\left|\phi_k - \phi_s \right|}{\prod_{s\neq n}\left|\phi_s - \phi_n \right|}\\
& = \varepsilon \mathcal{D}_{\phi}N_1\mathcal{C}_3^{N_1} \frac{\phi_k}{\left|\phi_k -\phi_n \right|}\frac{ \prod_{s\neq k}\phi_s^2}{\prod_{s\neq k}\left|\phi_s - \phi_k \right|^2} \frac{\prod_{s\neq k}\left|\phi_k - \phi_s \right|}{\prod_{s\neq n}\left|\phi_s - \phi_n \right|}\\
&\overset{\eqref{eq:TauSeparationDelta}}{\leq}\varepsilon \frac{\mathcal{D}_{\phi}}{2\tau}N_1\mathcal{C}_3^{N_1} \frac{ \prod_{s\neq k}\phi_s}{\prod_{s\neq k}\left|\phi_s - \phi_k \right|} \frac{\prod_{s\neq k}\phi_s}{\prod_{s\neq n}\left|\phi_s - \phi_n \right|}\\
&=\varepsilon \frac{\mathcal{D}_{\phi}}{2\tau}N_1\mathcal{C}_3^{N_1}  \frac{\phi_n}{\phi_k} \frac{ \prod_{s\neq k}\phi_s}{\prod_{s\neq k}\left|\phi_s - \phi_k \right|} \frac{\prod_{s\neq n}\phi_s}{\prod_{s\neq n}\left|\phi_s - \phi_n \right|}\\
& \overset{\text{Prop.~\ref{prop:RatioConstituents}}}{\leq} \varepsilon \frac{\mathcal{D}_{\phi} \sqrt{\mathcal{C}_1}}{2\tau}N_1\mathcal{C}_3^{N_1}  \frac{\phi_n}{\phi_k} \frac{\prod_{s\neq n}\phi_s}{\prod_{s\neq n}\left|\phi_s - \phi_n \right|}.
\end{array}
\end{equation*}
Note that, for $n\in [\lfloor \eta N_1 \rfloor]$ and $j\in [N_1]$,
$\frac{\phi_n}{\phi_j} = e^{-\Delta\left(\lambda_j-\lambda_n \right)}\leq e^{\Delta \left( \lambda_{\lfloor \eta N_1 \rfloor }- \lambda_1\right)}\overset{\eqref{eq:EigDiff}}{\leq}e^{\Delta \Upsilon \left(\lfloor \eta N_1 \rfloor^2-1 \right)}$.

The expression $\frac{\prod_{s\neq n}\phi_s}{\prod_{s\neq n}\left|\phi_s - \phi_n \right|}$ can be upper bounded as in \eqref{eq:fracAsymp} and then arguments akin to those in the proof of Theorem~\ref{thm:unifbdd11} and already employed earlier on in this proof (which leverage Corollary~\ref{Cor:ThetaAssymp}, Proposition~\ref{Prop:VFrakAsymp} and Lemma~\ref{lem:Lsquared}), can be applied to address each of the regimes in \eqref{Eq:Regimes}, taking into account that $n\in [\lfloor \eta N_1 \rfloor]$, which allows us to obtain the desired decay rates (compared to the case of $k$, which can be arbitrary in $[N_1]$, whence the corresponding term $\frac{\prod_{s\neq k}\phi_s}{\prod_{s\neq k}\left|\phi_s - \phi_k \right|}$ can only be upper bounded by a constant).
\begin{itemize}
    \item \underline{Regime 1:} Let $\Delta_*>0$ be fixed and $\Delta\geq \Delta_*$. Then,
\begin{equation*}
\begin{array}{lll}
N_1\left| \left[\hat{V}^{-1}V \right]_{n,k} \right| &\leq \varepsilon \frac{\mathcal{D}_{\phi}\sqrt{\mathcal{C}_1}}{2\tau}N_1\mathcal{C}_3^{N_1}  e^{\Delta \Upsilon \left(\lfloor \eta N_1 \rfloor^2-1 \right)} e^{-\upsilon\Delta\Psi\left(\lfloor \eta N_1 \rfloor; N_1  \right)+O(1)}.
\end{array}
\end{equation*}
For $N_1\gg 1$, in view of Proposition~\ref{Prop:VFrakAsymp} it is $\Psi(\lfloor\eta N_1\rfloor; N_1) \propto N_1^3 $, and hence we have 
\begin{equation*}
\begin{array}{lll}
\Delta_*\left(\frac{\upsilon}{2} \Psi\left(\lfloor \eta N_1 \rfloor; N_1  \right)-\Upsilon \left(\lfloor \eta N_1 \rfloor^2-1 \right)\right)-N_1\ln\left(\mathcal{C}_3 \right)-\ln(N_1)-O(1) >0,
\end{array}
\end{equation*}
whence
\begin{equation*}
\begin{array}{lll}
N_1 \left| \left[\hat{V}^{-1}V \right]_{n,k} \right| &\leq \varepsilon \frac{\mathcal{D}_{\phi}\sqrt{\mathcal{C}_1}}{2\tau}e^{-\frac{\upsilon}{2}\Delta_*\Psi\left(\lfloor \eta N_1 \rfloor; N_1  \right)},\quad N_1\gg 1.
\end{array}
\end{equation*}
Again, Proposition~\ref{Prop:VFrakAsymp} ensures that $\frac{\upsilon}{2} \Delta_* \Psi(\lfloor\eta N_1\rfloor; N_1) \propto N_1^3 $ for $N_1 \gg 1$.

\item \underline{Regime 2:} Let $\vartheta>0$ and $\Delta_*>0$ be fixed. Then, 
\begin{equation*}
\begin{array}{lll}
N_1 \left| \left[\hat{V}^{-1}V \right]_{n,k} \right| &\leq \varepsilon \frac{\mathcal{D}_{\phi}\sqrt{\mathcal{C}_1}}{2\tau}N_1\mathcal{C}_3^{N_1}  e^{\Delta \Upsilon \left(\lfloor \eta N_1 \rfloor^2-1 \right)} e^{-\upsilon\Delta\Psi\left(\lfloor \eta N_1 \rfloor; N_1  \right)+o(1)}.
\end{array}
\end{equation*}
Choosing $N_1^{(2)}(\vartheta)\in \mathbb{N}$ such that, for all $N_1 \geq N_1^{(2)}(\vartheta)$,
\begin{equation*}
        \Delta_*\left(\frac{\upsilon}{2} \Psi\left(\lfloor \eta N_1 \rfloor; N_1  \right)-\Upsilon \left(\lfloor \eta N_1 \rfloor^2-1 \right)\right)-N_1\ln\left(\mathcal{C}_3 \right)-\ln(N_1) >0,\quad \frac{\upsilon}{2} \Psi\left(\lfloor \eta N_1 \rfloor; N_1  \right)\geq \vartheta,
\end{equation*}
we have that, for all $N_1\geq N_1^{(2)}(\vartheta)$,
\begin{equation*}
\begin{array}{lll}
N_1 \left| \left[\hat{V}^{-1}V \right]_{n,k} \right| &\leq \varepsilon \frac{\mathcal{D}_{\phi}\sqrt{\mathcal{C}_1}}{2\tau}e^{-\vartheta \Delta +o(1)},\quad \Delta \gg 0.
\end{array}
\end{equation*}

\item \underline{Regime 3:} Considering $N_1\Delta = T>0$, we have
\begin{equation*}
\begin{array}{lll}
N_1 \left| \left[\hat{V}^{-1}V \right]_{n,k} \right| &\leq \varepsilon \frac{\mathcal{D}_{\phi}\sqrt{\mathcal{C}_1}}{2\tau}N_1\mathcal{C}_3^{N_1}  e^{T \Upsilon \frac{\lfloor \eta N_1 \rfloor^2-1 }{N_1}} e^{-\upsilon T \frac{\Psi\left(\lfloor \eta N_1 \rfloor; N_1  \right)}{N_1}+O(N_1)}.
\end{array}
\end{equation*}
For $N_1\gg 1$, again in view of Proposition~\ref{Prop:VFrakAsymp},
\begin{equation*}
        \frac{\upsilon T}{2N_1} \Psi\left(\lfloor \eta N_1 \rfloor; N_1  \right)-\frac{\Upsilon T}{N_1} \left(\lfloor \eta N_1 \rfloor^2-1 \right)-N_1\ln\left(\mathcal{C}_3 \right)-\ln(N_1)-O(N_1) >0,
\end{equation*} 
whence
\begin{equation*}
\begin{array}{lll}
\left| \left[\hat{V}^{-1}V \right]_{n,k} \right| &\leq \varepsilon \frac{\mathcal{D}_{\phi}\sqrt{\mathcal{C}_1}}{2\tau}e^{-\frac{\upsilon T}{2N_1}\Psi\left(\lfloor \eta N_1 \rfloor; N_1  \right)}, \quad N_1\gg 1
\end{array}
\end{equation*}
where Proposition~\ref{Prop:VFrakAsymp} guarantees that $\frac{\upsilon T}{2N_1}  \Psi(\lfloor\eta N_1\rfloor; N_1) \propto N_1^2 $ for $N_1\gg 1$.
\end{itemize}
Therefore, for any $n\in [\lfloor \eta N_1 \rfloor]$, in all the regimes in \eqref{Eq:Regimes}, the term $N_1 \max_{k\neq n}\left|\left[ \hat{V}^{-1}V \right]_{n,k} \right|$ exhibits the same asymptotic behavior as $\max_{n\in [\lfloor \eta N_1 \rfloor]}\left|\mathcal{K}_y(n)\right|$, obtained in Theorem~\ref{Thm:FirstOrdCond}.

It remains to consider the diagonal term  $N_1\left|1 -\left[\hat{V}^{-1}V \right]_{n,n} \right|$ with $n\in [\lfloor \eta N_1 \rfloor]$.  Recall that 
\begin{equation*}
\begin{array}{lll}
\left[\hat{V}^{-1} V\right]_{n, n}=\hat{L}_n\left(\phi_n\right)=\prod_{s \neq n} \frac{\phi_n-\hat{\phi}_s}{\hat{\phi}_n-\hat{\phi}_s}.
\end{array}
\end{equation*}
To simplify the derivations, we introduce the following notations 
\begin{equation*}
\varkappa_n = \phi_n - \hat{\phi}_n,\quad \Lambda_{n,N_1} = \left\{\frac{1}{\hat{\phi}_n - \hat{\phi}_s} \right\}_{s\in [N_1]\setminus \left\{n \right\}},\quad \left|\Lambda_{n,N_1} \right| = \left\{|q|\ ; \ q\in \Lambda_{n,N_1} \right\},
\end{equation*}
where $\varkappa_n$, $n\in [N_1]$, satisfies Theorem~\ref{thm:unifbdd11}. Then,
\begin{equation*}
\begin{array}{lll}
\left[\hat{V}^{-1} V\right]_{n, n}&= \prod_{s \neq n} \frac{\phi_n-\hat{\phi}_s}{\hat{\phi}_n-\hat{\phi}_s} = \prod_{s\neq n}\frac{\varkappa_n+ \hat{\phi}_n-\hat{\phi}_s}{\hat{\phi}_n - \hat{\phi}_s} = \prod_{s\neq n}\left(1+\frac{\varkappa_n}{\hat{\phi}_n - \hat{\phi}_s} \right)= 1+ \sum_{r=1}^{N_1-1}\varkappa_n^r\mathfrak{S}_r\left(\Lambda_{n,N_1} \right),
\end{array}
\end{equation*}
where the elementary symmetric polynomial $\mathfrak{S}_r$ is defined in \eqref{eq:SymmPolDef}. Hence,
\begin{equation*}
\begin{array}{lll}
N_1 \left|1-\left[\hat{V}^{-1} V\right]_{n, n} \right|\leq N_1 \sum_{r=1}^{N_1-1}\left|\varkappa_n \right|^r \mathfrak{S}_r\left(\left|\Lambda_{n,N_1} \right| \right).
\end{array}
\end{equation*}
Applying the MacLaurin inequality of Lemma~\ref{Lem:MacIneq}, we get
\begin{equation*}
\begin{array}{lll}
\sum_{r=1}^{N_1-1}\left|\varkappa_n \right|^r \mathfrak{S}_r\left(\left|\Lambda_{n,N_1} \right| \right) &= \sum_{r=1}^{N_1-1}\binom{N_1-1}{r}\left|\varkappa_n \right|^r\mathcal{M}_r\left( \left|\Lambda_{n,N_1} \right|\right)\leq \sum_{r=1}^{N_1-1}\binom{N_1-1}{r}\left|\varkappa_n \right|^r \mathcal{M}_1\left(\left|\Lambda_{n,N_1} \right| \right)^r\\
&=\left(1+\left|\varkappa_n \right|\mathcal{M}_1\left(\left|\Lambda_{n,N_1} \right| \right)  \right)^{N_1-1} -1,
\end{array}
\end{equation*}
where the last equality follows form the binomial formula.

Since $\mathcal{M}_1\left(\left|\Lambda_{n,N_1} \right| \right) = \frac{1}{N_1} \sum_{m\neq n}   \frac{1}{\left| \hat{\phi}_n - \hat{\phi}_m\right|}$, we can invoke Theorem~\ref{thm:unifbdd11} and bound $ \left|\varkappa_n \right|\mathcal{M}_1\left(\left|\Lambda_{n,N_1} \right| \right)$ as follows:

\begin{equation*}
\begin{array}{lll}
\left|\varkappa_n \right|\mathcal{M}_1\left(\left|\Lambda_{n,N_1} \right| \right) &\leq  \varepsilon \mathcal{D}_{\phi} \frac{1}{N_1} \sum_{m\neq n}   \frac{\phi_n \left( \prod_{s\neq n}\phi_s^2\right)}{\left| \hat{\phi}_n - \hat{\phi}_m\right|\prod_{s\neq n}|\phi_{s}-\phi_{n}|^2 }\leq 2\varepsilon \mathcal{D}_{\phi} \frac{1}{N_1} \sum_{m\neq n}   \frac{\phi_n \left( \prod_{s\neq n}\phi_s^2\right)}{\left| \phi_n - \phi_m\right|\prod_{s\neq n}|\phi_{s}-\phi_{n}|^2 }\\
&\overset{}{\leq }\varepsilon \frac{\mathcal{D}_{\phi}}{\tau} \frac{N_1-1}{N_1}\frac{ \prod_{s\neq n}\phi_s^2}{\prod_{s\neq n}|\phi_{s}-\phi_{n}|^2 }\leq \varepsilon \frac{\mathcal{D}_{\phi}}{\tau} \frac{ \prod_{s\neq n}\phi_s^2}{\prod_{s\neq n}|\phi_{s}-\phi_{n}|^2 }\\
&\overset{\eqref{eq:fracAsymp}}{\leq} \varepsilon \frac{\mathcal{D}_{\phi}}{\tau}e^{-2\upsilon \Delta \Psi\left( \lfloor \eta N_1 \rfloor; N_1 \right)}\cdot\begin{cases}
    e^{2\theta^3_{1,\Delta,N_1}}, & n=1\\
    e^{2\theta^2_{n,\Delta}+2\theta^3_{n,\Delta,N_1}}, & n\neq 1
\end{cases},
\end{array}
\end{equation*}
where the second inequality relies on \eqref{eq:SwitchEstimates}, which guarantees that $\left| \hat{\phi}_n - \hat{\phi}_m\right| \geq \frac{1}{2} \left| \phi_n - \phi_m\right|$, while the third inequality follows from \eqref{eq:TauSeparationDelta}, which guarantees that $\frac{|\phi_n-\phi_m|}{\phi_n} \geq 2 \tau$, and from the fact that the summation can be then replaced by $N_1-1$ identical terms. The fourth inequality follows from $\frac{N_1-1}{N_1} < 1$. Finally, the last step uses the fact that $n\in [\lfloor \eta N_1 \rfloor ]$.

We obtain the upper bound
\begin{equation}\label{eq:rho}
\begin{array}{lll}
N_1 \left|1-\left[\hat{V}^{-1} V\right]_{n, n} \right| \leq N_1 \left[\left(1+\varrho_{n,\Delta,N_1} \right)^{N_1-1}-1\right],
\end{array}
\end{equation}
where
\begin{equation*}
\begin{array}{lll}
\varrho_{n,\Delta,N_1} := 
\begin{cases}
\varepsilon \frac{\mathcal{D}_{\phi}}{\tau}e^{-2\upsilon \Delta \Psi\left( \lfloor \eta N_1 \rfloor; N_1 \right)+2\theta^3_{n,\Delta,N_1}}, & n=1,\\
\varepsilon \frac{\mathcal{D}_{\phi}}{\tau}e^{-2\upsilon \Delta \Psi\left( \lfloor \eta N_1 \rfloor; N_1 \right)+2\theta^2_{n,\Delta}+2\theta^3_{n,\Delta,N_1}}, & n \neq 1.
\end{cases}
\end{array}
\end{equation*}
In both cases, and in all the regimes in \eqref{Eq:Regimes}, the expression $\left(1+\varrho_{n,\Delta,N_1} \right)^{N_1-1}$ is bounded uniformly in its parameters, as a consequence of Corollary~\ref{Cor:ThetaAssymp} and of Proposition~\ref{Prop:VFrakAsymp}. Then, 
\begin{equation}\label{eq:DiagTermBound}
\begin{array}{lll}
N_1 \left|1-\left[\hat{V}^{-1} V\right]_{n, n} \right| &\leq N_1 \left[\left(1+\varrho_{n,\Delta,N_1} \right)^{N_1-1}-1\right]= N_1 \left[e^{(N_1-1)\ln\left(1+\varrho_{n,\Delta,N_1} \right)}-1\right]\\
&\leq \left(1+\varrho_{n,\Delta,N_1} \right)^{N_1-1}N_1(N_1-1)\ln\left(1+\varrho_{n,\Delta,N_1} \right) \\
&\leq \left(1+\varrho_{n,\Delta,N_1} \right)^{N_1-1}N_1(N_1-1)\varrho_{n,\Delta,N_1},
\end{array}
\end{equation}
where the second inequality follows from applying Lagrange's theorem to function $e^x$ on the interval from $0$ to $(N_1-1)\ln\left(1+\varrho_{n,\Delta,N_1}\right)$, with the derivative of the function computed at the maximum of the interval in view of the function's monotonicity, while the second inequality follows from the fact that $\ln(1+x)\leq x$ for all $x\geq 0$.
Since $\left(1+\varrho_{n,\Delta,N_1} \right)^{N_1-1}$ is bounded in all regimes \eqref{Eq:Regimes}, the analysis of the asymptotic behavior of \eqref{eq:DiagTermBound} boils down to that of $\varrho_{n,\Delta,N_1}$ in \eqref{eq:rho}, which follows arguments similar to those used for $N_1 \max_{k \neq n}\left| \left[\hat{V}^{-1}V \right]_{n,k} \right|$, $k\neq n$, to show that also the term $N_1\max_{n\in [\lfloor \eta N_1 \rfloor]} \left|1-\left[\hat{V}^{-1} V\right]_{n, n} \right|$ exhibits the same asymptotic behavior as $\max_{n\in [\lfloor \eta N_1 \rfloor]}\left|\mathcal{K}_y(n)\right|$, obtained in Theorem~\ref{Thm:FirstOrdCond}, in all the regimes in \eqref{Eq:Regimes}.
Therefore, the same asymptotic behavior is exhibited by $\max_{n\in [\lfloor \eta N_1 \rfloor]}\mathcal{I}_n^{(2)}$.

The result of Theorem~\ref{thm:PronyAmplitudeCondition} then follows by combining the asymptotic behavior of $\max_{n\in [\lfloor \eta N_1 \rfloor ]}\mathcal{I}_n^{(1)}$ and of $\max_{n\in [\lfloor \eta N_1 \rfloor ]}\mathcal{I}_n^{(2)}$, both of which are the same as $\max_{n\in [\lfloor \eta N_1 \rfloor]}\left|\mathcal{K}_y(n)\right|$, obtained in Theorem~\ref{Thm:FirstOrdCond}, in all the regimes in \eqref{Eq:Regimes}.

\end{document}